\documentclass[11pt]{article}
\usepackage{amsthm,amsfonts,amssymb,amsmath}
\usepackage{stmaryrd}
\usepackage{cite}
\usepackage[pdftex,colorlinks,backref=page,citecolor=blue,bookmarks=false]{hyperref}
\usepackage{cleveref}
\usepackage{epsfig}
\usepackage{url}
\usepackage{xcolor}
\usepackage{multicol,graphicx}
\usepackage{fullpage}
\usepackage{bbm}
\usepackage{thm-restate}
\usepackage{tocloft}
\usepackage[square,sort,comma,numbers]{natbib} 
\usepackage{setspace}
\usepackage[shortlabels]{enumitem}

\numberwithin{equation}{section}

\newtheorem{thm}[equation]{Theorem}
\newtheorem{cor}[equation]{Corollary}
\newtheorem{lem}[equation]{Lemma}
\newtheorem{prop}[equation]{Proposition}
\newtheorem{claim}[equation]{Claim}
\newtheorem{obs}[equation]{Observation}

\newtheorem{conj}[equation]{Conjecture}

\crefname{thm}{Theorem}{Theorems}
\crefname{cor}{Corollary}{Corollaries}
\crefname{lem}{Lemma}{Lemmas}
\crefname{prop}{Proposition}{Propositions}
\crefname{claim}{Claim}{Claims}
\crefname{obs}{Observation}{Observations}
\crefname{fact}{Fact}{Facts}
\crefname{conj}{Conjecture}{Conjectures}
\crefname{ques}{Question}{Questions}
\crefname{prob}{Problem}{Problems}

\theoremstyle{definition}
\newtheorem{defn}[equation]{Definition}

\newtheorem{ex}[equation]{Example}

\theoremstyle{remark}

\newtheorem{rem}[equation]{Remark}

\newcommand{\eps}{\varepsilon}
\newcommand{\cP}{\mathcal{P}}
\newcommand{\cG}{\mathcal{G}}
\newcommand{\cS}{\mathcal{S}}
\newcommand{\cI}{\mathcal{I}}

\newcommand{\cE}{\mathcal{E}}
\newcommand{\cC}{\mathcal{C}}
\newcommand{\cF}{\mathcal{F}}
\newcommand{\cX}{\mathcal{X}}

\newcommand{\cA}{\mathcal{A}}
\newcommand{\cB}{\mathcal{B}}
\newcommand{\cU}{\mathcal{U}}
\newcommand{\cQ}{\mathcal{Q}}

\newcommand{\bP}{\mathbb{P}}
\newcommand{\bE}{\mathbb{E}}
\newcommand{\bR}{\mathbb{R}}
\newcommand{\bfx}{\mathbf{x}}
\newcommand{\bfy}{\mathbf{y}}
\newcommand{\bfw}{\mathbf{w}}

\newcommand{\nonexE}{\cE^{\mathrm{non-ext}}}
\newcommand{\exE}{\cE^{\mathrm{ext}}}

\newcommand{\ds}{d^{1}}

\def \Gp {G^+}
\def \Fp {F^+}
\def \Fm {F^-}

\def \Dp {D^+}

\def \Fs {F^*}
\def \Fss {F^{**}}

\newcommand{\comp}[1]{#1^{\mathsf{c}}}

\newcommand{\floor}[1]{\left\lfloor #1 \right\rfloor}
\newcommand{\ceil}[1]{\left\lceil #1 \right\rceil}

\title{Exact supported co-degree bounds for Hamilton cycles}
\author{Shoham Letzter\footnotemark[1]
\qquad
Arjun Ranganathan\footnotemark[1]
}

\date{}

\begin{document}

\maketitle

\begin{abstract}
	\setlength{\parskip}{\medskipamount}
    \setlength{\parindent}{0pt}
    \noindent

	For any $k\ge 3$ and $\ell \in [k-1]$ such that $(k,\ell) \ne (3,1)$, we show that any sufficiently large $k$-graph $G$ must contain a Hamilton $\ell$-cycle provided that it has no isolated vertices and every set of $k-1$ vertices contained in an edge is contained in at least $\left(1 - \frac{1}{\lfloor{\frac{k}{k-\ell}\rfloor}(k-\ell)}\right)n - (k - 3)$ edges. We also show that this bound is tight for infinitely many values of $k$ and $\ell$ and is off by at most $1$ for all others, and is hence essentially optimal. This improves an asymptotic version of this result due to Mycroft and Z{\'a}rate-Guer{\'e}n~\cite{mycroft2025positive}, and the case $\ell = k-1$ completely resolves a conjecture of Illingworth, Lang, M\"uyesser, Parczyk and Sgueglia~\cite{illingworth2025spanning}. 

	These results support the utility of \emph{minimum supported co-degree} conditions in a $k$-graph, a recently introduced variant of the standard notion of minimum co-degree applicable to $k$-graphs with non-trivial strong independent sets. Our proof techniques involve a novel blow-up tiling framework introduced by Lang~\cite{lang2023tiling}, avoiding traditional approaches using the regularity and blow-up lemmas.
\end{abstract}

\renewcommand{\thefootnote}{\fnsymbol{footnote}} % Make affiliation marks symbols

\footnotetext[1]{Department of Mathematics, University College London, Gower Street, London WC1E~6BT, UK. Emails: 
  \textsf{\{\href{mailto:s.letzter@ucl.ac.uk}{s.letzter},
           \href{mailto:arjun.ranganathan.24@ucl.ac.uk}{arjun.ranganathan.24}\}@ucl.ac.uk}. Research of SL supported by the Royal Society.}

\newpage

\tableofcontents

\newpage

\section{Introduction}\label{sec:intro}
    A widespread research theme in extremal graph theory is to determine sufficient conditions that ensure the existence of specified spanning structures in graphs and hypergraphs. A classic result in this vein, due to Dirac~\cite{dirac1952some}, states that any graph on $n\geq 3$ vertices with minimum degree at least $n/2$ contains a Hamilton cycle, and it is not too difficult to see that the degree condition is the best possible. Dirac's theorem has been generalised to various distinct settings over the years (see the surveys \cite{gould2014recent,kuhn2014hamilton,simonovits2020embedding,kuhn2012survey,frieze2019hamilton,rodl2010dirac}) and this has contributed to the development of several powerful techniques, such as regularity, absorption, and rotation-extension methods. In this paper, we will focus on hypergraph extensions of Dirac's theorem.

	A \emph{$k$-uniform hypergraph} or simply a \emph{$k$-graph} $G$ consists of a set of vertices $V(G)$ and a set of edges $E(G)$, where each edge consists of exactly $k$ vertices. The most common generalisation of a minimum degree condition from graphs to hypergraphs is accomplished via the notion of a \emph{minimum co-degree} -- which is the minimum $d$ such that every set of $k-1$ vertices in $G$ is contained in at least $d$ edges -- and we denote it here by $\delta(G)$. The larger uniformity $k$ allows for multiple distinct ways to define a cycle in a $k$-graph. The cyclic structures we search for are \emph{$\ell$-cycles}, which are perhaps the most common extensions of graph cycles. Intuitively, we think of an $\ell$-cycle as a spanning path formed by a cyclic set of edges such that each edge has exactly $\ell$ vertices in common with the preceding edge (it could intersect other edges as well). More formally, given an $n$-vertex $k$-graph $G$ and any $\ell \in [k-1]$ such that $k - \ell$ divides $n$, we define a \emph{Hamilton $\ell$-cycle} in $G$ to be an ordering of the vertices of $G$, say $v_1 \dots v_n$, such that the vertices of the subsequence $v_{i(k - \ell)+1} \dots v_{i(k - \ell) +k}$ form an edge for all $i \ge 0$, where we view the indices cyclically modulo $n$. Since each edge of an $\ell$-cycle contains $k-\ell$ vertices that were not in the previous edge, we trivially require that $k - \ell$ divides $n$ so that when we ``cycle around'' the vertices of $G$, the original sequence of edges gets repeated. We call $(k-1)$-cycles and $1$-cycles \emph{tight} and \emph{loose} cycles respectively.
    
The question of determining sufficient minimum co-degree conditions to ensure a tight Hamilton cycle was first raised by Katona and Kierstead~\cite{katona1999hamiltonian} who conjectured that if $\delta(G)\ge (n-k+2)/2$ then $G$ has a tight Hamilton cycle, and they showed that this bound is the best possible. R\"{o}dl, Ruci\'nski and Szemer\'edi~\cite{rodl2006dirac,rodl2008approximate} first proved this asymptotically for $k\geq 3$ and $\delta(G)\geq n/2+o(n)$, and later obtained an exact result for $k=3$~\cite{rodl2011dirac}. Additionally, Liu and Liu \cite{liu2022hamiltonian} made progress towards an exact result for $k = 4$. In fact, for any $\ell \in [k-1]$ such that $k - \ell$ divides $k$, a tight Hamilton cycle will contain such a Hamilton $\ell$-cycle (provided the necessary divisibility condition holds). If we further suppose that $k$ divides $n$, then this Hamilton $\ell$-cycle will imply the existence of a perfect matching which necessitates $\delta(G) \ge n/2 - k$ as shown in \cite{kuhn2010hamilton,rodl2009perfect}, and hence the aforementioned tight cycle bound of $\delta(G) \ge n/2 + o(n)$ is tight. If $k$ does not divide $n$, then R\"{o}dl, Ruci\'nski and Szemer\'edi~\cite{rodl2009perfect} have determined the minimum codegree threshold for a near-perfect matching of size $\floor{n/k}$ to be roughly $n/k$. It is therefore plausible that the correct minimum codegree threshold for a spanning $\ell$-cycle, when $k-\ell$ divides $k$ but is not $1$ and $n$ is not divisible by $k$, is significantly lower than $n/2$. Conversely, if $k-\ell$ does not divide $k$, a series of works by K\"uhn and Osthus~\cite{kuhn2006loose}, Keevash, K\"uhn, Mycroft and Osthus~\cite{keevash2011loose}, H\`an and Schacht~\cite{han2010dirac} and K\"uhn, Mycroft and Osthus~\cite{kuhn2010hamilton} have established that 
\[
\delta(G)\geq \frac{n}{\ceil{\frac{k}{k-\ell}}(k-\ell)}+o(n)
\]
suffices and is optimal up to the $o(n)$ term, yielding a much lower threshold. Exact bounds, however, have proven notoriously difficult to obtain and are known only in a few special cases. To the best of our knowledge, these are $(k,\ell)=(3,2)$ by R\"{o}dl, Ruci\'nski and Szemer\'edi~\cite{rodl2011dirac}, $(k,\ell)=(3,1)$ by Czygrinow and Molla~\cite{czygrinow2014tight}, $(k,\ell)=(4,2)$ by Garbe and Mycroft~\cite{garbe2018hamilton}, and $k\ge 3$ and $\ell<k/2$ by Han and Zhao~\cite{han2015minimum}. We refer the reader to the surveys~\cite{kuhn2014hamilton,zhao2016recent,rodl2010dirac} for more detailed discussions.

A drawback of a minimum co-degree condition is that it tends to be a fairly strong requirement. For instance, if we start with a complete $k$-graph and remove all edges that contain a fixed pair of vertices, it is immediate that the resulting hypergraph has minimum co-degree equal to zero. However, if $k\ge 3$, it is not hard to see that even fairly small hypergraphs of this form are quite dense, and trivially contain Hamilton cycles. In fact, there are several natural hypergraph classes that are quite dense, but have minimum co-degree equal to zero (multipartite hypergraphs, for instance), and hence the previous theorems are not immediately applicable. This motivates the notion of the \emph{minimum supported co-degree} of a $k$-graph $G$, which we denote $\delta^*(G)$, and define as the maximum integer $d$ such that every $(k-1)$-set that is contained in at least one edge is contained in at least $d$ edges. A minimum supported co-degree requirement immediately handles the previous issues of zero co-degree sets, although it does not rule out isolated vertices, and hence we will usually impose the (fairly weak) requirement of all vertices being non-isolated.

Multiple recent studies have investigated questions that arise from replacing a minimum co-degree with a minimum supported co-degree requirement in classical extremal hypergraph theoretic problems. The notion of supported co-degrees was introduced by Balogh, Lemons and Palmer~\cite{balogh2021maximum} (they use the term ``positive co-degree'' instead) to formulate a generalisation of the Erd\H{o}s-Ko-Rado theorem, providing bounds on the sizes of intersecting families satisfying a minimum supported co-degree condition. Subsequently, Frankl and Wang~\cite{frankl2025intersecting} improved these bounds in almost all cases, and Spiro~\cite{spiro2023t} has extended these studies to the more general case of $t$-intersecting families. Other papers have worked on generalisations of the Andr\'asfai-Erd\H{o}s-S\'os theorem~\cite{liu2024positive} and unique colourability of hypergraphs~\cite{liu2024uniquely}.

Halfpap, Lemons and Palmer~\cite{halfpap2025positive} investigated variants of hypergraph Tur\'an problems, and established the asymptotic supported co-degree threshold for a host $3$-graph to contain a copy of a fixed graph $F$ for several distinct $3$-graphs $F$, and Wu~\cite{wu2025positive} has consequently addressed this question for other examples. Halfpap, Lemons and Palmer~\cite{halfpap2025positive} also study ``jumps'' in this threshold, which has further been tackled by Balogh, Halfpap, Lidick\'y, and Palmer~\cite{balogh2024positive}. Pikhurko~\cite{pikhurko2023limit} recently proved that these problems are well-defined for a generalisation to the ``minimum positive/supported $\ell$-co-degree'' for any $\ell \in [k-1]$, providing analogous results to those Lo and Markstr\"om~\cite{lo2014degree} regarding the usual notion of $\ell$-co-degree.

The question of finding minimum supported co-degree conditions that ensure the existence of spanning structures in hypergraphs was first raised by Halfpap and Magnan~\cite{halfpap2024positive}. They prove an exact optimal lower bound on the minimum supported co-degree required to guarantee a perfect matching in $3$-graphs and a slightly weaker bound for all higher uniformities, which was later improved to an exact tight bound by Mycroft and Z\'arate-Guer\'en~\cite{mycroft2025matching}. Furthermore, Halfpap and Magnan also establish an exact best-possible minimum supported co-degree condition for Hamilton Berge cycles and an asymptotic one for $3$-uniform loose cycles. In a different direction, Illingworth, Lang, M\"uyesser, Parczyk and Sgueglia~\cite{illingworth2025spanning} showed that an $n$-vertex $k$-graph with $\delta^*(G)\ge n/2 + o(n)$ contains a spanning $k$-sphere, asymptotically confirming a conjecture of Georgakopoulos, Haslegrave, Montgomery and Narayanan~\cite{georgakopoulos2022spanning}.

The problem of determining the optimal minimum supported co-degree for $k$-uniform $\ell$-cycles was initially suggested by Halfpap and Magnan in~\cite[Section 5]{halfpap2024positive}, and reiterated by Illingworth, Lang, M\"uyesser, Parczyk and Sgueglia~\cite{illingworth2025spanning} for the special case of tight cycles. This question was recently tackled by Mycroft and Z\'arate-Guer\'en~\cite{mycroft2025positive}, who prove an asymptotically optimal result for all $k\ge 3$ and $\ell\in [k-1]$. 

Our main result provides an improved sufficient minimum supported co-degree condition for $k$-uniform $\ell$-cycles for all $k\ge 3$ and $\ell\in [k-1]$ except $(k,\ell)=(3,1)$, which is the best possible for infinitely many values of $k$ and $\ell$, and is off by at most $1$ for all values of $k$ and $\ell$. We discuss the optimality of our result in \Cref{sec:lower bound}. Importantly, we emphasise that our results are (essentially) exact and not asymptotic, a rather stark difference from the minimum co-degree version of the problem where very few cases have been resolved exactly.

\begin{thm} \label{thm:main}
	Let $1 \le \ell \le k-1$ be such that $k \ge 3$ and $(k,\ell) \neq (3,1)$, let $t = \floor{\frac{k}{k-\ell}}(k - \ell)$, and let $n$ be sufficiently large and divisible by $k-\ell$.
	Then every $k$-uniform $n$-vertex hypergraph without isolated vertices and having minimum supported co-degree at least $(1 - 1/t)n-(k-3)$ will contain a Hamilton $\ell$-cycle.
\end{thm}

This improves the aforementioned asymptotic results of Mycroft and Z\'arate-Guer\'en~\cite{mycroft2025positive}, who showed that $\delta^*(G) \ge (1 - 1/t)n + o(n)$ suffices. Our methods are largely different from theirs. Our analysis splits into two regimes: the extremal and non-extremal (see \Cref{sec:stability} for details). The extremal regime does not feature in \cite{mycroft2025positive} at all, and for this we use ad-hoc structural analysis with lemmas about random matchings in bipartite graphs and about finding almost spanning subhypergraphs with large minimum supported codegree in almost complete hypergraphs. For the non-extremal case, while Mycroft and Z\'arate-Guer\'en rely on a version of the regularity lemma for hypergraphs (called the weak regularity lemma) along with the absorption method, our approach stems from Lang's breakthrough work \cite{lang2023tiling} on hypergraph tilings, which allows us to completely avoid the regularity lemma and greatly simplifies the use of the absorption method. For the latter part, however, we use a novel idea from \cite{mycroft2025positive} about weighted fractional matchings that, coupled with Farkas' linear-algebraic lemma, allows one to find spanning $\ell$-cycles in certain $k$-partite $k$-graphs, which is one of the steps in our proof. 

We point out that Illingworth, Lang, M\"uyesser, Parczyk and Sgueglia~\cite{illingworth2025spanning} conjectured that for all $k\ge 3$, any sufficiently large $k$-graph with $\delta^*(G) \ge (1-1/k)n$ contains a tight Hamilton cycle, which is a special case of \cref{thm:main} (up to the $k-3$ term). Observe that if $k$ divides $n$, then a tight Hamilton cycle in $G$ will contain a perfect matching. Thus, our result implies that $\delta^*(G) \ge (1 - 1/k)n - (k-3)$ ensures a perfect matching, which recovers a result of Mycroft and Z\'arate-Guer\'en~\cite{mycroft2025matching} up to an additive constant of one, who show that a bound of $(1 - 1/k)n - (k - 2)$ suffices.

As alluded to above, the bound in \cref{thm:main} is off by one in most cases. In fact, in many cases (namely whenever $t \ge \ell+2$ or when $\ell = k-1$, that is, we are seeking a tight Hamilton cycle), our techniques can be adapted to prove \cref{thm:main} with the improved bound of $\delta^*(G) \ge (1 - 1/t)n - (k - 2)$ instead (which, for instance, would immediately imply the previously discussed optimal perfect matching bound). 
However, since this will require some technical modifications, we choose to present a single unified proof for $\delta^*(G) \ge (1 - 1/t)n - (k - 3)$ that will work for all cases instead. We point out the part of our proof that requires the exact degree condition in \Cref{rem:exact degree}, and briefly discuss how to suitably alter our proof and remove the extra one in the bound in all relevant cases in \Cref{sec:improving the bound}.

Finally, we remark that in a recent personal communication, Mycroft and Z\'arate-Guer\'en informed us that they are preparing a manuscript where they prove that any sufficiently large $3$-graph with no isolated vertices and $\delta^*(G) \ge n/2$ contains a loose Hamilton cycle. This improves the previous asymptotic bound of $n/2 + o(n)$ established by them~\cite{mycroft2025positive} and Halfpap and Magnan~\cite{halfpap2024positive} (which our methods can recover as well) and extends \cref{thm:main} to include $(k,\ell) = (3,1)$, the only case we do not handle.

\paragraph{Organisation of the paper.} In \Cref{sec:overview}, we first provide lower bound constructions to prove the optimality of \cref{thm:main}. We then split the proof of \cref{thm:main} into two complementary cases and handle them separately in \cref{thm:Hamilton cycle non-extremal,thm:Hamilton cycle extremal}. We also provide brief proof overviews for both these theorems in \Cref{sec:proof-overviews}. Then in \Cref{sec:prelims} we introduce some notation and definitions and provide some basic tools.

\Crefrange{sec:non-extremal}{sec:proof Hamilton paths} are dedicated to proving \cref{thm:Hamilton cycle non-extremal}, and \Crefrange{sec:extremal hypergraphs}{sec:Hamilton cycle extremal} prove \cref{thm:Hamilton cycle extremal}. These two parts are treated largely independently, and can be read as such. We conclude with some open problems in \Cref{sec:conclusion}.

\section{Overview}\label{sec:overview}

	In this section we first give three extremal constructions (see \Cref{sec:lower bound}). We then state, in \Cref{sec:stability}, two theorems that will split the task of finding a Hamilton $\ell$-cycle in a $k$-graph with appropriate minimum supported codegree into two cases: non-extremal and extremal, and observe that these theorems together imply our main result. Finally, we give brief proof sketches for each theorem in \Cref{sec:proof-overviews}.

\subsection{Extremal lower bound constructions}\label{sec:lower bound}
    We first discuss the optimality of the lower bound in \cref{thm:main}. To begin with, we say that a subset $U$ of at most $k$ vertices in a $k$-graph $G$ is \emph{supported} if there is an edge containing $U$. Furthermore, if $|U| \le k-1$, set $d_G^1(U)$ to be the number of vertices $v \in V(G)$ such that $U \cup \{v\}$ is supported. We say that a subset $A$ of vertices is a \emph{strong independent set} if every edge intersects $A$ in at most one vertex. 
    
    We will provide three lower bound constructions. The first one will show that, at the very least, we need $\delta^*(G) \ge (1 - 1/t)n - (k - 2)$ in \cref{thm:main} for any $k\ge 3$, $\ell \in [k-1]$ and $n$ which is divisible by $k-\ell$. The second one, which is a slight modification of the first, shows that if $k$ is odd, $\ell = (k - 1)/2$, and $n$ satisfies some divisibility conditions, then we require $\delta^*(G) \ge (1 - 1/t)n - (k - 3)$, which is why our main result is essentially tight and is the strongest possible bound that applies to all relevant $k$, $\ell$ and $n$. The final one will be an entirely different construction that applies only to the $(k,\ell) = (3,1)$ case and shows that we need $\delta^*(G) \ge n/2$, which is in line with the second example.
    
    \begin{ex}\label{ex:lower bd weak}
        Let $k\ge 3$, $\ell \in [k-1]$ and $t = \floor{\frac{k}{k-\ell}}(k - \ell)$. Let $n$ be chosen such that $k-\ell$ divides $n$. Now, consider an $n$-vertex $k$-graph $G$ with a partition $A\sqcup B$ of its vertex set such that $|A| = \floor{n/t} + 1$ and $E(G)$ consists of all $k$-sets $e \in \binom{V(G)}{k}$ that satisfy $|e \cap A| \le 1$ (and so $A$ is a strong independent set). Suppose $S$ is any supported $(k-1)$-set. If $|S\cap A| = 0$, then $S$ supports all vertices outside itself, and so $d_G^1(S) = n-k+1$. If $|S\cap A| = 1$, then it supports all vertices in $B\setminus S$, meaning that $d_G^1(S) = n - \floor{n/t} - (k-2)$, and this is equal to $\delta^*(G)$ since there is no supported set with $|S\cap A| \ge 2$. The following observation shows that $G$ cannot contains a Hamilton $\ell$-cycle.
	\end{ex}

	\begin{obs}\label{obs:strong ind set}
		Let $H$ be an $n$-vertex $k$-uniform $\ell$-cycle. Then the maximum size of a strong independent set in $H$ is at most $\floor{n/t}$.
	\end{obs}
	
	\begin{proof}
		As noted previously, since $H$ is an $\ell$-cycle, $k-\ell$ divides $n$. Recall that we can label the vertices in $H$ as $v_1 \dots v_n$ such that $v_{i(k-\ell) + 1} \dots v_{i(k - \ell) + k}$ is an edge for all $i\ge 0$, with indices seen cyclically modulo $n$. Define the set of \emph{segments} to be the collection of subsequences $\cS = \{v_{i(k - \ell) + 1} \dots v_{i(k - \ell) + t}: i\ge 0\}$ and the set of \emph{intervals} $\cI = \{v_{i(k - \ell) + 1} \dots v_{(i+1)(k - \ell)} : i\ge 0\}$, where we view indices cyclically modulo $n$ as usual. Let $X$ be a strong independent set.

		First, we observe that every segment contains exactly $t/(k - \ell)$ intervals, and that every interval is part of $t/(k - \ell)$ segments. Since the intervals partition $V(H)$, we see that every vertex is contained in $t/(k - \ell)$ segments. The crucial observation is that, due to the definition of $\ell$-cycles and since trivially $t\le k$, every segment is contained in an edge of $H$, and consequently every segment contains at most one vertex of $X$. Now, consider the set of vertex-segment pairs $\{(v, P) : v\in X, P\in \cS, x\in P\}$. Since each vertex belongs to $t/(k - \ell)$ segments, we see that the size of this set is exactly $t|X|/(k - \ell)$. However, as observed previously, each segment $P\in \cS$ can contain at most one vertex $v\in X$, and hence the number of such pairs is at most $|\cS| = n/(k - \ell)$. Comparing the two shows $|X| \le n/t$ and hence $|X| \le \floor{n/t}$.
	\end{proof}

	We remark that Mycroft and Z{\'a}rate-Guer{\'e}n~\cite[Section 1.4]{mycroft2025positive} prove the same bound with this extremal construction, but we include our proof as it is simpler. We also wish to point out that Illingworth, Lang, M\"uyesser, Parczyk and Sgueglia~\cite{illingworth2025spanning} provide a version of this construction for tight cycles.

\begin{ex}\label{ex:lower bd strong}
    Next, we provide a specialised modification of the previous example. Let $k\ge 3$ be odd, $\ell = (k-1)/2$, and consequently $t = \floor{\frac{k}{k-\ell}}(k - \ell) = (k+1)/2 = \ell+1 = k - \ell$. Let $n$ be such that $t$ divides $n$ and $n/t+1$ is even. Yet again, we consider an $n$-vertex $k$-graph $G$ with a partition $A\sqcup B$ of its vertex set with $|A| = n/t + 1$. Write $A =  \{a_1, \dots, a_{|A|}\}$ and let $E(G)$ consist of all $e\in \binom{V(G)}{k}$ such that $e\cap A \subseteq \{a_{2i-1}, a_{2i}\}$ for some $i\in [|A|/2]$ (so the graph $\partial^2[A]$ is a perfect matching). Clearly $|e\cap A| \le 2$ for all $e \in E(G)$. If $S$ is any supported $(k-1)$-set, then, similar to the previous example, it is not too hard to see that
    \[d_G^1(S) =
		\left\{
		\begin{array}{ll}
            n - (k-1) & \textnormal{if} \ |e\cap A|=0,\\
            |B| - (k - 2) +1 = |B| - (k-3) & \textnormal{if} \ |e\cap A|=1,\\
            |B| - (k - 3) & \textnormal{if} \ |e\cap A|=2,\\
        \end{array}
		\right.
    \]
    and so $\delta^*(G) = |B| - (k - 3) =  n - n/t - (k - 2)$. We now argue that $G$ cannot contain a Hamilton $\ell$-cycle, which shows that the bound in \cref{thm:main} cannot be improved.

    To the contrary, suppose that there exists such a cycle $C = v_1 \dots v_n$, and let $E(C)$ denote the edges of $G$ that are part of the cycle $C$. Every edge $e \in E(C)$ shares $\ell = (k-1)/2$ vertices with the preceding and succeeding edges of the cycle, and hence contains a unique vertex $v_e$ at position $\ell + 1 = t$ (the ``middle'' vertex) that is not part of any other edge of $E(C)$.
	Let $C^t$ denote the set of these vertices, so that $|C^t| = n/t$. Another important observation we require is that if $u\in A$ and $u v_1 \dots v_{k-1} \in E(G)$, then $v v_1 \dots v_{k-1} \in E(G)$ for any $v\in B \setminus \{v_1, \dots, v_{k-1} \} $.
    
    Consider any edge $e\in E(C)$ such that $e\cap A \ne \emptyset$. If $v_e \in B$, then there exists some $u \in e \cap A$ that is distinct from $v_e$. From our observation above, and since $e$ is the only edge of $C$ that contains $v_e$, the vertex sequence obtained by swapping $u$ and $v_e$ in the sequence corresponding to $C$ still yields a Hamilton $\ell$-cycle. Hence, we may assume that $v_e \in A$ for all $e\in E(C)$ that intersect $A$. Next, suppose there exists some $u \in A \setminus C^t$. The discussion in the previous paragraph implies that are two edges $e,f \in E(C)$ that contain $u$. By choice of $C$, as $A$ intersects $e$ and $f$, we know that $v_e, v_f \in A$. Hence, we may conclude that $u v_e$ and $u v_f$ are both supported pairs due to the edges $e$ and $f$ respectively, meaning that $u$ is a vertex of degree at least two in $\partial^2[A]$, a contradiction. Thus, we see that $A\subseteq C^t$, which implies $|A| \le |C^t| = n/t$, providing the desired contradiction.
\end{ex}

\begin{ex}\label{ex:lower bd 3 unif loose}
    Suppose $k=3$ and $\ell=1$, so that $t = \floor{\frac{k}{k-\ell}}(k-\ell) = 2$, and suppose $n\equiv 2 \pmod{4}$. Let $G$ be an $n$-vertex $3$-graph that consists of the union of two complete $3$-graphs $H_1$ and $H_2$, each on $n/2 +1$ vertices, such that $H_1$ and $H_2$ have precisely two vertices in common, say $a$ and $b$. Since any supported $2$-set must be contained entirely in $H_i$ for some $i$, we see that $\delta^*(G) = |H_i| - 2 = n/t -1$.

    Suppose $G$ has a $3$-uniform loose cycle $C= v_1 \dots v_n$. Then, since any two consecutive vertices are supported, we see that any $v\in V(H_i)\setminus \{a,b\}$ must be preceded and succeeded only by vertices of $V(H_i)$ in the cycle $C$. From this, it is easy to see that $C$ must contain a sequence of consecutive vertices corresponding to a loose Hamilton path in $H_i$ for each $i$. However, any $3$-uniform loose path must have odd order (since the first edge has three vertices and each consecutive edge adds two new vertices), and $|H_i| = n/2+1$ is even, which is a contradiction.
\end{ex}

We point out that the last example was brought to our attention by Richard Mycroft and Camila Z\'arate-Guer\'en, and we are thankful for that.

\subsection{Proof of the main result using two key theorems}\label{sec:stability}
	We prove our main result, \cref{thm:main}, by splitting it into two disjoint and complementary cases. Roughly speaking, we handle hypergraphs that are structurally ``near-extremal'' and ``non-extremal'' separately. We classify hypergraphs into one of these two types based on the observation that both extremal examples contain a large ``sparse'' set of size roughly $n/t$ with very few supported pairs.

    The first theorem treats $k$-graphs that are far from extremal, that is, where every set of size at least $n/t$ induces several supported pairs. Note that here we allow for a minimum supported co-degree which is slightly lower than the exact bound given in \Cref{thm:main}.

	\begin{restatable}{thm}{thmHamiltonCycleNonExtremal}
		\label{thm:Hamilton cycle non-extremal}
		Let $k \ge 3$ and $\ell \in [k-1]$ be such that $(k,\ell) \neq (3,1)$, and set $t = \floor{\frac{k}{k-\ell}}(k-\ell)$. Let $1/n \ll \eps \ll \mu \ll 1/k$ such that $n$ is divisible by $k-\ell$. If $G$ is an $n$-vertex $k$-graph with no isolated vertices and $\delta^*(G)\geq (1-1/t-\eps)n$ such that every set of $\floor{n/t}$ vertices of $G$ contains at least $\mu n^2$ supported pairs, then $G$ has a Hamilton $\ell$-cycle.
	\end{restatable}

    The second theorem deals with $k$-graphs that are close to the two extremal examples above, namely where there is a set of size $\floor{n/t}$ which contains few supported pairs. This result requires the exact bound on the minimum supported co-degree.

	\begin{restatable}{thm}{thmHamiltonCycleExtremal}
		\label{thm:Hamilton cycle extremal}    
		Let $k \ge 3$ and $\ell \in [k-1]$ be such that $(k,\ell) \neq (3,1)$, and set $t = \floor{\frac{k}{k-\ell}}(k-\ell)$.
		Let $1/n\ll \eps \ll 1/k\leq 1/3$ be such that $n$ is divisible by $k-\ell$. 
		Suppose $G$ is an $n$-vertex $k$-graph with no isolated vertices and $\delta^*(G)\geq (1-1/t)n-(k-3)$. If there is a subset $A \subseteq V(G)$ with $|A|=\floor{n/t}$ that contains at most $\eps n^2$ supported pairs, then $G$ has a Hamilton $\ell$-cycle.
	\end{restatable}

	Notice that the proof of \Cref{thm:main} follows immediately from \Cref{thm:Hamilton cycle non-extremal,thm:Hamilton cycle extremal}.
	For the most part, the proofs of the two theorems use separate arguments and tools. In the rest of the section, we provide short overviews for the proofs of \cref{thm:Hamilton cycle non-extremal,thm:Hamilton cycle extremal}.

\subsection{Proof overviews} \label{sec:proof-overviews}
\subsubsection{Proof overview for \cref{thm:Hamilton cycle non-extremal}}\label{sec:pf overview non-extremal}
In this section, we provide a brief overview of the key ideas we use to prove \cref{thm:Hamilton cycle non-extremal}, and include a more detailed proof sketch in \Cref{sec:pf sketch non-extremal}. 

The approach we use is built on a blow-up tiling technique introduced by Lang~\cite{lang2023tiling}, and extended in~\cite{illingworth2025spanning,lang2024hypergraph} (we make precise what we mean by a blow-up in \Cref{sec:main lemmas non extremal}). Our ``blow-up tiling lemma'', namely \cref{lem:blow-up chain}, allows us to tile $G$ with almost balanced blow-ups of smaller $k$-graphs that obey approximate versions of the degree condition and the non-extremal structure of $G$. These tiles will have a special cyclic linkage property in terms of common edges, which will then reduce our problem to finding a Hamilton $\ell$-path within each blown up tile, because these can be ``chained together'' to form a Hamilton $\ell$-cycle in $G$. 

Hence, given any ``non-extremal'' $F$ and a sufficiently large nearly balanced blow-up $F^*$, we want a Hamilton $\ell$-path in $F^*$. We will first partition $F^*$ into complete $k$-partite $k$-graphs with carefully selected part sizes to ensure each has an almost spanning $\ell$-path. These $k$-partite $k$-graphs will be formed by splitting up the complete $k$-partite $k$-graphs corresponding to the blow-ups of edges of $F$. We determine how to split up these $k$-graphs based on edge weights provided by a perfect fractional matching (defined in \Cref{sec:Farkas}) in $F$. Technically, we work with a ``vertex weighted'' version of fractional matchings  due to Mycroft and Z{\'a}rate-Guer{\'e}n~\cite{mycroft2025positive}, which is crucial for $\ell$-cycles that are not necessarily tight. 

Finally, we use the minimum supported co-degree condition to find a path in $F^*$ (obtained by blowing up a suitable walk in $F$) that can be used to link together the nearly spanning $\ell$-paths found in the aforementioned $k$-partite subgraphs and absorb a few uncovered vertices.

\subsubsection{Proof overview for \cref{thm:Hamilton cycle extremal}}\label{sec:pf overview extremal}
We provide a short overview of our proof strategy for \cref{thm:Hamilton cycle extremal} for the special case of tight cycles (which is what motivates the general strategy) and include a more detailed sketch including $\ell$-cycles in Section~\ref{sec:pf sketch extremal}. 
Suppose $G$ is an ``extremal'' $n$-vertex $k$-graph, meaning it contains a set $A$ of size roughly $n/k$ (since $t = k$ for tight cycles) within which there are few supported pairs. Let $B =  V(G) \setminus A$. \Cref{ex:lower bd weak} suggests that a tight Hamilton cycle has a strong independent set of size $n/k$, whose vertices are exactly the $k$th vertices along the cycle. We build a Hamilton cycle in $G$ based on this observation.

We first show that we can transfer a few vertices across $A$ and $B$ to ensure high co-degrees for supported sets across the partition while maintaining $|A| \approx n/k$. We then replace a few vertices $A$ with short tight paths, and this will recover $|B| = (k-1)|A|$.

We then partition $B$ into $k-1$ sets $B_1, \dots, B_{k-1}$ such that $|A| = |B_i|$ for all $i$. We define an auxiliary $k$-partite $k$-graph $\Gp$ with parts $B_1 \times \dots \times B_{k-1} \times A$ whose edges correspond to vertex sequences that form tight paths, and randomly construct a perfect matching $M$. The randomness will maintain the high co-degree conditions so that the edges of $M$ can be connected into a tight Hamilton cycle in $G$.

\section{Preliminaries}\label{sec:prelims}

	In this section we provide notation and a few simple preliminary results and standard probabilistic tools that will be used in the proofs of both the non-extremal and extremal theorems. More specialised notation and results, pertaining to only one of the theorems, will be mentioned at the beginning of the relevant part of the paper.

\subsection{Notation}\label{sec:notation}
For any $m_1, m_2 \in \mathbb{N}$ with $m_1 \le m_2$, we let $[m_1]$ denote the set $\{1, 2, \dots, m_1\}$ and often use $[m_1, m_2]$ to denote $\{m_1, m_1 + 1, \dots, m_2\}$. 
Given any set $S$ and any $0\le m \le |S|$, we use $\binom{S}{m}$ to denote the set of all $m$-subsets of $S$. Throughout the paper, we frequently use ``$\ll$'' notation. We write $a \ll b$ to mean that for any choice of $b > 0$, there exists some $a_0 > 0$ such that the statement in question holds for all $a \le a_0$. We may sometimes write $b \gg a$ for the same. 

For a $k$-graph $G = (V(G), E(G))$, we use $|G|$ to denote the number of vertices and $e(G)$ to denote the number of edges. We often simply write $v_1 \dots v_k$ for an edge $\{v_1, \dots, v_k\} \in E(G)$. We let $E^*(G)$ denote the set of ordered edges of $G$, that is, the set of all ordered tuples $(v_1, \dots, v_k)$ such that $\{v_1, \dots, v_k\} \in E(G)$ (and hence each edge of $E(G)$ leads to $k!$ ordered edges in $E^*(G)$). We sometimes denote this ordered edge as $v_1 \dots v_k$ as well when it is clear from context (or explicitly specified) that we are dealing with ordered edges.

For any $U \subseteq V(G)$, we let $G[U]$ denote the subgraph of $G$ induced by $U$. We say that a vertex subset $S \subseteq V(G)$ with $|S|\le k$ is \emph{supported in $G$} if these exists some edge of $G$ containing $S$. Analogous to ordered edges, we define an ordered supported set to be an ordered tuple $v_1 \dots v_s = (v_1, \dots, v_s)$ such that the set $\{v_1, \dots, v_s\}$ is supported in $G$. For any $i\in [k-1]$, we let define the \emph{$i$-shadow of $G$}, denoted $\partial^i G$, to be the $i$-graph with $V(\partial^i G)=V(G)$ whose edges are precisely the supported $i$-sets of $G$. We define $\partial_G^i[U]$ to be the subgraph induced by a vertex subset $U$, and will often drop the subscript $G$ and write $\partial^i[U]$, as long as it does not lead to any ambiguity.

Given any $n\gg k\geq 3$ and $\ell\in [k-1]$, we set $t=t(k,\ell)=\floor{\frac{k}{k-\ell}}(k-\ell)$, and we refer to this parameter $t$ throughout the paper. Define $\exE_\eps(n,k)$ to be the set of all $n$-vertex $k$-graphs $G$ with $n$ divisible by $k - \ell$ that have no isolated vertices, satisfy $\delta^*(G)\geq n - \floor{n/t} - (k-3)$, and for which there is a set of $\floor{n/t}$ vertices containing at most $\eps n^2$ supported pairs. We define the family $\nonexE_{\eps,\mu}(n,k,\ell)$ as the collection of all $n$-vertex $k$-graphs with $n$ divisible by $k - \ell$ that have no isolated vertices, satisfy $\delta^*(G) \ge (1 - 1/t - \eps)n$ and for which every subset of at least $n/t$ vertices contains at least $\mu n^2$ supported pairs.

For a $k$-graph $G$ and $\ell \in [k-1]$, a \emph{$k$-uniform $\ell$-walk} $W$ in $G$ is a sequence of (not necessarily distinct) vertices $v_1\dots v_r$ such that $r \equiv k \pmod{k-\ell}$ and $v_{s(k-\ell)+1} \ldots v_{s(k-\ell)+k}$ is an edge for every integer $s \in \left[0,\frac{r-k}{k-\ell} \right]$.
In particular, consecutive edges $v_{s(k-\ell)+1} \ldots v_{s(k-\ell)+k}$ and $v_{(s+1)(k-\ell)+1} \ldots v_{(s+1)(k-\ell)+k}$ share the $\ell$ vertices $v_{(s+1)(k-\ell)+1}, \ldots, v_{(s+1)(k-\ell)+\ell}$ (they possibly share more vertices as vertices are allowed to repeat). We say that the $\ell$-walk $W=v_1\dots v_r$ has \emph{order} $r$, and will denote its order as $|W|$. We call such a $W$ an \emph{$\ell$-path} if all the $v_i$ are distinct. We call $W$ a \emph{tight walk} if $\ell = k-1$ and call it a \emph{loose walk} if $\ell = 1$. We define tight and loose paths analogously.

A \emph{$k$-uniform $\ell$-cycle} $C$ is a cyclic sequence of distinct vertices $v_1 \ldots v_r$, where $r$ is divisible by $k-\ell$ and $v_{s(k-\ell)+1} \ldots v_{s(k-\ell)+k}$ is an edge for every $s \in \left[0, \frac{r-k}{k-\ell} \right]$ (addition of indices is taken modulo $r$). A $k$-uniform $(k-1)$-cycle is called a \emph{tight cycle}, and a $1$-cycle is called a \emph{loose cycle}. Finally, we say that an $\ell$-path or an $\ell$-cycle is \emph{Hamiltonian} if every vertex of $G$ appears in the corresponding sequence exactly once, and will call these structures Hamilton $\ell$-paths and Hamilton $\ell$-cycles respectively.

Additionally, we also introduce some non-standard notation that is motivated by our proof strategy. For any supported $i$-set $S\subseteq V(G)$, we define the \emph{vertex neighbourhood} of $S$, denoted $N_G^1(S)$, to be the set of all vertices $v\in V(G)$ such that $S\cup\{v\}$ is a supported $(i+1)$-set, and set $d_G^1(S)=|N_G^1(S)|$ to be the \emph{vertex co-degree} of $S$. For $S=\{v_1,\dots,v_i\}$, we will simply write $N_G^1(v_1\dots v_i)$ and $d_G^1(v_1\dots v_i)$ for $N_G^1(S)$ and $d_G^1(S)$ respectively. As always, we may sometimes drop the subscript $G$ if there is no ambiguity in the host graph. Furthermore, for any $U\subseteq V(G)$, we will denote $N_G^1(S)\cap U$ as $N_U^1(S)$ and $d_U^1(S)=|N_U^1(S)|$.
Finally, for an arbitrary $i$-set $S$ (which may not be supported), we define the \emph{degree of $S$}, denoted $d(S)$, to be the number of edges of $G$ which contain $S$. The \emph{minimum $i$-degree} of $G$, denoted $\delta_i(G)$, is the maximum integer $d$ such that every set of $i$ vertices is contained in at least $d$ edges.

\subsection{Preliminary results}

    First, we mention a few quick facts about $t = \floor{\frac{k}{k-\ell}}(k-\ell)$. Since these are fairly simple, we may often use these without quoting the observation below.
	\begin{obs}
		\label{obs:useful t info}
		Let $k \ge 3$, $\ell \in [k-1]$, and $t = \floor{\frac{k}{k-\ell}}(k - \ell)$. Then $\ell +1 \le t \le k$, $t \ge \frac{k+1}{2}$, and $t \ge 3$ unless $(k,\ell) = (3,1)$.
	\end{obs}
    \begin{proof}
        We prove the claims sequentially. For the first one, we see that
        \[
            t > \left( \frac{k}{k-\ell} - 1 \right) (k - \ell) = \ell,
        \]
        and the fact that $t\le k$ is immediate from the definition.
        
        We prove the second claim in two cases. If $\ell \ge k/2$, then we are done by the first part. Otherwise, we will have $k - \ell > k/2$ which implies $t\ge k - \ell \ge (k + 1)/2$.
        
        The final claim is easy since $t \in \{1,2\}$ requires at least one of $\floor{\frac{k}{k - \ell}}$ and $(k - \ell)$ to be equal to one and the other to be equal to either one or two depending on the desired value of $t$. Then, as $k \ge 3$, it is simple to check that $t = 2$ for $(k, \ell) = (3, 1)$ and $t \ge 3$ otherwise.
    \end{proof}

Next, we show a quick result about certain vertices in $\ell$-paths that will be required for both the extremal and non-extremal cases. Intuitively speaking, these vertices form the ``sparse'' set of size $n/t$ in the extremal constructions. In a vertex sequence corresponding to an $\ell$-path or $\ell$-cycle, an edge always starts at a vertex immediately after any of these special vertices. This motivates the proof strategy described in \Cref{sec:pf overview extremal} for near-extremal hypergraphs.

\begin{prop}
\label{prop:special vxs in ell paths}
    Let $k\ge 3$, $\ell \in [k-1]$ and $t = \floor{\frac{k}{k - \ell}}(k-\ell)$. Suppose $G$ is a $k$-graph and $W=v_1\dots v_r$ is a vertex sequence, where $r$ is divisible by $k-\ell$. Let $E_W = \{v_{i(k-\ell)+1} \dots v_{i(k-\ell)+k}: i\ge 0\}$ be a collection of subsequences of $W$ and $A=\{v_j:j\equiv 0\pmod t \}$, where  we view all indices cyclically modulo $r$. Then every $e \in E_W$ contains exactly one vertex from $A$.
\end{prop}

\begin{proof}
    The edges of $W$ are of the form $v_{i(k-\ell)+1} \dots v_{i(k-\ell)+k}$. Since $t\le k$, each edge contains at least one vertex from $A$. We will show that no edge has two vertices from $A$. 

	Fix an edge $v_{i(k-\ell)+1} \ldots v_{i(k-\ell)+k}$.
    Observe that if $v_j\in A$, then $j$ is a multiple of $k-\ell$. Consequently, if $j$ is the smallest index such that $v_j\in A\cap\{v_{i(k-\ell)+1},\dots,v_{i(k-\ell)+k}\}$, then $j\ge(i+1)(k-\ell)$. Now, we see that
    \begin{align*}
        j+t&\ge (i+1)(k-\ell)+\floor{\frac{k}{k-\ell}}(k-\ell)\\
        &>(i+1)(k-\ell)+\left(\frac{k}{k-\ell}-1\right)(k-\ell)\\
        &=i(k-\ell)+k,
    \end{align*}
    which means there is no other vertex of $A$ in the edge $\{v_{i(k-\ell)+1},\dots,v_{i(k-\ell)+k}\}$.
\end{proof}

\subsection{Probabilistic tools}

Next, we state (without proof) standard Chernoff bounds for binomial and hypergreomtric random variables. 
\begin{lem}
\label{lem:Chernoff}
Let $X_1,\dots,X_n$ be random variables that take values in $\{0,1\}$ such that $\bP[X_i=1]=p$ for all $i$. Suppose $X=X_1+\dots + X_n$ has either a binomial or a hypergeometric distribution with sample size $n$ and success probability $p$. Then for any $a>0$,
\[
\bP[|X-\bE X|>a]=\bP[|X-np|>a]\leq 2\exp\left(-\frac{a^2}{3np}\right)
\]
\end{lem}

Another tool we will require is McDiarmid's inequality. 
\begin{lem}[McDiarmid's inequality \cite{mcdiarmid1989method}]
\label{lem:McDiarmid}
    Let $X_1, \ldots, X_s$ be independent random variables taking values in sets $\Omega_1, \ldots, \Omega_s$, respectively. Let $f:\Omega_1\times\dots\times\Omega_s\to \bR$ be a function. Suppose there exist constants $c_1,\dots, c_s\in \bR$ such that, for any $i\in [s]$,
    \[
    |f(x_1,\dots,x_s)-f(x_1',\dots,x_s')|\le c_i
    \]
    for any $(x_1,\dots,x_s), (x_1',\dots,x_s')\in \Omega_1\times\dots\times\Omega_s$ that differ only in the $i$th coordinate. Then the random variable $Z=f(X_1,\dots,X_s)$ satisfies the following. For every $\lambda\ge 0$,
    \[
    \bP[|Z-\bE Z|\ge \lambda]\le 2\exp\left(\frac{-2\lambda^2}{c_1^2+\dots+c_s^2} \right).
    \]
\end{lem}

\section{Non-extremal hypergraphs} \label{sec:non-extremal}
	We will first prove \cref{thm:Hamilton cycle extremal} (the non-extremal case) across the next six sections. In this section, we first provide a detailed proof sketch, and then briefly describe the principal results and techniques that we utilise. We finally state three key lemmas -- \cref{lem:blow-up chain,lem:concentration,lem:Hamilton paths in blow-ups} -- and show how they imply \cref{thm:Hamilton cycle extremal}. The remaining five sections and \Cref{appendix:proof-blow-up-tiling} then prove these three chief lemmas.

\subsection{Proof sketch for \cref{thm:Hamilton cycle non-extremal}}\label{sec:pf sketch non-extremal}
Our goal is to show every $G \in \nonexE_{\eps, \mu} (n, k, \ell)$ has a Hamilton $\ell$-cycle. For the sake of clarity, we only sketch the proof for tight cycles ($\ell = k-1$, in which case $t = k$), and briefly discuss the necessary modifications for $\ell$-cycles at the end. Our proof has three main steps

\paragraph{Tiling.} As mentioned in \Cref{sec:pf overview non-extremal}, the main lemma behind our strategy, namely \cref{lem:blow-up chain}, provides a tiling of $G$ with almost balanced blow-ups (defined in \Cref{sec:main lemmas non extremal}) of smaller $k$-graphs that obey (almost) the same minimum supported co-degree condition as $G$ and also have an approximate version of the non-extremal structure of $G$. Moreover, these blown-up tiles will be linked in a cyclic fashion, where each tile shares a common edge with the previous and next tile in the chain, and is disjoint from the rest. Consequently, within each blown-up tile, we find a tight Hamilton path that starts at the edge shared with the previous tile, and ends at the edge shared with the next tile. The tight Hamilton paths in consecutive tiles can then be ``linked together'', and repeating this process along the chain of cyclically linked tiles forms a tight Hamilton cycle in $G$.

\paragraph{Blowing up a fractional matching.} Thus (morally speaking), our problem is reduced to finding a tight Hamilton path in a $k$-graph $F^*$ that is an almost balanced blow-up of some $F \in \nonexE_{\eps, \mu} (n, k, k-1)$, with the start and end edges of the path specified a priori. Our strategy starts by (nearly) partitioning $F^*$ into a collection of balanced complete $k$-partite $k$-graphs (which will be suitable subgraphs of the $k$-graphs formed by blow-ups of edges of $F$) that cover almost all of $F^*$. To this end, we first find a perfect fractional matching in $F$, which is a set of non-negative edge weights such that the total weight on each vertex is precisely one. We do this via Farkas' lemma, a standard linear algebraic technique. We do not describe how we use it to find the perfect fractional matching here as this part is fairly disjoint from the rest of the proof, we just mention that we make use of both the degree condition and the non-extremal structure of $F$. Now, for each vertex $v \in V(F)$, let $B_v \subseteq V(F^*)$ denote the blown-up part corresponding to $v$. We partition each $B_v$ into parts corresponding to each edge of $e\in E(F)$ that contains $v$ such that the proportion of $B_v$ covered by this part $B_v^e$ is given by the weight of the edge $e$ in the fractional matching. Thus, we end up with a subgraph $T_e$ (whose parts are the sets $B_v^e$ for all vertices $v\in e$) of the original $k$-partite $k$-graph formed by the blow-up of $e$. All the parts of $T_e$ will have the same size (governed by the weight of $e$), and it is easy to see $T_e$ has a spanning tight cycle. Ideally, since the fractional matching is perfect, these disjoint subgraphs will span all the vertices of $G$. However, we may end up having to round down the sizes of these parts $B_v^e$ to integer values, meaning that we lose some vertices and consequently the $k$-graphs $\{T_e\}_{e \in E(F)}$ only cover almost all the vertices of $G$.

\paragraph{Connection and absorption.} Finally, we find a tight path in $F^*$ that can be used to link up the spanning tight paths found in the aforementioned $k$-partite subgraphs, and can also ``absorb'' into itself the remaining few uncovered vertices (technically, we first find such a path and then apply the strategy outlined in the previous paragraph to the leftover). In order to build such a tight path $P$, we will heavily rely on the fact that $\delta^*(F) > n/2$ (which will hold even in the $\ell$-cycle case as $(k, \ell) \ne (3, 1)$ implies $t \ge 3$). Since $\delta^*(F) > n/2$, it is easy to argue that we can find a tight walk $W$ in $F$ that contains (several copies of) every edge of $F$. We can also specify the start and end edges, which, as discussed, is important to our proof. This is blown up to a tight path $P$ in $F^*$ by walking along $W$ and picking a new copy of a vertex whenever we see a repeated vertex, which is possible if the blow-up is sufficiently large relative to $|F|$. We then show that we can replace any segment of $P$ formed by a blow-up of an edge $e$ by a spanning path in the $k$-graph $T_e$, which allows us to connect the necessary tight paths. All that is left is to imbue $W$, and consequently $P$, with some absorption features. Given a vertex $v \in V(F)$, the ``simple absorber'' we use will consist of two supported sets $W_1$ and $W_2$ such that both $W_1 W_2$ and $W_1 v W_2$ are tight paths. In a manner similar to what we describe above, we can exploit the connectivity property due to the degree condition to ensure that the tight walk $W$, which contains every edge of $F$ as a subwalk, also has several absorber subwalks $W_1 W_2$ for each vertex $v$. The blown-up path $P$ in $F^*$ then retains these absorption attributes, allowing it to incorporate into itself the few unused vertices in $F^*$. 

\paragraph{Adjustments for $\ell$-cycles.}
The broad proof structure for $\ell$-cycles remains the same, but the details need more care. To begin with, the complete $k$-partite subgraphs $T_e$ will now be modified to have carefully chosen part sizes (described in \cref{prop:ell cycle in k-partite k-graph}) that will ensure a spanning $\ell$-cycle instead. As mentioned in \Cref{sec:pf overview non-extremal}, this is done using a vertex-weighted generalisation of fractional matchings introduced by Mycroft and Z{\'a}rate-Guer{\'e}n~\cite{mycroft2025positive}, which will lead to the ``correct'' imbalance in the parts of $T_e$. The connectivity aspect of the proof remains fairly similar. In fact, we usually form the required $\ell$-walks by building tight walks whose orders satisfy specific divisibility conditions such that the same vertex sequences will also correspond to $\ell$-walks and $\ell$-paths (see \cref{obs:tight to ell walk}). Additionally, while the starting point for the absorbers we use for $\ell$-cycles are the same as the simple absorbers in the tight case, we need to be careful with the length of the walks and we need to be able to absorb $(k-\ell)$ vertices at a time, as detailed in \Cref{sec:absorbing walk}.

 \subsection{Main lemmas and proof of \cref{thm:Hamilton cycle non-extremal}}\label{sec:main lemmas non extremal}
Before we state our three chief lemmas, we introduce necessary notation.
Given a $k$-graph $F$, we say that a $k$-graph $F^*$ is a \emph{blow-up} of $F$ if it can be obtained by replacing each vertex $v\in V(F)$ with an independent set $B_v$ and each edge $v_1 \dots v_k\in E(F)$ with every possible $k$-partite edge across $B_{v_1},\dots, B_{v_k}$, that is, with a complete $k$-partite $k$-graph across the corresponding parts. We call $F^*$ a \emph{$(\gamma,m)$-regular blow-up} if the size of each part lies in the interval $[(1-\gamma)m,(1+\gamma)m]$. 
We say that $F^*$ is a \emph{$(\gamma,m)$-nearly-regular blow-up} if all but at most one part has size in the interval $[(1-\gamma)m,(1+\gamma)m]$, and the remaining part (if it exists) has size exactly one. In particular, every $(\gamma,m)$-regular blow-up is a $(\gamma,m)$-nearly-regular blow-up.
If $F^*$ is a blow-up of $F$, we define the \emph{projection map} $\phi: V(F^*)\to V(F)$ defined by $\phi(u)=v$ if $u$ is a blow-up of $v$, that is, $u\in B_v$. For any set $U\subseteq V(F^*)$, we denote $\phi(U)=\{\phi(u): u\in U\}$. Finally, we will often refer to blow-ups of ordered edges, which we formalise now. For an arbitrary ordered edge $e = v_1 \dots v_k \in E^*(F)$, we say that an ordered edge of $F^*$, say $e^* = u_1 \dots u_k \in E^*(F^*)$, is a \emph{blow-up of $e$} if $u_i \in B_{v_i}$ for every $i\in [k]$. 

Here is the general ``tiling by blow-ups'' lemma. It is very similar to Lemma 2.1 from \cite{illingworth2025spanning}, with the only difference being that here the lemma is stated for general hypergraph properties $\cG$ and $\cS$, whereas in \cite{illingworth2025spanning} $\cG$ is the family of $n$-vertex $k$-graphs $G$ with $\delta^*(G) \ge (1/2 + \eps)n$ and $\cS$ is the family of $s$-vertex $k$-graphs $S$ with $\delta^*(S) \ge (1/2 + \eps/2)s$. The more general statement given here follows from the proof of Lemma 2.1 in \cite{illingworth2025spanning}, by replacing the specific properties used there by $\cG$ and $\cS$ and using the assumption about the number of ways to extend a vertex in any $G \in \cG$ to a hypergraph in $\cS$; we elaborate a bit more on this in \Cref{appendix:proof-blow-up-tiling}. We remark that Lang and Sanhueza-Matamala \cite{lang2024hypergraph} provide a more general result (see \cite[Proposition 6.3]{lang2024hypergraph}), but with slightly stronger assumptions and a different conclusion. 

\begin{restatable}{lem}{lemBlowUpChain}
\label{lem:blow-up chain}
Let $1/n\ll 1/m_2\ll 1/m_1\ll 1/s,\gamma\ll \eps,1/k\le 1/2$. 
Let $\cG$ and $\cS$ be families of $n$-vertex and $s$-vertex $k$-graphs such that for every $G \in \cG$ and every $W \subseteq V(G)$ of size at most $2k$, there are at least $(1 - 1/s^2)\binom{n-|W|}{s-|W|}$ sets $U \subseteq V(G)$ of size $s$ that contain $W$, such that $G[U]$ is a copy of a graph in $\cS$.
Suppose that $G\in \cG$. Then there exists a sequence of $s$-vertex $k$-graphs $F_1,\dots, F_r$ and a sequence of subgraphs $F_1^*,\dots, F_r^*\subseteq G$ such that the following hold for all $i,j\in [r]$ with addition modulo $r$:

\begin{itemize}
    \item $F_i$ is a copy of a $k$-graph in $\cS$,
    \item $F_i^*$ is a $(\gamma,m_i^*)$-nearly-regular blow-up of $F_i$ for some $m_i^*\in[m_1,m_2]$,
    \item $V(F_1^*)\cup\dots\cup V(F_r^*)=V(G)$,
	\item $V(F_i^*)\cap V(F_j^*)=\emptyset$ unless $j \in \{i-1,i,i+1\}$,
    \item $V(F_i^*)\cap V(F_{i+1}^*)$ consists of exactly $k$ vertices that induce an edge in $F_i^*$ and $F_{i+1}^*$ that is disjoint from the singleton parts of both blow-ups (if they exist).
\end{itemize}
\end{restatable}

In order to use \Cref{lem:blow-up chain} in our setting, we will prove the following ``structural inheritance'' lemma.

\begin{lem}
\label{lem:concentration}
    Let $1/n\ll 1/s\ll \eps\ll \mu'\ll \mu\ll 1/k \le 1/3$ and let $\ell \in [k-1]$. Let $G\in\nonexE_{\eps,\mu} (n,k,\ell)$, let $W \subseteq V(G)$ be a subset of size at most $2k$, and let $S$ be a uniformly random $s$-set from $V(G)$ that contains $W$. Then, with probability at least $1 - 1/s^2$, the subgraph $G[S]$ is in $\nonexE_{2\eps,\mu'}(s,k,\ell)$.
\end{lem}

As explained before, roughly speaking, our plan is to find a spanning $\ell$-path in each blow-up in the structure given by \Cref{lem:blow-up chain}, and then join up these paths. The following lemma shows that, under a necessary divisibility condition, every such blow-up indeed contains a spanning $\ell$-path.

\begin{lem}
\label{lem:Hamilton paths in blow-ups}    
Let $1/m\ll 1/s\ll \gamma\ll \eps\ll \mu\ll 1/k\le 1/3$, let $\ell \in [k-1]$, and suppose $(k,\ell) \ne (3,1)$. Suppose $F\in \nonexE_{\eps,\mu}(s,k,\ell)$ and $F^*$ is a $(\gamma,m)$-nearly-regular blow-up of $F$ satisfying $|\Fs| \equiv k \pmod{k-\ell}$. Let $e_1,e_2\in E(F^*)$ be two ordered edges that are vertex-disjoint from each other and from the singleton part of $F^*$ (if it exists). Then there exists a spanning $\ell$-path in $F^*$ from $e_1$ to $e_2$.
\end{lem}

Assuming~\Cref{lem:blow-up chain,lem:concentration,lem:Hamilton paths in blow-ups}, it is fairly straightforward to prove~\Cref{thm:Hamilton cycle non-extremal}.

\begin{proof}[Proof of~\Cref{thm:Hamilton cycle non-extremal}]
	Let us set up the chain of parameters $1/n\ll 1/m_2\ll 1/m_1\ll 1/s\ll\gamma\ll \eps\ll \mu'\ll \mu \ll 1/k\le 1/3$. Let $\cG=\nonexE_{\eps,\mu}(n,k,\ell)$ and $\cS=\nonexE_{2\eps,\mu'}(s,k,\ell)$. Then, by \Cref{lem:concentration}, we can apply \Cref{lem:blow-up chain} to $G\in \cG$, and let $F_1,\dots, F_r\in \cS$ and $\Fs_1, \ldots, \Fs_r$ be the $k$-graphs so obtained, where each $\Fs_i$ is a $(\gamma,m_i^*)$-nearly-regular blow-up of $F_i$ for some $m_i^*\in[m_1,m_2]$. 
    
    For every $i\in[r]$, let $e_i = E(\Fs_i) \cap E(\Fs_{i+1})$, and fix any ordering $v_{i,1} \ldots v_{i,k}$ of $e_i$. In particular, we have $e_r=e_0=E(\Fs_r)\cap E(\Fs_1)$ (as we think of the indices $i\in[r]$ cyclically modulo $r$). We define integers $s_1, \ldots, s_r \in [k-\ell]$ and $t_1, \ldots, t_r \in  [\ell+1, k]$, as follows. Take $s_1 = 1$. Given $s_i$, let $t_i$ be the unique integer in $[\ell+1, k]$ such that 
	\begin{equation} \label{eqn:divisibility}
		|\Fs_i| - (s_i-1) - (k - t_i + 1) \equiv k \pmod{k-\ell}.
	\end{equation}
	Finally, set $s_{i+1} = t_i-\ell$. Define $\Fss_i := \Fs_i - \{v_{i-1,1}, \ldots, v_{i-1,s_i-1}, v_{i,t_i}, \ldots, v_{i,k}\}$. Since $e_{i-1}$ and $e_i$ are both disjoint of the singleton part of $\Fs_i$, it is immediate that $\Fss_i$ is a $(2\gamma,m_i^*)$-nearly-regular blow-up of $F_i$, using $s_i +k - t_i \le 2k \ll \gamma m_i^*$. Additionally, we remark that $\bigcup_{i\in[r]} V(\Fss_i)=V(G)$. Indeed, let us define $U_i=V(\Fs_i)\setminus V(\Fss_i)$. By choice, $U_i\subseteq V(e_{i-1})\bigsqcup V(e_i)$. But then
    \[
        U_i\cap V(e_i)=\{v_{i,t_i},\dots, v_{i,k} \}\subseteq \{v_{i,s_{i+1}},\dots,v_{i,k}\}\subseteq V(\Fss_{i+1}),
    \]
    and similarly $U_i\cap V(e_{i-1})\subseteq V(\Fss_{i-1})$. Hence, $\bigcup_{i\in[r]} V(\Fss_i)=\bigcup_{i\in[r]} V(\Fs_i)=V(G)$.
    
	Now, our goal is to find a spanning $\ell$-path, say $P_i$, in $\Fss_i$ from $S_i := v_{i-1,s_i} \ldots, v_{i-1,k}$ to $T_i := v_{i,1} \ldots v_{i,t_i-1}$. To this end, we first find ordered edges $e_{S_i}, e_{T_i}\in E(\Fss_i)$ such that $e_{S_i}$ starts with $S_i$, $e_{T_i}$ ends with $T_i$, and $e_{S_i}$ and $e_{T_i}$ are vertex-disjoint from each other and from the singleton part of $\Fss_i$, if it exists. Applying the supported co-degree condition $\delta^*(F_i)\ge (1-1/k-2\eps)s$ on $\phi(S_i)$ and $\phi(T_i)$ provides edges $e_{S_i}'$, and $e_{T_i}'$ in $F$ which satisfy the above requirements (with $\phi(S_i)$ and $\phi(T_i)$ in place of $S_i$ and $T_i$ respectively), and we let $e_{S_i}, e_{T_i}\in E(\Fss_i)$ be arbitrary blow-ups of $e_{S_i}'$, and $e_{T_i}'$ respectively that contains $S_i$ and $T_i$, respectively. Since $|\Fss| \equiv k \pmod{k-\ell}$ (see \eqref{eqn:divisibility}), \Cref{lem:Hamilton paths in blow-ups} then provides the desired path $P_i$. 

    The key point to note is that, for each $i\in [r]$, the sequence of the final $\ell$ vertices of $P_i$ exactly coincides with the sequence of the first $\ell$ vertices of $P_{i+1}$, namely $v_{i,t_i-\ell}\dots v_{i,t_i-1}$. Set $Q_i$ to be the vertex sequence obtained by removing the last $\ell$ vertices of $P_i$. Let $Q$ be the cyclic vertex sequence obtained by sequentially concatenating the sequences $Q_1\dots Q_r$. Then, by the observation above, we may conclude that the vertex sequence $Q$ forms an $\ell$-cycle in $G$. Since each $P_i$ spans $\Fss_i$ and $\bigcup_{i\in[r]} V(\Fss_i)=V(G)$, the $\ell$-cycle $Q$ spans $V(G)$, completing the proof.
\end{proof}

\paragraph{Organisation of the next five sections.}
Since \cref{lem:blow-up chain} is based on~\cite[Lemma 2.1]{illingworth2025spanning} with largely the same proof, we provide only a sketch and postpone it to \Cref{appendix:proof-blow-up-tiling}. 
We prove some useful facts about $\nonexE_{n,k,\ell}$ in \Cref{sec:basics non-extremal}, and use them to prove \cref{lem:concentration} in \Cref{sec:inheritance}. Then, as discussed in \Cref{sec:pf overview non-extremal}, we describe how we find a suitable fractional matching and an absorbing and connecting walk in \Cref{sec:Farkas,sec:absorbing walk}, respectively. We put it all together to prove \cref{lem:Hamilton paths in blow-ups} in \Cref{sec:proof Hamilton paths}.

\section{Basic tools}\label{sec:basics non-extremal}
	In this section we prove a few preliminary results that we will use later.
	Recall that $\nonexE_{\eps,\mu}(n,k,\ell)$ is the family of $n$-vertex graphs with minimum positive co-degree at least $(1-1/t-\eps)n$ where every $\floor{n/t}$ vertices span at least $\mu n^2$ supported pairs. To begin with, we prove an approximate alternate description for $\nonexE_{\eps,\mu}(n,k,\ell)$ that we use when proving \cref{lem:concentration}. The aim of \cref{lem:concentration} is to show that, roughly speaking, the properties of graphs in $\nonexE_{\eps,\mu}(n,k,\ell)$ are preserved under taking a random subset of vertices $S$. A naive approach to guaranteeing that the non-extremal properties are preserved takes a union bound over all subsets of $S$ of size $|S|/t$, which would not yield an effective bound. Instead, using also the minimum supported co-degree condition, we show that it suffices to guarantee that non-neighbourhoods of vertices with vertex co-degree close to $(1-1/t)n$ have many supported pairs. This is formalised in the following proposition.

	\begin{prop}
	\label{prop:approx non extremal}
		Let $k\ge 3$ and $\ell\in [k-1]$ be integers, $t=\floor{\frac{k}{k-\ell}}(k-\ell)$, and let $1/n\ll\eps\ll \mu'\ll \mu\ll 1/k$. 
		\begin{enumerate}[label = \rm(\roman*)]
			\item \label{itm:non-ext-1}
				Suppose $G\in \nonexE_{\eps,\mu}(n,k,\ell)$. Then for all $v\in V(G)$, either $d^1(v)\ge (1-1/t+\mu')n$ or $V(G)\setminus N^1(v)$ contains at least $\mu'n^2$ supported pairs.
			\item \label{itm:non-ext-2}
				Let $G$ be an $n$-vertex $k$-graph with no isolated vertices and $\delta^*(G) \ge (1 -1/t - \eps)n$. Suppose that for all $v\in V(G)$, either $d^1(v)\ge (1-1/t + 3\mu' t)n$ or $V(G)\setminus N^1(v)$ contains at least $\mu n^2$ supported pairs. Then every set of $\floor{n/t}$ vertices in $G$ contains at least $\mu' n^2$ supported pairs.
		\end{enumerate}
	\end{prop}

	\begin{proof}
		Item \ref{itm:non-ext-1} is fairly simple to prove. Suppose there exists some $v\in V(G)$ such that $d^1(v)\le (1-1/t+\mu')n$, and hence $|V(G)\setminus N_G^1(v)|\ge n/t-\mu' n$. Let $A$ be any set of the form $(V(G)\setminus N_G^1(v))\sqcup B$, where $B$ is an arbitrary set of distinct vertices from $N_G^1(v)$ such that $|B|$ is the smallest possible value required to make $|A|\ge \floor{n/t}$; then $|B| \le \mu'n$ (note that $B$ is potentially empty). Since $G\in \nonexE_{\eps,\mu}(n,k,\ell)$, the set $A$ contains at least $\mu n^2$ supported pairs. As $|B|\le \mu'n$, at most $\mu'n^2$ of these supported pairs have a vertex in $B$, and thus at least $(\mu-\mu') n^2\ge \mu' n^2$ lie entirely within $V(G)\setminus N_G^1(v)$.

		Now for \ref{itm:non-ext-2}. Suppose there exists some $U\subseteq V(G)$ with $|U|\ge \floor{n/t} \ge n/t -1$ that contains less than $\mu' n^2$ supported pairs. Then, by the handshaking lemma on the graph of supported pairs in $U$, there exists some $v\in U$ such that $d_U^1(v)\le 3\mu't n$. Since $|V(G)\setminus U|\le n- n/t$, we see that $d_G^1(v)\le (1-1/t+3\mu' t)n$, and hence, by assumption on $G$, the set $T=V(G)\setminus N_G^1(v)$ contains at least $\mu n^2$ supported pairs. Since $G$ has no isolated vertices and $\delta^*(G) \ge (1 - 1/t - \eps)n$, we know that $d_G^1(v)\ge (1-1/t-\eps)n$, and hence $|T|\le n/t+\eps n$. By our choice of $v$, we have $|T\cap U|\ge |U|-3\mu't n\ge n/t-4\mu' tn$, and consequently $|T\setminus U|\le n/t+\eps n-(n/t-4\mu' tn)\le 8\mu'tn$. Hence, the number of supported pairs in $T$ which contain at least one vertex outside $U$ is at most $8\mu'tn^2$. 
		Consequently, $T\cap U\subseteq U$ contains at least $(\mu-8\mu't)n^2\ge \mu' n^2$ supported pairs, a contradiction to our assumption on $U$.
	\end{proof}

	In a similar vein, but somewhat in the ``opposite direction'', we also show that taking a blow-up approximately preserves non-extremal structure.
	Here we use the alternative characterisation of non-extremal graphs given by the previous proposition.

	\begin{prop}
	\label{prop:blowup is nonextremal}
		Let $0 \le \gamma\ll \eps\ll \mu'\ll \mu\ll 1/k\le 1/3$, let $m, s \ge 1$ be integers, and $t=\floor{\frac{k}{k-\ell}}(k-\ell)$. Suppose $F\in \nonexE_{\eps,\mu}(s,k,\ell)$ and $F^*$ is a $(\gamma,m)$-regular blow-up of $F$. Then $\Fs\in \nonexE_{2\eps,\mu'}(|F^*|,k,\ell)$.
	\end{prop}
	\begin{proof}
		Set $|\Fs|=N$. By the definition of $F^*$, it is straightforward to see that $(1-\gamma)ms \le N\le (1+\gamma)ms$.
		First, notice that since $F$ has no isolated vertices, neither does $\Fs$.
		Second, since $\delta^*(F)\ge (1-1/t-\eps)s$ and $F^*$ is a $(\gamma,m)$-regular blow-up of $F$, we see that 
		\begin{align*}
			\delta^*(F^*)
			\ge\left(1-\frac{1}{t}-\eps \right)(1-\gamma)ms
			\ge\left(1-\frac{1}{t}-2\eps \right)(1+\gamma)ms
			\ge\left(1-\frac{1}{t}-2\eps \right)N,
		\end{align*}
		using $\gamma\ll \eps$.

		Finally, we need to show that any subset of $V(F^*)$ of size at least $N/t$ contains at least $\mu' N^2$ supported pairs. We do this via \Cref{prop:approx non extremal}. Fix $\eta$ such that $\mu' \ll \eta \ll \mu$. Let $w\in V(F^*)$ be arbitrary, and let $\phi(w)=v\in V(F)$ (recall that $\phi : V(\Fs) \to V(F)$ maps all copies in $\Fs$ of a vertex $v$ in $F$ to $v$). Since $F\in \nonexE_{\eps,\mu}(s,k,\ell)$, by \Cref{prop:approx non extremal}~\ref{itm:non-ext-1}, we know that either $\ds(v)\ge (1-1/t+\eta)s$ or $V(F)\setminus N^1(v)$ contains at least $\eta s^2$ supported pairs.
        If the former is true, then since every vertex neighbour of $v$ is blown up to at least $(1-\gamma)m$ vertices in $F^*$, we have
        \[
            \ds_{F^*}(w)\ge \left(1-\frac{1}{t} +\eta \right)(1-\gamma)ms 
			\ge \left(1 - \frac{1}{t} + 3\mu' t\right)(1 + \gamma)ms
			\ge \left(1-\frac{1}{t} + 3\mu't \right)N,
        \]
        where we use $N\le (1+\gamma)ms$ and $\mu'\ll \eta \ll 1/t$.
        In the other case, we conclude that the number of supported pairs in $V(F^*)\setminus N^1(w)$ is at least
        \[
            \eta s^2 \cdot (1-\gamma)^2 m^2 \ge \frac{\eta^2}{4} (1+\gamma)^2 m^2  s^2 \ge \frac{\eta^2}{4} N^2.
        \]
		Since $\mu' \ll \eta^2$, \Cref{prop:approx non extremal}~\ref{itm:non-ext-2} implies that every set of at least $N/t$ vertices contains at least $\mu' N^2$ supported pairs, as desired.
    \end{proof}

	Next, we also present a simple proposition that helps provide a minimum degree condition based on the minimum supported co-degree. This is Proposition~$1.4$ of~\cite{mycroft2025positive}. We include the proof here for completeness. 
	\begin{prop}
	\label{prop:codeg to deg}    
		Let $G$ be a $k$-graph. Then every supported set $S$ in $G$ with $|S|\le k-1$ satisfies
		\[
		d(S)\ge\frac{\delta^*(G)^{k-|S|}}{(k-|S|)!}.
		\]
	\end{prop}
\begin{proof}
    By definition, for every supported set $S$ of size at most $k-1$, there are at least $\delta^*(G)$ vertices $v$ such that $v\in N^1(S)$. By repeatedly applying this observation, we see that there must be at least $\delta^*(G)^{k-|S|}$ ordered sequences of vertices $(v_1,\dots, v_{k-|S|})$ such that $S\cup\{v_1,\dots, v_{k-|S|}\}$ is an edge of $G$. Hence, we have
    \[
    d(S)\ge \frac{\delta^*(G)^{k-|S|}}{(k-|S|)!}. \qedhere 
    \]
\end{proof}
\section{Structural inheritance}\label{sec:inheritance}
We now prove our structural inheritance lemma. As noted above, the crucial observation here is the alternative characterisation of non-extremal graphs given by \Cref{prop:approx non extremal}~\ref{itm:non-ext-2}.

\begin{proof}[Proof of \Cref{lem:concentration}]
    Write $W = \{w_1, \dots, w_r\}$, where $r = |W| \le 2k$. Let $y_1, \dots, y_{s-r}$ be $s - r$ vertices of $V(G)$ picked independently and uniformly at random (with repetition). Set $T = \{x_1, \dots, x_s\}$, where $x_i = w_i$ for $i\le r$ and $x_i = y_{i - r}$ for $i>r$. Let $T_0=T\setminus W = \{x_{r+1}, \dots, x_s\}$ and recall that $S$ is a randomly uniform set of $s$ vertices in $V(G)$ that contains $W$. Let $\cA$ be the event that ${x_1, \dots, x_s}$ are distinct. Let $\cB_1$ and $\cB_2$ be the events that $G[S]\in \nonexE_{2\eps,\mu'}(s,k,\ell)$ and $G[T]\in \nonexE_{2\eps,\mu'}(s,k,\ell)$ respectively. Then, we first see that
    \begin{align*}
		\bP[\comp{\cA}] \le \frac{s^2}{n}=o(1).
	\end{align*}
    Consequently, we have
    \begin{align*}
        \bP[\comp\cB_1]&=\bP[\comp\cB_2\mid \cA]
        =\frac{\bP[\comp\cB_2\cap \cA]}{\bP[\cA]}
        \le \bP[\comp\cB_2](1+o(1)).
    \end{align*}
    Hence, it will suffice to show that $\bP[\comp\cB_2]\le \frac{1}{2s^2}$.

    Let $H=G[T]$ and fix parameters $\mu_1,\mu_2$ such that $\mu'\ll \mu_1\ll \mu_2\ll \mu$. We define the following two events.
    \begin{itemize} 
		\item 
			$\cC_1=\{\frac{d^1_H(U)}{s}\ge \frac{d^1_G(U)}{n}-\eps \textnormal{ for every } U\subseteq T \text{ with } |U|\le k-1 \text{ which is supported in $G$}\}$,
		\item $\cC_2 = \{\text{if $x \in T$ satisfies $d_G^1(x) \le (1 - 1/t + \mu_2)n$ then $T \setminus N_H^1(x)$ has at least }\mu_1 s^2 \ \text{supported} \allowbreak \text{pairs}\}$.
    \end{itemize}
    First, we will argue that $\cC_1\cap \cC_2$ implies $\cB_2$. It will subsequently suffice to show that each $\cC_i$ occurs with high probability, and then use a union bound. 

    Indeed, from $\cC_1$ and from $G \in \nonexE_{\eps, \mu}(n,k,\ell)$, we know that $H$ has no isolated vertices and that $\delta^*(H)\ge (1-1/t-2\eps)s$. Now, consider any $x\in V(H)$ such that $d_H^1(x)\le (1-1/t+3\mu't)s$. From $\cC_1$, we see that $d_G^1(x)\le (1 - 1/t + 3\mu't + \eps)n\le (1 - 1/t + \mu_2)n$, and from $\cC_2$ we conclude that $V(H)\setminus N^1_H(x)$ contains at least $\mu_1 s^2$ supported pairs. Then by \Cref{prop:approx non extremal}~\ref{itm:non-ext-2}, we know that every $\floor{s/t}$-subset in $V(H)$ has at least $\mu' s^2$ supported pairs. Thus, overall, we have $H\in \nonexE_{2\eps,\mu'}(s,k,\ell)$.

    Let us start with $\cC_1$. Let $I\subseteq [s]$ be any subset such that $|I|\in[k-1]$. Define the random variables $X_I=\{x_i:i\in I\}$ and $d_I = |N^1_H(X_I) \cap T_0|$. Consider the event 
	\begin{equation*}
		\cE_I=\{X_I \textnormal{ is supported in $G$ and } d_H^1(X_I)/s\le d_G^1(X_I)/n-\eps\}. 
	\end{equation*}
	Let $\cU_I=\{U\subseteq V(G):|U|\le |I|, \bP[X_I=U]>0\}$. 
	Then, observe that
    \begin{align*}
        \bP[\cE_I]&= \sum_{\substack{U\subseteq V(G) \\ |U| \le |I|}} \bP[\cE_I\cap \{X_I=U\}]\\
        &=\sum_{U\in \cU_I} \bP[\cE_I\mid X_I=U]\cdot\bP[X_I=U].
    \end{align*}
	In particular, to show that $\bP[\cE_I]\le e^{-\Omega(s)}$, we may consider any arbitrary $U\in \cU_I$, condition on $X_I=U$, and argue that $\bP[\cE_I\mid X_I=U]\le e^{-\Omega(s)}$ for any such choice of $U$.

    If $U$ is not supported in $G$, then $\bP[\cE_I\mid X_I=U]=0$, and hence we may assume that it is supported. We may lower bound the expected value of $d_I$ as
    \begin{align*}
        \bE d_I  &= 
		d^1_G(U) \cdot \left(1 - \left(1 - \frac{1}{n}\right)^{s - |U \cup W|}\right) \\
        &\ge\frac{(s - 3k) \cdot \delta^*(G)}{2n} \\
        &\ge \frac{s-3k}{2}\cdot\left(1-\frac{1}{t}-\eps\right)\ge \frac{s}{4},
    \end{align*}
	using $t \ge 3$ and $k \ll s$. 
	Notice that changing the outcome of any single $x_i$ changes $d_I$ by at most $1$, and so by McDiarmid's inequality (\Cref{lem:McDiarmid}),
    \[
        \bP\left[d_I-\bE d_I\le -\frac{\eps \bE d_I}{2}\right]\le 
		2\exp\left(\frac{-2\eps^2 (\bE d_I)^2}{s}\right)
		= e^{-\Omega(s)}.
    \]
    Thus, since $d_H(X_I) \ge d_I$, with probability at least $1-e^{-\Omega(s)}$, we have
    \[
    \frac{d_H(X_I)}{s} \ge \frac{d_I}{s}\ge \frac{1-\eps/2}{s} \cdot (s-r) \cdot \frac{d}{n} \ge 
	\left(1 - \frac{\eps}{2}\right)\left(1 - \frac{k}{s}\right) \cdot \frac{d}{n} \ge
	\frac{d}{n}-\eps,
    \]
	where we use that $k/s\ll\eps$.
    Hence, as argued previously, we now have $\bP[\cE_I]\le e^{-\Omega(s)}$ for every $I\subseteq [s]$ with $|I|\in[k-1]$. Taking a union bound over all possibilities of $I$, we see that $\bP[\cC_1]\ge 1-(k-1)\binom{s}{k-1}e^{-\Omega(s)} \ge 1 - \frac{1}{4s^2}$. 

    Next, we will show that $\cC_2$ occurs with high probability. For each $i\in [s]$, define the event 
	\begin{equation*}
		\cF_i=\{T\setminus N^1_H(x_i) \textnormal{ has less than } \mu_1s^2 \textnormal{ supported pairs and } d_G^1(x_i)< (1-1/t+\mu_2)n\}. 
	\end{equation*}
    Fix any integer $i \in [r+1, s]$. Similar to our previous calculations, we note that
    \[
    \bP[\cF_i]=\sum_{u\in V(G)}\bP[\cF_i\cap\{x_i=u\}]= \sum_{u\in V(G)} \bP[\cF_i\mid x_i=u]\cdot \bP[x_i=u].
    \]
    Consider any arbitrary $u\in V(G)$. We will argue that, for each fixed $i \in [r+1, s]$, we have $\bP[\cF_i\mid x_i=u]\le e^{-\Omega(s)}$, and then the above equation then guarantees that $\bP[\cF_i]\le e^{-\Omega(s)}$. We comment at the end of the proof on how to handle $\cF_i$ for $i \in [r]$. 

    If $d_G^1(u)\ge (1-1/t+\mu_2)n$, then $\bP[\cF_i \mid x_i = u]=0$, and we are done.
	Hence, we may suppose $d_G^1(u)\le (1-1/t+\mu_2)n$. Since $G\in \nonexE_{\eps,\mu}(n,k,\ell)$, by \Cref{prop:approx non extremal}~\ref{itm:non-ext-1}, there must be at least $\mu_2 n^2$ supported pairs in $V(G)\setminus N_G^1(u)$. Let us call this set of supported pairs $P$. From \Cref{prop:codeg to deg}, any supported pair $w_1 w_2\in P$ satisfies
    \[
        d_G(w_1 w_2)\ge \frac{\delta^*(G)^{k-2}}{(k-2)!} \ge \left(\frac{\delta^*(G)}{k-2}\right)^{k-2}\ge \left(\frac{n-n/t-\eps n}{k-2} \right)^{k-2}\ge \left(\frac{n}{3k}\right)^{k-2}
    \]
    since $t\ge 2$. Let $E_P$ be the set of all edges of $G$ containing at least one pair of $P$. We can bound the number of such edges from below as
    \begin{align*}
        |E_P|\ge \frac{1}{\binom{k}{2}}\sum_{w_1 w_2\in P} d_G(w_1 w_2)\ge \frac{2\mu_2}{(3k)^{k-2}\cdot k^2}\cdot n^k
		\ge \frac{2\mu_2}{(3k)^k}\cdot n^k.
    \end{align*}
	For each edge $e \in E_P$, the probability that $e \in E(H[T_0 \setminus \{x_i\})$ is 
	\begin{align*}
		&\bP\left[\bigcup_{J \subseteq [r+1,s] \setminus \{i\}, \, |J| = k} \left\{\{x_j : j \in J\} = e\right\}\right] \\ 
		\ge &\sum_{\substack{J \subseteq [r+1,s] \setminus \{i\}, \\ |J| = k}} \bP[\{x_j : j \in J\} = e] - \sum_{\substack{J,J' \subseteq [r+1,s] \setminus \{i\}, \\ |J|=|J'|=k, J \neq J'}} \bP[\{x_j : j \in J\} = \{x_{j'} : j' \in J'\} = e] \\
		  \ge &\frac{(s-r-1)(s-r-2) \ldots (s-r-k)}{n^k} - \frac{\big((s-r-1)(s-r-2) \ldots (s-r-k)\big)^2}{n^{k+1}} \ge \frac{s^k}{2n^k},
	\end{align*}
	where we used that if $J, J'$ are two distinct sets of size $k$ then $|J \cup J'| \ge k+1$.
	We thus obtain the following lower bound on $Z := |E_P \cap E(H[T_0 \setminus \{x_i\}])|$.
    \begin{align*}
        \bE Z&
		\ge \sum_{e \in E_P} \bP\big[e \in E(H[T_0 \setminus \{x_i\}])\big] 
		\ge \frac{\mu_2}{(3k)^k} \cdot s^k 
		= \beta \mu_2 s^k. 
    \end{align*}
    where $\beta=1/(3k)^k$ is a function of $k$. Next, observe that if a single vertex in $T_0 \setminus \{x_i\}$ is replaced by an arbitrary vertex of $V(G)$, then the number of edges of $E_P$ in $H[T_0 \setminus \{x_i\}]$ changes by at most $s^{k-1}$ since this bounds the maximum degree of a vertex in $H$. Using McDiarmid's inequality (\Cref{lem:McDiarmid}), we have
    \[
    \bP[Z-\bE Z\le -\frac{1}{2} \cdot \bE Z]\le 2\exp\left(\frac{- \beta^2 \mu_2^2 s^{2k}}{2 \cdot s\cdot s^{2k-2}}\right)\le \exp(-\Omega(s)).
    \]
    It readily follows that $Z\ge \beta\mu_2 s^k/2$ with probability at least $1-e^{-\Omega(s)}$. 
	Each edge of $E_P$ counted in $Z$ contains at least one pair of $P$, and each pair of $P$ is counted at most $s^{k-2}$ times. Hence, with probability at least $1-e^{-\Omega(s)}$, there are at least $\beta\mu_2 s^2/2\ge \mu_1s^2$ supported pairs of $P$ in $H$. By choice, all of these lie outside $N_H^1(x_i)=N_H^1(u)$.

    Thus, as noted initially, we conclude that $\bP[\cF_i]\le e^{-\Omega(s)}$ for all $i\in[r+1, s]$. Finally, if $i \in [r]$, we have $\bP[\cF_i \cap \{x_i = u\}]=0$ for all $u\ne w_i$, which means that $\bP[\cF_i] =\bP[\cF_i \mid x_i = w_i]$, and hence we can perform the same calculations as above with $u = w_i$. Additionally, here we will instead have $Z = E_P\cap E(H[T_0])$, but the bounds are identical anyway. 
	Hence, by taking a union bound over all $i \in [s]$, we conclude that $\bP[\cC_2]\ge 1-se^{-\Omega(s)} \ge 1- \frac{1}{4s^2}$.

    Since $\cC_1\cap \cC_2$ implies $\cB_2$, we use a union bound and see that
    \[
        \bP[\comp{\cB_2}]\le \bP[\comp{\cC_1}\cup \comp{\cC_2}]\le \frac{1}{2s^2},
    \]
     and thus $\bP[\cB_2]\ge 1-1/2s^2$, as required.
\end{proof}

\section{Perfect weighted fractional matchings}\label{sec:Farkas}
Our overall approach will rely on finding a large, nearly perfect matching in a regular blow-up $F^*$ of some $F \in \nonexE_{\eps,\mu}(n,k,\ell)$, and finding a way to connect the edges of this matching via an $\ell$-path. Towards that end, we make use of a recently introduced idea of Mycroft and Z{\'a}rate-Guer{\'e}n~\cite{mycroft2025positive} that they call \emph{perfect weighted fractional matchings}, and we borrow their definitions and notation. 

To begin with, we define a \emph{fractional matching} in a $k$-graph $H$ to be a weight function that maps the edges of $H$ to the unit interval, say $q:E(H)\to [0,1]$, such that for every vertex $v\in V(H)$ the sum of $q(e)$ for all edges $e$ containing $v$ is at most $1$, that is,
\[
\sum\limits_{\substack{e\in E(H)\\ e\ni v}} q(e)\le 1.
\]
We say that $q$ is a \emph{perfect fractional matching} if this sum is precisely $1$ for every vertex $v\in V(H)$. Note that a (perfect) matching is a special type of a (perfect) fractional matching where the edge weights are restricted to the set $\{0,1\}$. 

Now, suppose we have $k$ weights $w_1,\dots,w_k\in [0,1]$. Then a \emph{$(w_1,\dots,w_k)$-fractional matching} is similar to a fractional matching $q$, except that any edge $e$ places weight $w_iq(e)$ on its $i$th vertex. Of course, we now need to consider edges with a pre-defined vertex ordering. Hence, instead, we will simply consider all possible permutations of the vertices of each edge.

Recall from \Cref{sec:notation} that, for a given a $k$-graph $H$, we let $E^*(H)$ denote the set of ordered edges of $H$. For every vertex $v\in V(H)$ and $i\in [k]$, we write $E_i^v(H)\subseteq E^*(H)$ for the set of all ordered edges with $v$ as the $i$th vertex. Fix $\bfw=(w_1,\dots,w_k)\in \bR^k_{\ge 0}$. A \emph{$\bfw$-weighted fractional matching} is a weight function $q:E^*(H)\to\bR$ such that $q(e)\ge 0$ for all $e\in E^*(H)$, and 
\[
\sum\limits_{i\in [k]}\sum\limits_{e\in E_i^v(H)} w_i\cdot q(e)\le 1
\]
for every vertex $v\in V(H)$. We call it \emph{perfect} if we have equality in the above inequality for all $v\in V(H)$. Observe that a (perfect) fractional matching is simply a (perfect) $(1,\dots,1)$-weighted fractional matching. Following~\cite{mycroft2025positive}, we will focus on perfect weighted fractional matchings with weights $\bfw^*=(w_1,\dots, w_k)$, defined by 
\begin{equation*}
	w_i = \left\{	
		\begin{array}{ll}
			k-1 & i = 1 \\
			t-1 & 2 \le i \le k,
		\end{array}
		\right.
\end{equation*}
where 
\begin{equation*}
	t := t(k,\ell) = \floor{\frac{k}{k-\ell}}(k-\ell).
\end{equation*}
The reason for this particular choice of weights is that any complete $k$-uniform $k$-partite hypergraph with parts $A, B_1, \ldots, B_{k-1}$, such that $|A|/(k-1) = |B_i|/(t-1)$ for every $i \in [k-1]$ (and an additional divisibility condition holds), has a spanning $\ell$-cycle (see \Cref{prop:ell cycle in k-partite k-graph} and Proposition 1.8 in \cite{mycroft2025positive}). This means that if an ordered edge $e = v_1 \ldots v_k$ receives some weight $q(e)$, then there is an $\ell$-cycle that we can use cover almost all the vertices in the blow-up of $e$ where $v_i$ is blown-up by $w_i \floor{q(e)}$ vertices, for $i \in [k]$, a key step in our proof. 

A convenient and often elegant method to find perfect fractional matchings is using Farkas' lemma, a linear algebraic result that is quite handy for this purpose -- see~\cite{keevash2015geometric,rodl2006perfect,lo2024towards} for examples. In order to state the lemma, we introduce some simple notation. Given a vector $\mathbf{v}=(v_1,\dots,v_n)\in\bR^n$ and any $a\in \bR$, we write $\mathbf{v}\ge a$ to mean that $v_i\ge a$ for all $i\in[n]$. We use analogous notation for $\mathbf{v}\le a$, $\mathbf{v}> a$ and $\mathbf{v}< a$. Finally, we use $\mathbf{1}$ to denote the all ones vector.

\begin{lem}[Farkas' lemma]
\label{lem:farkas}    
Let $\mathbf{A}\in \mathbb{R}^{m\times n}$ and $\mathbf{b}\in \mathbb{R}^m$. Then either there exists $\bfx\in\mathbb{R}^n$ such that $\mathbf{Ax}=\mathbf{b}$ and $\bfx\geq 0$, or there exists $\bfy\in \mathbb{R}^m$ such that $\mathbf{A}^\top\bfy\geq 0$ and $\mathbf{b}^\top\bfy<0$.
\end{lem}

Using this, we can show that non-extremal $k$-graphs have a $w^*$-weighted perfect fractional matching.

\begin{lem}
    \label{lem:wt perfect frac matching}
    Let $k\ge 3$ and $\ell\in [k-1]$ be integers, $t=\floor{\frac{k}{k-\ell}}(k-\ell)$, and let $1/n\ll\eps\ll 1/k$, where $n$ is divisible by $t$. Let $\bfw^*=(w_1,\dots,w_k)$ where $w_1=k-1$ and $w_i=t-1$ for all $2\le i\le k$. Then every $H\in \nonexE_{\eps,4\eps}(n,k,\ell)$ has a $\bfw^*$-weighted perfect fractional matching.
\end{lem}

\begin{proof}
	Let $V(H)=\{v_1,\dots,v_n\}$. For each edge $e\in E(H)$ and each $v_i\in e$, define $\chi_i(e)\in\bR^n$ to be the vector with $i$th coordinate $k-1$, $j$th coordinate $t-1$ for all $j$ such that $v_j\in e\setminus \{v_i\}$ and all other coordinates $0$. Let $\cX=\{\chi_i(e):e\in E(H), v_i\in e\}$.

	Suppose $H$ does not have a $\bfw^*$-weighted perfect fractional matching. This means that there is no set of weights $\{q(\bfx):q(\bfx)\ge 0,\bfx\in \cX\}$ which satisfies $\sum_{\bfx\in\cX} q(\bfx)\cdot \bfx=\mathbf{1}$, for otherwise, defining for each edge $e = \{v_{i_1}, \ldots, v_{i_k}\}$ and $j \in [k]$, 
	\begin{equation*}
		q'(v_{i_j} v_{i_1} \ldots v_{i_{j-1}} v_{i_{j+1}} \ldots v_{i_k}) = q(\chi_{i_j}(e)),
	\end{equation*}
	(and $q'(e) = 0$ for all other ordered edges corresponding to $e$), would constitute a $\bfw^*$-weighted perfect fractional matching. 
	Hence, by~\Cref{lem:farkas}, there exists $\bfy\in \bR^n$ such that $\bfy\cdot\mathbf{1}<0$ and $\bfy\cdot\bfx\ge 0$ for every $\bfx\in \cX$.

	Let $\bfy=(y_1,\dots,y_n)$. Then $\bfy\cdot\mathbf{1}=\sum y_i<0$. By potentially reordering the vertices, we may assume without loss of generality that $y_i\le y_j$ for all $i\le j$. Let $r$ be the smallest index such that $v_r\in N^1(v_1)$, that is, $v_1v_r$ is a supported pair. Since $H\in \nonexE_{\eps,4\eps}(n,k,\ell)$, we know that $\delta^*(H)\ge (1-1/t-\eps)n$, and hence $r\le n/t+\eps n+1 \le n/t + 2\eps n$.  

    Overall, we hope to obtain a contradiction by finding a ``low-weight'' edge of $H$, say $e=v_{i_1}\dots v_{i_k}$, which satisfies the following two key properties.
    \begin{itemize}
        \item $\frac{n}{t}\cdot\frac{t-2}{(t-1)(k-2)} \cdot ((k-1)y_{i_1}+(t-1)y_{i_2})\le y_1+\dots+ y_{2n/t}$, and
        \item $i_j\le 2n/t+1$ for all $j\ge 3$.
    \end{itemize}
    Indeed, given any such edge $e$, we let $\bfx\in \bR^n$ be the vector with $x_{i_1}=k-1$, $x_{i_j}=t-1$ for all $2\le j\le k-1$, and zero otherwise. Then clearly $\bfx=\chi_{i_1}(e) \in \cX$. Consequently, we obtain a contradiction as
    \begin{align*}
         0&\le \frac{n}{t}\cdot\frac{t-2}{(t-1)(k-2)}\cdot \bfy\cdot\bfx\\
          &= \frac{n}{t}\cdot\frac{t-2}{(t-1)(k-2)}\left((k-1)y_{i_1}+\sum\limits_{j\ge 2}(t-1)y_{i_j}  \right)\\
           &= \frac{n}{t}\cdot\frac{t-2}{(t-1)(k-2)}\cdot\big((k-1)y_{i_1}+(t-1)y_{i_2}\big) +\frac{n}{t} \cdot\left(\sum\limits_{j\ge 3}\frac{t-2}{k-2} \cdot y_{i_j}  \right)\\
          &\le \left(y_1+\dots +y_{\frac{2n}{t}}\right)+\frac{n}{t}\cdot (t-2) \cdot y_{\frac{2n}{t}+1} \\
          &\le\left(y_1+\dots +y_{\frac{2n}{t}}\right)+\left(y_{\frac{2n}{t}+1}+\dots +y_n\right)\\
          &<0,
    \end{align*}
    where the penultimate inequality uses $2\le t\le k$ and the first and final ones follow from the assumptions on $\bfy$.

    Let us try to find such an $e$. First, suppose $r\le n/t$. Set $i_1=1$ and $i_2=r$. By repeatedly using the minimum supported co-degree condition, we can find vertices $v_{i_3},\dots, v_{i_k}$ such that $v_{i_j}\in N^1(v_{i_1},\dots,v_{i_{j-1}})$ and $i_j\le n/t+\eps n+1$ for all $j\ge 3$. Observe that $e=v_{i_1}\dots v_{i_k}\in E(G)$. Furthermore, we have
    \begin{align*}
		\frac{n}{t}\cdot\frac{t-2}{(t-1)(k-2)} \cdot ((k-1)y_{i_1}+(t-1)y_{i_2})
		&= \frac{n}{t}\left(\frac{(t-2)(k-1)}{(t-1)(k-2)} \cdot y_1+\frac{t-2}{k-2} \cdot y_{\frac{n}{t}+1}\right)\\
		&\le \frac{n}{t}\left(y_1+y_{\frac{n}{t}+1}\right)\\
		&\le \left(y_1+\dots +y_{\frac{n}{t}}\right)+\left(y_{\frac{n}{t}+1}+\dots+y_{\frac{2n}{t}} \right),
    \end{align*}
    where we use $2\le t\le k$. Thus, $e$ meets the required criteria.

    Hence, we may assume that $n/t+1\le r\le n/t+2\eps n$. Let $A=\{v_1, \ldots, v_r\}$ and $B=\{v_{r+1},\dots ,v_{2n/t}\}$. Since $H\in \nonexE_{\eps,4\eps}(n,k,\ell)$ and $|A|=r\ge n/t$, there are at least $4\eps n^2$ supported pairs in $A$. This implies that we can greedily find a matching $M_1$ of supported pairs in $A$ that contains $v_1v_r$ and which spans at least $4\eps n$ vertices. 
	
    Now, we note that $|B|=2n/t-r$ and $|A\setminus V(M_1)|\le r-4\eps n$. 
	Since $r\le n/t+2\eps n$, we see that $|B|\ge |A\setminus V(M_1)|$. Hence, we can pair up each vertex of $|A\setminus V(M_1)|$ with a distinct vertex of $B$ arbitrarily (note that these vertex pairs need not be supported in $H$). Add this collection of pairs $M_1$ and call the resulting set of vertex pairs $M_2$. 

    For our final step, we consider all vertices in $C=\{v_1,\dots,v_{2n/t}\}\setminus V(M_2)$. By construction every vertex in $\{v_1,\dots,v_r\}$ is part of some pair in $M_2$. Thus, we may deduce that $C\subseteq \{v_{r+1},\dots, v_{2n/t}\}=B$. Since $|V(M_2)|$ is even, the same must be true of $|C|=2n/t-|V(M_2)|$. Hence, we can arbitrarily pair up all vertices in $C$. We add this collection of vertex pairs to $M_2$, and let the resulting set be $M_3$. Since every vertex in $v_1,\dots v_{2n/t}$ is contained in exactly one pair of $M_3$, we have $|M_3|=n/t$. 

    Let $v_{i_1}v_{i_2}$ to be a supported pair in $M_1$ such that $(k-1)y_{i_1}+(t-1)y_{i_2}\le (k-1)y_i+(t-1)y_j$ for all $v_iv_j\in M_1$ with $i<j$. Since $v_1v_r\in M_1$ and $i_1\ge 1$, we have $i_2\le r$. By construction, for any pair $v_iv_j\in M_3\setminus M_1$ with $i<j$, we have $v_j\in B$, and hence $j > r$. This implies
    \[
    (k-1)y_{i_1}+(t-1)y_{i_2}\le (k-1)y_1+(t-1)y_r\le (k-1)y_i+(t-1)y_j
    \]
    for all $v_iv_j\in M_3$, where the first inequality uses that $v_1v_r\in M_1$ and the definition of $v_{i_1}v_{i_2}$.  

    Since $v_{i_1}v_{i_2}$ is a supported pair, by repeatedly using the minimum supported co-degree condition, we can find an edge $v_{i_1}v_{i_2}\dots v_{i_k}\in E(H)$ such that $v_{i_j}\in N^1(v_{i_1},\dots,v_{i_{j-1}})$ and $i_j\le n/t+\eps n+1\le 2n/t+1$ for all $j\ge 2$. Additionally, we have
	\begin{align*}
		\frac{n}{t}\cdot\frac{t-2}{(t-1)(k-2)} \cdot ((k-1)y_{i_1}+(t-1)y_{i_2})
		&= \frac{t-2}{(t-1)(k-2)}\cdot |M_3| \cdot \left((k-1)y_{i_1}+(t-1)y_{i_2}\right) \\
		&\le \frac{t-2}{(t-1)(k-2)}\sum\limits_{v_iv_j\in M_3} \left((k-1)y_i+(t-1)y_j\right)\\
		&\le \sum\limits_{v_iv_j\in M_3} \left(\frac{(t-2)(k-1)}{(t-1)(k-2)} y_i+\frac{t-2}{k-2}y_{j}\right)\\
		&\le \sum\limits_{v_iv_j\in M_3} (y_i+y_j)\\
		&=y_1+\dots +y_{\frac{2n}{t}},
	\end{align*}
    where we have used $2\le t\le k$ as before. Since $e$ satisfies the desired conditions, the proof is complete.
\end{proof}

Notice that in the previous lemma we assume (for convenience) that $n$ is divisible by $t$. One can remove the divisibility assumption by taking an appropriate blow-up, as follows.

\begin{cor}
    \label{cor:wt perfect frac matching}
    Let $k\ge 3$ and $\ell\in [k-1]$ be integers, $t=\floor{\frac{k}{k-\ell}}(k-\ell)$, and let $1/n\ll\eps \ll \mu \ll 1/k$. Let $\bfw^*=(w_1,\dots,w_k)$ where $w_1=k-1$ and $w_i=t-1$ for all $2\le i\le k$. Then every $H\in \nonexE_{\eps,\mu}(n,k,\ell)$ has a $\bfw^*$-weighted perfect fractional matching.
\end{cor}

\begin{proof}
	Let $H^*$ be the $t$-blow-up of $H$, namely were each vertex is replaced by an independent set of size $t$.
	Then, clearly, $|H^*|$ is divisible by $t$. Moreover, by \Cref{prop:blowup is nonextremal} and because $H^*$ is a $(0,t)$-blow-up of $H$, we have that $H^* \in \nonexE_{\eps, 4\eps}(nt,k,\ell)$.
	Thus, by \Cref{lem:wt perfect frac matching} the graph $H^*$ has a $\bfw^*$-weighted perfect fractional matching $q$.
	This can be translated to a $\bfw^*$-weighted perfect fractional matching $q'$ in $H$ as follows. For any ordered edge $e \in E^*(H)$, let $F_e \subseteq E^*(H^*)$ be the collection of all ordered edges of $H^*$ that are blow-ups of $e$. Define 
    \[
        q'(e) = \frac{1}{t} \sum_{e^* \in F_e} q(e^*).
    \]
    It is fairly straightforward to verify that $q'$ is a $\bfw^*$-weighted perfect fractional matching in $H$.
\end{proof}

\section{Finding an absorbing and connecting walk}\label{sec:absorbing walk}
We first recall the definitions of $\ell$-walks and ordered supported sets from \Cref{sec:notation}. We begin this section by introducing some convenient notation and terminology.
Given a tight  walk $W=v_1\dots v_r$ we say that two subwalks $W_1=v_i\dots  v_{i+t}$ and $W_2=v_j\dots  v_{j+s}$ are \emph{disjoint} if $j>i+t$ or $i>j+s$, i.e.\ if their index sets are disjoint (hence, they need not be disjoint as vertex sets). Given vertex sequences (not necessarily $\ell$-walks) $W_1,\dots, W_s$, we let $W_1 \dots W_s$ denote the concatenation of $W_1, \ldots, W_s$ in order. Next, if $S=u_1\dots u_s$ and $T=v_1\dots v_t$ are ordered supported sets, we say that an $\ell$-walk $W$ \emph{joins} $S$ and $T$ if its first $s$ vertices are $u_1 \dots u_s$ and its last $t$ vertices are $v_1 \dots v_t$.

We first note that any tight walk of suitable order contains an $\ell$-walk.
\begin{obs}
\label{obs:tight to ell walk}
    Let $k\ge 2$ and $\ell\in[k-1]$. Suppose $W=v_1\dots v_r$ is a $k$-uniform tight walk in a $k$-graph $G$ such that $r\equiv k\pmod{k-\ell}$. Then $v_1 \ldots v_r$ is an $\ell$-walk in $G$ comprised of edges of the form $v_{i+1}\dots v_{i+k}$, whenever $i\equiv 0\pmod{k-\ell}$ and $i \in [0, r-k]$.
\end{obs}

Our strategy for proving \cref{lem:Hamilton paths in blow-ups} will first involve replacing each edge $e\in E(F)$ with a large complete $k$-partite $k$-graph (the blow-up of $e$) in $F^*$, where the sizes of the parts will be determined based on a perfect weighted fractional matching that we obtain via \cref{cor:wt perfect frac matching}. We will consider a suitable $\ell$-path in $F^*$ that spans all the vertices of this $k$-partite $k$-graph. 
Then, we will use a ``connecting path'' to move to the next $k$-partite $k$-graph, and repeat this procedure. 
The following proposition will central to finding spanning $\ell$-paths in suitable $k$-partite $k$-graphs.

\begin{prop}
\label{prop:ell cycle in k-partite k-graph}
    Let $H$ be a complete $k$-partite $k$-graph on $r$ vertices with parts $A,B_1,\dots, B_{k-1}$, where $r$ is at least $1$ and is divisible by $t(k-1)$. Let $|A|=r/t$ and $|B_i|=r(t-1)/t(k-1)$ for all $i\in[k-1]$. Then there exists a spanning $\ell$-cycle in $H$ which contains an ordered edge $e\in B_1\times \dots \times B_{t-1}\times A\times B_t\times \dots\times B_{k-1}$.
\end{prop}
\begin{proof}
    Set $q=(t-1)(k-1)$. We define a sequence $X_1\dots X_{q+k-1}$ of (non-distinct) parts belonging to the set $\{A,B_1\dots,B_{k-1}\}$ as
    \[
        X_1\dots X_{q+k-1}=B_1\dots B_{t-1}A^{(1)}B_t\dots B_{2t-2}A^{(2)}\dots\dots A^{(k-2)} B_{q-t+2}\dots B_qA^{(k-1)},
    \]
	where $A^{(i)}=A$, and, $B_i = B_j$ for the unique $j \in [k-1]$ such that $i \equiv j \pmod{k-1}$, for every $i$ (that is, we think of the indices $B_1 \dots B_q$ cyclically modulo $k-1$).
	To begin with, we argue that we can sequentially list all the vertices of $H$ as $v_1\dots v_r$ with $v_i\in X_i$ for all $i$, where we view the indices of $X_i$ cyclically modulo $q+k-1$.

    Indeed, we first note that $r\equiv 0\pmod{q+k-1}$, since $q+k-1=t(k-1)$ divides $r$ by assumption. This means that the sequence $X_1\dots X_r$ is formed by repeating the sequence $X_1\dots X_{q+k-1}$ exactly $r/(q+k-1)=r/t(k-1)$ times. The set $A$ appears $k-1$ times in $X_1\dots X_{q+k-1}$, and hence appears exactly $r/t=|A|$ times in $X_1\dots X_r$. Similarly, each $B_i$ appears precisely $t-1$ times in $X_1\dots X_{q+k-1}$, and hence there are $r(t-1)/t(k-1)=|B_i|$ instances of $B_i$ in $X_1\dots X_r$. Thus, we can find the desired sequence $v_1\dots v_r$.

    Since $r$ is divisible by $t$, we know that $r\equiv 0\pmod{k-\ell}$. We claim that $v_1\dots v_r$ is the required spanning $\ell$-cycle in $H$. It suffices to show that $v_{i+1}\dots v_{i+k}$ forms an edge for all $i \le r-k$ such that $i\equiv 0\pmod{k-\ell}$.
    
    Let us first fix any subsequence $v_{i+1}\dots v_{i+k}$ of $k$ consecutive vertices such that $i\equiv 0\pmod{k-\ell}$. By construction, for each $i\in [k-1]$, any $k$ consecutive vertices in $v_1\dots v_r$ can contain at most one vertex from $B_i$. By \Cref{prop:special vxs in ell paths}, we know that $v_{i+1}\dots v_{i+k}$ contains at most one vertex of the form $v_j$ with $j\equiv 0\pmod{t}$, and consequently at most one vertex from $A$. Hence, $v_i\dots v_{i+k-1}$ contains exactly one vertex from each of $A,B_1,\dots,B_{k-1}$, implying that it forms an edge of $H$.
	For the final part of the proposition, notice that $v_1 \ldots v_k$ is an ordered edges in the desired set that is contained in the $\ell$-cycle.
\end{proof}

When proving \Cref{lem:Hamilton paths in blow-ups}, we will see that the divisibility condition in the previous proposition might lead to a few leftover vertices in the blow-up of each edge $e$ of $F$.
We thus require that the connecting path in $\Fs$ has some ``absorption'' features to incorporate into itself these leftover vertices. We start by building a tight walk in $F$ that has all the required features, in a manner such that intermediate edges can be deleted to form an $\ell$-walk with suitable start and end vertices. We argue how this can be blown up to a tight path in $F^*$ by selecting distinct vertices along the walk. 

As preparation, we prove a handy result which is based on the proof of of~\cite[Lemma 3.3]{illingworth2025spanning}. Given any two ordered supported sets, we wish to find a tight walk from one ordered edge to the other.

\begin{prop}
\label{prop:initial walk supported sets}
    Suppose $G$ is an $n$-vertex $k$-graph with $\delta^*(G)>n/2$ and no isolated vertices, where $n\ge 4k$. Let $S=u_1\dots u_s$ and $T=v_1\dots v_t$ be ordered supported sets, where with $s,t\in [k]$. Then there exists a tight walk in $G$ joining $S$ and $T$, and has at least $k$ vertices between them.
\end{prop}

\begin{proof}
    Consider any edge $f=x_1\dots x_k$ that is vertex disjoint from $S\cup T$ (which exists since $n\ge 4k$, $G$ has no isolated vertices, and $\delta^*(G) > n/2$). We will construct a tight walk joining $u_1\dots u_s$ and $x_1\dots x_k$, and another one that joins $x_1\dots x_k$ and $v_1\dots v_t$.

    Since $\delta^*(G)>n/2\ge 2k$, there exists an edge $e_S=u_1\dots u_s u_{s+1}\dots u_k$ that contains $S$ and is vertex disjoint from $\{x_1,\dots, x_k\}\cup \{v_1,\dots,v_t\}$. We sequentially find vertices $z_1,\dots, z_k$ such that 
    \[
    z_i\in N^1(u_{i+1}\dots u_k z_1\dots z_{i-1})\cap N^1(x_{i+1}\dots x_k z_1\dots z_{i-1}).
    \]

    Define the following tight walks.
    \[
    W_i=z_1\dots z_{i-1} x_i\dots x_k
    \]
    for all $i\in[k]$, and 
    \[
    Q=u_1\dots u_k z_1\dots z_k.
    \]
    By choice of $z_1,\dots,z_k$, it is not hard to see that these sequences are indeed tight walks (in fact, $W_1,\dots,W_k$ are edges). Additionally, we crucially note that $W_{i+1}W_i$ is a tight walk for all $i\in[k-1]$-- intuitively, we ``cycle through'' the edge $W_{i+1}=z_1\dots z_i x_i\dots x_k$ one vertex at a time for $i-1$ steps, then replace $z_i$ with $x_i$, and cycle through $x_{i+1} \ldots x_k z_1 \ldots z_{i-1}x_i$ for $k-i$ steps. 
	Then $R=QW_k\dots W_1$ forms a tight walk that joins $S=u_1\dots u_s$ and $x_1\dots x_k$. The same arguments can be used to construct a tight walk $R'$ joining $x_1\dots x_k$ and $v_1\dots v_t$. Then $RR'$ is the required tight walk.
\end{proof}
Next, we state an easy observation which will serve as the building block for the absorption.

\begin{obs}
\label{obs:vx tight absorber}
    Let $k\ge 3$ and suppose $G$ is an $n$-vertex $k$-graph with $\delta^*(G)> n/2$, and $v\in V(G)$ is an arbitrary vertex. Then there exist ordered supported sets $W_1$ and $W_2$ in $G$ such that $W_1 W_2$ and $W_1vW_2$ are tight walks and $|W_i|= k-1$ for $i\in[2]$.
\end{obs}

\begin{proof}
    Consider any edge $e=u_1\dots u_{k-1} v$ containing $v$. Since $\delta^*(G)> n/2$, we may sequentially pick vertices $w_1,\dots,w_{k-1}$ such that 
    \[
    w_i\in N^1(u_{i+1}\dots u_{k-1} v w_1\dots w_{i-1})\cap N^1(u_i\dots u_{k-1} w_1 \dots w_{i-1}).
    \]
	Then $W_1=u_1 \dots u_{k-1}$ and $W_2=w_1 \dots w_{k-1}$ are the required ordered supported sets. 
\end{proof}

As mentioned previously, proving \Cref{lem:Hamilton paths in blow-ups} will require a connecting path in $F^*$ that can absorb a few uncovered vertices, and the number of such vertices will turn out to be a multiple of $k-\ell$. Hence, our next step is to build larger tight walks that can integrate into themselves multisets of $k-\ell$ vertices. We would also like to able to delete edges from them to convert them to $\ell$ walks when necessary (via \cref{obs:tight to ell walk}), and thus will also have a restriction on the order of such a walk.

Suppose we have some fixed $k\ge 3$ and $\ell \in [k-1]$. Given a sequence $U=v_1 \dots v_s$ of (not necessarily distinct) vertices, we define a \emph{tight $U$-absorber} to be a sequence of tight walks $(W_0,\dots ,W_s)$ such that $|W_i| \ge k-1$ for every $i$ and $W_0\dots W_s$ and $W_0v_1W_1\dots W_{s-1}v_sW_s$ are both tight walks in $G$.
We define a \emph{$U$-absorber} to be a sequence of finite vertex sequences $(W_0,\dots, W_s)$ with $|W_i| \ge k-1$ for every $i$, such that the vertex sequences formed by $W_0\dots W_s$ and $W_0 v_1 W_1 \ldots W_{s-1} v_s W_s$ are both $\ell$-walks. Note that for tight walks, this is exactly the same as the notion of a tight $U$-absorber. We will often simply refer to the walk $W_0\dots W_s$ as a (tight) $U$-absorber.

Given this, we argue that we can find absorbers of any desired length.

\begin{prop}
\label{prop:sequence tight absorber}
     Suppose $G$ is an $n$-vertex $k$-graph with $\delta^*(G)>n/2$ and no isolated vertices, where $n\ge 4k$, and let $\ell \in [k-1]$. Let $U=v_1\dots v_s$ be any finite vertex sequence. Then, for any $q\in\{0,\dots,k-\ell-1\}$, $G$ contains a tight $U$-absorber $W=W_0\dots W_s$ such that $|W|\equiv q\pmod{k-\ell}$.
\end{prop}
\begin{proof}
    Using \Cref{obs:vx tight absorber} for each $i\in[s]$, we can find tight walks $A_i$ and $B_i$ such that $A_iB_i$ and $A_iv_iB_i$ are tight walks, and $|A_i|,|B_i|= k-1$. Furthermore, since $\delta^*(G)>n/2$ and $n \ge 4k$, by \Cref{prop:initial walk supported sets}, we can find a tight walk $F_i$ joining $B_i$ and $A_{i+1}$ for all $i\in[s-1]$. Hence, $R=A_1B_1F_1A_2B_2\dots F_{s-1} A_{s}B_{s}$ and $R'=A_1v_1B_1F_1 A_2v_2B_2\dots F_{s-1} A_s v_s B_s$ are both tight walks. 
    
	Next, suppose that $|R|\equiv r \pmod{k-\ell}$. Set $p=q-r\pmod {k-\ell}$ and suppose that $A_1 = x_1 \dots x_{k-1}$ (viewed as an ordered supported set). Then, using minimum supported co-degree condition, we successively pick vertices $y_1,\dots, y_p$ where $y_i\in N^1(y_{i-1}\dots y_1 x_1\dots x_{k-i})$, which ensures that $y_p\dots y_1 A_1$ is a tight path. Set $W_0=y_p\dots y_1 A_1$, $W_i=B_i F_i A_{i+1}$ for all $i\in[s-1]$ and $W_s=B_s$. Since $|A_i|, |B_i| = k-1$ for every $i \in [s]$, we have $|W_i| \ge k-1$ for every $i \in [s]$. Then the $W_0\dots W_s$ constitutes the desired tight $U$-absorber. 
\end{proof}

Next, we use the absorbers we just found to modify the above walk obtained in \Cref{prop:initial walk supported sets} to control its order, which will allow us to convert it into a suitable $\ell$-walk by deleting edges.

\begin{prop}
\label{prop:correct order walk supported sets}
    Suppose $G$ is an $n$-vertex $k$-graph with $\delta^*(G)>n/2$ and no isolated vertices, where $n\ge 4k$, and let $\ell \in [k-1]$. Suppose $S$ and $T$ are ordered supported sets of size $s$ and $t$ respectively, with $s,t\in[k]$. Then, for any $q\in\{0,\dots,k-\ell-1\}$, there exists a tight walk in $G$ joining $S$ and $T$ whose order is at least $2k$ and is congruent to $q \pmod{k-\ell}$.
\end{prop}

\begin{proof}
    Consider any sequence of $k-\ell$ vertices $U=v_1\dots v_{k-\ell}$. Applying \Cref{prop:sequence tight absorber}, we can find a tight $U$-absorber $W_0 \dots W_{k-\ell}$ whose order is congruent to $k\pmod{k-\ell}$. Crucially, by construction, the same is true of the tight walk  $W_0v_1W_1\dots W_{k-\ell-1}v_{k-\ell}W_{k-\ell}$.
	Let $S'$ be the ordered supported set corresponding to the first $k - 1$ vertices of $W_0$ and $T'$ be the ordered supported set corresponding to the last $k - 1$ vertices of $W_{k-\ell}$. By \Cref{prop:initial walk supported sets}, we have tight walks $R_S$ and $R_T$ of order at least $k$ that join $S$ to $S'$ and $T'$ to $T$ respectively. Consequently, the tight walk $R_S W_0\dots W_{k-\ell} R_T$ joins $S$ and $T$. Suppose that the order of this walk is congruent to $r\pmod{k-\ell}$, and set $p=q-r\pmod{k-\ell}$. Then the walk $R W_0 v_1 W_1\dots W_{p-1} v_p W_p W_{p+1} W_{p+2}\dots W_{k-\ell}Q_{k-\ell} R'$ meets the conditions we need.
\end{proof}

We now put it all together. As earlier, we let $E^*(G)$ denote the set of ordered edges of $G$. We wish to find a tight walk that contains every ordered edge, $e\in E^*(G)$, has several (tight) absorbers for each $(k-\ell)$-sequence, and whose order is congruent to $k\pmod{k-\ell}$.

\begin{lem}
\label{lem:final tight walk}
    Let $G$ be an $n$-vertex $k$-graph with $\delta^*(G)>n/2$ and no isolated vertices, where $n\ge 4k$. Let $\ell\in[k-1]$ and let $d$ be any positive integer. Suppose $S=u_1\dots u_s$ and $T=v_1\dots v_t$ are ordered supported sets of order $s$ and $t$ respectively, where $s,t\in[k]$. Then there exists a tight walk $W$ that satisfies the following conditions.
    \begin{enumerate}[label=\textnormal{(T\arabic*)}, ref=(T\arabic*)]
        \item\label{itm:T1} $W$ joins $S$ and $T$,
        \item\label{itm:T2} $|W|\equiv k\pmod{k-\ell}$,
        \item\label{itm:T3} for every ordered edge $e\in E^*(G)$, the walk $W$ contains a subwalk $W_e$ corresponding to $e$, such that these subwalks are pairwise disjoint, and the first vertex of each such subwalk is at a position congruent to $1\pmod{k-\ell}$ in $W$,
        \item\label{itm:T4} for each $(k-\ell)$-sequence $U$, the walk $W$ contains at least $d$ subwalks that are tight $U$-absorbers, all such subwalks are pairwise disjoint from each other, and the first vertex of each of these subwalk is at a position congruent to $1\pmod{k-\ell}$ in $W$, and 
        \item\label{itm:T5} the subwalks in~\ref{itm:T3} and~\ref{itm:T4} are pairwise disjoint from each other.
    \end{enumerate}
\end{lem}
\begin{proof}
    Let $\cP=\{P_1,\dots, P_x\}$ be a collection of tight walks satisfying the following properties.
    \begin{itemize}
        \item $|P_i|\equiv k\pmod{k-\ell}$ for all $i\in[x]$,
        \item for each $e\in E^*(G)$, there exist some distinct $i_e \in [x]$ such that $P_{i_e} = e$,, 
		\item for each $(k-\ell)$-sequence $U$, there exists some $I_U \subseteq [x]$ such that $|I_U|\ge d$ and every $P_i$ with $i \in I_U$ is a tight $U$-absorber, and
        \item for any two $(k-\ell)$-sequences $U\ne U'$, the sets $I_U$ and $I_{U'}$ are disjoint, and $i_e \notin I_U$ for any ordered edge $e$ and $(k - \ell)$-sequence $U$.
    \end{itemize}
    Observe that we can always find such a collection $\cP$ using \Cref{prop:sequence tight absorber}. Then, by \Cref{prop:correct order walk supported sets}, we may find a set of tight walks $\cQ=\{Q_0,\dots, Q_x\}$ such that $Q_0$ joins $S$ and the first $k-1$ vertices of $P_1$, $Q_i$ joins the last $k-1$ vertices of $P_i$ and the first $k-1$ vertices of $P_{i+1}$ for all $i\in[x-1]$, $Q_x$ joins the last $k-1$ vertices of $P_x$ and $T$ and $|Q_i|\ge 2k$ for all integers $i \in [0,x]$. Additionally, we also impose that $|Q_0| \equiv k-1 \pmod{k-\ell}$, $|Q_i| \equiv k - 2 \pmod{k - \ell}$ for every $i \in [x-1]$, and $|Q_x| \equiv 2k-1 \pmod{k - \ell}$. Let $Q_0'$ be obtained by removing the last $k - 1$ vertices from $Q_0$, so $|Q_0'| \equiv 0 \pmod{k - \ell}$; let $Q_i'$ be obtained by removing the first and last $k - 1$ vertices in $Q_i$ for $i \in [x-1]$, so $|Q_i'| \equiv -k \pmod{k - \ell}$; and let $Q_x'$ be obtained by removing the first $k - 1$ vertices from $Q_x$, so $|Q_x'| \equiv k \pmod{k - \ell}$.
	Let $W$ be the tight walk $Q_0' P_1 Q_1' \dots P_x Q_x'$. Clearly $W$ joins $S$ and $T$ by construction. It is also straightforward to see that
    \[
		|W|\equiv \sum_{j \in [0,x]}|Q_j'| + \sum_{j \in [x]}|P_j| \equiv k\pmod{k-\ell}.
    \]
    Furthermore, for every $i\in [x]$, we see that the first vertex of $P_i$ is at position
    \[
        |Q_0'|+\sum_{j=1}^{i-1}|P_j|+\sum_{j=1}^{i-1}|Q_j'| + 1
		\equiv 1 \pmod{k-\ell}.
    \] 
    in $W$. Thus, $W$ is the desired walk.
\end{proof}

Next, by applying \Cref{obs:tight to ell walk}, we immediately obtain the corresponding statement for $\ell$ walks.
\begin{cor}
\label{cor:final ell walk}
        Let $G$ be an $n$-vertex $k$-graph with $\delta^*(G)>n/2$ and no isolated vertices, where $n\ge 4k$. Let $\ell\in[k-1]$ and let $d$ be any positive integer. Suppose $S=u_1\dots u_s$ and $T=v_1\dots v_t$ are ordered supported sets of order $s$ and $t$ respectively, where $s,t\in[k]$. Then there exists an $\ell$-walk $W$ that satisfies the following conditions.
    \begin{enumerate}[label=\textnormal{(W\arabic*)}, ref=(W\arabic*)]
        \item\label{itm:W1} $W$ joins $S$ and $T$,
        \item\label{itm:W2} for every ordered edge $e\in E^*(G)$, the walk $W$ contains a subwalk corresponding to $e$ and all such subwalks are pairwise disjoint,
		\item\label{itm:W3} for each $(k-\ell)$-sequence $U$, the walk $W$ contains at least $d$ subwalks that are $U$-absorbers and all such subwalks are pairwise disjoint, and
        \item\label{itm:W4} the subwalks from claims~\ref{itm:W2} and~\ref{itm:W3} are pairwise disjoint from each other
    \end{enumerate}
\end{cor}

Finally, we note that in a regular blow-up of $F$ with $\delta^*(F) > 5s/9$ (say), one can find an $\ell$-path (rather than walk) with similar properties to the above.
\begin{lem}
\label{lem:walk to path Q}
    Let $1/m\ll 1/s\ll \gamma\ll 1/k\le 1/3$. Suppose $F$ is an $s$-vertex $k$-graph with $\delta^*(F)> 5s/9$ and let $F^*$ be a $(\gamma,m)$-regular blow-up of $F$. Let $e_1,e_2\in E(F^*)$ be any two ordered edges. Then there exists an $\ell$-path $Q$ in $F^*$ from $e_1$ to $e_2$ that satisfies the following conditions.
    \begin{enumerate}[label=\textnormal{(Q\arabic*)}, ref=(Q\arabic*)]
        \item \label{itm:Q1} $Q$ joins $e_1$ and $e_2$,
		\item \label{itm:Q2} $|Q|\le \gamma m$, 
        \item \label{itm:Q3} for every ordered edge $e\in E^*(F)$, the path $Q$ contains a subpath corresponding to blow-up of $e$, and these blow-ups are all pairwise disjoint,
        \item \label{itm:Q4} let $\cU=\{U_1,\dots,U_r\}$ be any collection of pairwise disjoint $(k-\ell)$-subsets of $V(F^*)$ such that $r\le k^2 s^k$. Then the path $Q$ contains at least one subpath that is a $U_i$-absorber for each $i\in[r]$, and all these subpaths are pairwise disjoint, and
        \item\label{itm:Q5} the subpaths from claims~\ref{itm:Q3} and~\ref{itm:Q4} are pairwise disjoint from each other.
    \end{enumerate} 
\end{lem}
\begin{proof}
    Since $\delta^*(F)>5s/9>s/2$, we can use \Cref{cor:final ell walk} to find an $\ell$-walk $W$ in $F$ that satisfies \ref{itm:W1}--\ref{itm:W4} with $s$, $k^2 s^k$, $\phi(e_1)$ and $\phi(e_2)$ in place of $n$, $d$, $S$ and $T$ respectively.

    Let $|W|=x$, where $x=x(s,k,\ell)$. We now blow this up to an $\ell$-path $Q$ in $F^*$. Based on our parameter choices, we may suppose that $x\le (1-\gamma)m$. Setting $W=w_1\dots w_x$, for each $i\in [x]$, we choose distinct vertices $v_i\in V(F^*)$ such that $\phi(v_i)=w_i$, which is possible since $x\le(1-\gamma)m\le |\phi^{-1}(w)|$ for every $w\in V(F)$. Then, by choice of $W$, it is clear that $Q=v_1\dots v_x$ is an $\ell$-path in $F^*$ which satisfies~\ref{itm:Q1}--\ref{itm:Q3}. To see \ref{itm:Q4}, we note that $W$ has at least $k^2 s^k$ subwalks that are $\phi(U_i)$-absorbers for each $i\in [r]$. Since $r\le k^2 s^k$ the $\ell$-path $Q$ will contain a distinct absorber for each $U_i$. Then property~\ref{itm:W4} of $W$ immediately implies~\ref{itm:Q5} for $Q$.
\end{proof}

\section{Proof of \Cref{lem:Hamilton paths in blow-ups}}\label{sec:proof Hamilton paths}
    Our goal in this section is to prove \Cref{lem:Hamilton paths in blow-ups}, which asserts that nearly-regular blow-ups of non-extremal graphs with a certain divisibility condition have a Hamilton $\ell$-path (with specified start and end). To avoid having to worry about the singleton part throughout the proof, we now state a version of \Cref{lem:Hamilton paths in blow-ups} for regular blow-ups (i.e.\ where there is no singleton part) and show that it implies \Cref{lem:Hamilton paths in blow-ups}. The rest of the section will then be dedicated to the proof of this new version.  
	As always, given $k$ and $\ell$, we set $t = \floor{\frac{k}{k-\ell}}(k-\ell)$ and recall that $t \ge 3$ for all $k\ge 3$ and $\ell \in [k-1]$ except $(k, \ell) \ne (3,1)$, as this will be fairly important to our proof.

	\begin{lem} \label{lem:Hamilton paths in blow-ups no singleton}
		Let $1/m\ll 1/s\ll \gamma\ll \eps\ll \mu\ll 1/k\le 1/3$, let $\ell \in [k-1]$, and suppose $(k,\ell) \ne (3,1)$. Suppose $F\in \nonexE_{\eps,\mu}(s,k,\ell)$ and $F^*$ is a $(\gamma,m)$-regular blow-up of $F$ satisfying $|\Fs| \equiv k \pmod{k-\ell}$. Let $e_1,e_2\in E(F^*)$ be two disjoint ordered edges. Then there exist a spanning $\ell$-path in $F^*$ from $e_1$ to $e_2$.
	\end{lem}

	\begin{proof}[Proof of \Cref{lem:Hamilton paths in blow-ups} using \Cref{lem:Hamilton paths in blow-ups no singleton}]
		If $F^*$ does not contain a singleton part, then we are done, so we may assume that it does. Let $u^*\in V(F),V(F^*)$ be the vertex corresponding to the singleton part. Since $F\in \nonexE_{\eps,\mu}(n,k,\ell)$, we know $\delta^*(F) \ge (1-1/t-\eps)s$, where $t\ge 3$ as $(k,\ell)\ne (3,1)$. Let $H=F-u^*$. It is easy to check that $H \in \nonexE_{2\eps, \mu/2}(n,k,\ell)$. In particular, $\delta^*(H)\ge \delta^*(F)-1 > s/2$. 
    
		We claim that there is a tight walk $W$ in $F$ such that $W$ starts in $\phi(e_1)$, ends in some edge $f$ in $H$, contains exactly one instance of $u^*$, and has order $k \pmod{k - \ell}$.
		Indeed, let $e^* = \{u^*, u_1, \ldots, u_{k-1}\}$ be any edge in $F$ that contains $u^*$, and let $f = \{u_1, \ldots, u_k\}$ be any edge in $F$ that contains $\{u_1, \ldots, u_{k-1}\}$ and avoids $u^*$ (which exists by the assumption on $\delta^*(F)$). By \cref{prop:correct order walk supported sets}, we can find a tight walk $W'$ in $H$ from $\phi(e_1)$ to $u_1 \ldots u_{k-1}$ such that $|W'|\equiv -1 \pmod{k-\ell}$. Then $W = W' u^* u_1 \ldots u_k$ satisfies the requirements. Write $W = w_1 \ldots w_x$ where $|W| = x$ can be bounded as a function of $s$, $k$ and $\ell$, and we may thus assume $x \le \gamma m$.
    
		Our next step is to convert this into a tight path in $F^*$ that avoids $e_2$. For each $i\in [x]$, we choose a vertex $v_i\in V(F^*)\setminus V(e_2)$ such that $\phi(v_i)=w_i$, all the $v_i$ are distinct, and the first $k$ vertices correspond to $e_1$. This is possible since $x \le \gamma m-k\le |\phi^{-1}(w)- V(e_2)|$ for every $w\in V(F)\setminus \{u^*\}$, and recalling that $u^*$ appears precisely once in $W$. Then $P=v_1 \dots v_x$ is a tight path in $F^* - V(e_2)$ from $e_1$ to $f^*$ (where $\phi(f^*)=f$), it contains $u^*$, and $|P| \equiv k \pmod{k - \ell}$. By \Cref{obs:tight to ell walk}, the sequence $v_1 \ldots v_x$ can be thought of as an $\ell$-path from $e_1$ to $f^*$.

		Set $H^*=F^*-\{v_1, \ldots, v_{x-k}\}$. Notice that if $H^*$ contains a spanning $\ell$-path from $f^*$ to $e_2$, then $F^*$ contains a spanning $\ell$-path from $e_1$ to $e_2$ (obtained by adding $v_1 \ldots v_{x-k}$ at the beginning).
		By \Cref{lem:Hamilton paths in blow-ups no singleton}, it will suffice to show that $|H^*|\equiv k\pmod{k-\ell}$ and that $H^*$ is a $(2\gamma,m)$-regular blow-up of $H$, since $f^*,e_2\in E(H^*)$ are disjoint ordered edges by choice. The first condition is immediate as $|H^*|=|F^*|-(x-k) \equiv |F^*| \pmod{k-\ell} \equiv k \pmod{k-\ell}$. Next, notice that $H^*$ is indeed a blow-up of $H$, as $u^* \notin V(H^*)$. Moreover, each $w\in V(H)$ is blown up to at least $(1-\gamma)m$ vertices in $F^*$, meaning that there are at least $(1-\gamma)m -x\ge (1-2\gamma) m$ copies of $w$ left in $H^*$. 
	\end{proof}

	Before turning to the proof of \Cref{lem:Hamilton paths in blow-ups no singleton}, we recall some notation.
	As defined in \Cref{sec:non-extremal}, we let $\phi:V(\Fs)\to V(F)$ denote the projection map, that is, $\phi(v')=v$ for all $v'$ in the independent set corresponding to $v$. As in \Cref{sec:Farkas}, we use $E^*(H)$ to denote the set of ordered edges of a given $k$-graph $H$, and, for every $v\in V(H)$ and $i\in[k]$, we write $E_i^v(H)\subseteq E^*(H)$ for the set of ordered edges with $v$ as their $i$th vertex. Finally, for any ordered edge $e=v_1\dots v_k\in E^*(F)$, we use $e^{+t}\in E^*(F)$ to be the ordered edge $v_2\dots v_{t-1} v_1 v_t\dots v_k$.

To begin with, we will integrate blow-ups of edges of $F$ into the $\ell$-path $Q$ in $\Fs$ obtained by \cref{lem:walk to path Q}. The idea is as follows. Every ordered edge $e\in E^*(F)$ is blown up to a $k$-partite $k$-graph in $\Fs$. We will consider a subgraph of this $k$-graph (obtained by paring down the sizes of the parts based on the $\bfw^*$-weighted perfect fractional matching) that will be almost perfectly balanced except for one part, and will find suitable a spanning $\ell$-path in this subgraph using \Cref{prop:ell cycle in k-partite k-graph}. The idea will then be to replace the edge of $Q$ that corresponds to $e^{+t}$ with this spanning $\ell$-path. Overall, we will partition almost all of $V(F^*)$ into such $k$-partite $k$-graphs, and replacing the relevant edges of $Q$ with the corresponding spanning $\ell$-paths will provide an $\ell$-path that covers all but at most a few vertices of $V(F^*)$. We will then use the absorption features of $Q$ to absorb these into the path as well. 

\begin{proof}[Proof of \Cref{lem:Hamilton paths in blow-ups no singleton}]
    Let $e_1$ and $e_2$ be as in the statement. As $F \in \nonexE_{\eps, mu}(s,k,\ell)$ and $t\ge 3$, we can apply \Cref{lem:walk to path Q} to find an $\ell$-path $Q$ in $\Fs$ from $e_1$ to $e_2$ satisfying \ref{itm:Q1}--\ref{itm:Q5} with $e_1$ and $e_2$. By~\ref{itm:Q2}, we know that $|V(Q)|\le \gamma m$.
	Write $\Fss := \Fs \setminus V(Q)$; so $\Fss$ is a $(2\gamma, m)$-regular blow-up of $F$. By \cref{prop:blowup is nonextremal} and \cref{cor:wt perfect frac matching}, there is a $\bfw^*$-weighted perfect fractional matching in $\Fss$, where we recall that $\bfw^*=(w_1,\dots,w_k)$ with $w_1=k-1$ and $w_i=t-1$ for all $2\le i\le k$. Denote this $\bfw^*$-weighted perfect fractional matching $q$.

	For each vertex $v\in V(F)$, we let $B_v\subseteq V(\Fss)$ denote the vertex subset corresponding to the blow-up of $v$. 
    For every ordered edge $e\in E^*(F)$, define 
    \[
        \hat{q}(e)=\sum_{\substack{e'\in E^*(\Fss)\\ \phi(e') = e}} q(e').
    \]
    Observe that
    \begin{align*}
        \sum_{i\in [k]}\sum_{e\in E^v_i(F)}w_i \hat{q}(e)&=\sum_{i\in [k]}\sum_{e\in E^v_i(F)}\sum_{\substack{e'\in E^*(\Fss)\\ \phi(e') = e }} w_i q(e')\\
        &=\sum_{v'\in B_v}\sum_{i\in [k]}\sum_{e'\in E^{v'}_i(\Fss)} w_i q(e')\\
        &=\sum_{v'\in B_v} 1=|B_v|.
    \end{align*}
    For every vertex $v\in V(F)$, we arbitrarily partition $B_v$ into sets $B_v^e$, for all ordered edges $e$ in $E^*(F)$ that contain $v$, and an additional set $B_v^*$, so that 
    $|B_v^e|=w_i\floor{\hat{q}(e)}$ for every $e\in E_i^v(F)$ and $i\in[k]$, and $B_v^*$ consists of all the remaining vertices. We see that
    \begin{align*}
        |B_v^*|&=|B_v|-\sum_{e\ni v} |B_v^e|\\
        &=\sum_{i\in[k]}\sum_{e\in E_i^v(F)}w_i (\hat{q}(e)-\floor{\hat{q}(e)})\\
        &\le \sum_{i\in[k]}\sum_{e\in E_i^v(F)} (k-1)\\
        &\le k^2 s^{k-1},
    \end{align*}
    since there are at most $s^{k-1}$ ordered edges in $F$ containing $v$ in the $i$th position. Finally, for any ordered edge $e=v_1 \dots v_k\in E^*(F)$, define $T_e\subseteq \Fss$ to be the complete $k$-partite $k$-graph with parts $B_{v_1}^e,\dots,B_{v_k}^e$. 

    We now modify $Q$ into the desired Hamilton $\ell$-path in $F^*$. We will first replace appropriate edges of $Q$ with spanning $\ell$-paths in each $T_e$, which will form an $\ell$-path that spans all the vertices except those in $\bigsqcup_{v\in V(F)} B_v^*$. We then absorb these remaining few vertices into $Q$ using property~\ref{itm:Q4}.

    We start with the latter part, detailing the absorption mechanism. Define $B^*=\bigcup_{v\in V(F)} B_v^*$. We will need the following simple claim.
	\begin{claim}
		$|B^*|$ is divisible by $k-\ell$ and is at most $k^2 s^k$.
	\end{claim}

	\begin{proof}
		Observe that
		\[
			\bigsqcup_{v\in V(F)} \bigsqcup_{\substack{e\in E^*(F)\\ e\ni v}} B_{v}^e= \bigsqcup_{e\in E^*(F)} T_{e},
		\]
		implying that $|B^*| = |\Fss| - \sum_{e \in E^*(F)}|T_e|$. Thus, for divisibility, it suffices to show that each $|T_e|$ and $|\Fss|$ are divisible by $k-\ell$.
		For $T_e$, this follows from the following calculation.
		\begin{equation}\label{eq:size Te}
			|T_e|=\sum_{i\in[k]} |B_{v_i}^e|=\sum_{i\in[k]}w_i\floor{\hat{q}(e)}=\floor{\hat{q}(e)}(k-1+(k-1)(t-1))=\floor{\hat{q}(e)}t(k-1).
		\end{equation}
		Since $Q$ is an $\ell$-path, we know $|V(Q)|\equiv k\pmod{k-\ell}$, and by assumption the same is true of $|\Fs|$, which means $|\Fs|-|V(Q)|\equiv 0\pmod{k-\ell}$. This establishes that $|B^*| \equiv 0 \pmod{k - \ell}$. 
		For the upper bound, notice that $|B_v^*|\le k^2 s^{k-1}$ for each $v\in V(F)$. Thus $|B^*|\le |V(F)|\cdot k^2 s^{k-1}=k^2 s^k$.
	\end{proof}

    By the last claim we can construct a partition $B^*=\bigsqcup_{i=1}^{r} W_i$, where each $|W_i|=k-\ell$, and so that $r\le k^2 s^k$. Then, by~\ref{itm:Q4}, we can modify $Q$ to an $\ell$-path $Q'$ from $e_1$ to $e_2$ such that $V(Q') = V(B^*)\cup V(Q)$ and, using~\ref{itm:Q5}, we deduce that $Q'$ retains property~\ref{itm:Q3}.
    Hence, we have 
    \[
        V(F^*)\setminus V(Q')=\bigsqcup_{v\in V(F)} \bigsqcup_{\substack{e\in E^*(f)\\ e\ni v}} B_v^e= \bigsqcup_{e\in E^*(F)} T_e.
    \]
    We move on to the final step, where we replace certain edges of $Q'$ with spanning paths in each $T_e$, forming a spanning $\ell$-path. The following claim is key to this.
    \begin{claim}
    \label{claim:cycle in edge blowup}
        For every ordered edge $e=v_1\dots v_k \in E^*(F)$ for which $T_e$ is non-empty, there exists a spanning $\ell$-cycle $C_e\subseteq T_e$ which contains an ordered edge that is a blow-up of the ordered edge $e^{+t}\in E^*(F)$.
    \end{claim}
    \begin{proof}
        As noted in~\eqref{eq:size Te}, we know $|T_e|=\floor{\hat{q}(e)}t(k-1)$, and hence $|T_e|$ is divisible by $t(k-1)$. Additionally, $|B_{v_1}^e|=(k-1)\floor{\hat{q}(e)}=|T_e|/t$ and $|B_{v_i}^e|=(t-1)\floor{\hat{q}(e)}=|T_e|(t-1)/t(k-1)$. Hence, by setting $A=B_{v_1}^e$ and $(B_1,\dots,B_{k-1})=(B_{v_2}^e,\dots, B_{v_k}^e)$ and applying \Cref{prop:ell cycle in k-partite k-graph}, we can find a spanning $\ell$-cycle $C_e\subseteq T_e$ that contains an ordered edge from $B_{v_2}^e\times\dots\times B_{v_t}^e \times B_{v_1}^e \times B_{v_{t+1}}^e \times\dots \times B_{v_k}^e$, that is, a copy of the ordered edge $e^{+t}$.
    \end{proof}

    Now, consider an arbitrary ordered edge $e\in E^*(F)$ for which $T_e$ is non-empty. From \Cref{claim:cycle in edge blowup}, we know that there exists a spanning $\ell$-cycle $C_e\subseteq T_e$ containing an ordered edge, say $c$, which is a copy of the ordered edge $e^{+t}\in E^*(F)$. Let $P_e= C_e- c$ be the $\ell$-path in $T_e$ obtained by removing the edge $c$ (that is, we just think of $P_e$ as consisting of all the edges of $C_e$ except $c$, ordered in a fashion that makes it an $\ell$-path). By~\ref{itm:Q3}, we know that $Q'$ has a blow-up of every ordered edge of $F$, and thus in particular contains an ordered edge $d$ which is a blow-up of $e^{+t}$. Let $f_1$ and $f_2$ be the edges that precede and succeed $d$ in $Q'$. Since $c$ and $d$ are blow-ups of the same ordered edge of $F$, and by definition of the $\ell$-path $P_e$, we may replace the subpath $f_1 d f_2$ in $Q'$ with the subpath $f_1 d P_e c f_2$ (here we view an $\ell$-path as a sequence of edges, without explicitly describing the vertex sequence as we did previously). This yields an $\ell$-path on vertex set $V(Q') \cup V(T_e)$. Since the $k$-graphs $T_e$ are pairwise disjoint for each $e\in E^*(F)$ and, by~\ref{itm:Q3}, $Q'$ contains disjoint subpaths corresponding to the set of edges $\{e^{+t}:e\in E^*(F)\}$, we may repeat this procedure for every $e\in E^*(F)$ (that is, for every $T_e$), and let $Q''$ be the resulting $\ell$-path. Then $Q''$ is an $\ell$-path from $e_1$ to $e_2$ that spans $V(F^*)$, precisely as required. 
\end{proof}

\section{Extremal hypergraphs}\label{sec:extremal hypergraphs}
Now we turn our attention to hypergraphs that obey the exact minimum supported co-degree condition and whose structure is close to those of the extremal examples, complementing \cref{thm:Hamilton cycle non-extremal}. More precisely, we will prove \cref{thm:Hamilton cycle extremal} (restated below) over the next four sections. In this section, we first provide a detailed proof sketch and then introduce various tools that we will need for the proof. \thmHamiltonCycleExtremal*

\subsection{Proof sketch for \cref{thm:Hamilton cycle extremal}}\label{sec:pf sketch extremal}
As in \Cref{sec:pf overview extremal}, in this sketch we focus on tight cycles, namely on the case $\ell = k-1$ (and $t = k$), and comment on the necessary modifications for general $\ell$-cycles at the end. Suppose $G\in \exE_\eps (n,k,k-1)$. Recall that this implies that $G$ has a vertex subset $A$ of size $\floor{n/k}$ with at most $\eps n^2$ supported pairs, and set $B= V(G)\setminus A$.

First, since there are so few supported pairs in $A$, there are very few vertices which have large vertex degree in $A$. We transfer these to $B$ and argue that, after this modification, every supported set that intersects $A$ has high co-degree into $B$. The extremal structure also implies that almost every $(k-1)$-subset of $B$ is \emph{$A$-rich} -- meaning that it supports almost all of $A$ (\Cref{defn:scarce and rich sets}). We then use a ``cleaning lemma'' (\cref{lem:dense subgraphs}) to find a nearly spanning subgraph of $A$-rich sets in $B$ with nearly complete minimum supported co-degree into $B$, and transfer the few uncovered vertices into $A$. Throughout, we maintain $|A| \approx n/k$. These modifications are summarised in \Cref{prop:final A' conditions}.

Next, we ``balance'' the sizes of $A$ and $B$ by replacing a few vertices of $A$ with short tight paths. The exact vertex degree condition (via \cref{obs:min 1-degree}) provides us with a few paths of length two in $\partial^2[A]$, and we extend these to tight paths, including one special path of some specified length to ensure that $|B| = (k-1)|A|$. Additionally, these paths will start and end with supported sets that have high co-degree across the $A\sqcup B$ partition. This balancing is achieved in \Cref{prop:size balancing final}.

The set $A$ now consists mostly of single vertices and a few tight paths.
As explained in \Cref{sec:pf overview extremal}, we arbitrarily partition $B$ into $k-1$ sets $B_1, \dots, B_{k-1}$ such that $|A| = |B_i|$ for all $i$, and define an auxiliary $k$-partite $k$-graph $\Gp$ with parts $B_1, \dots, B_{k-1}, A$ whose edges correspond to vertex sequences in $B_1 \times \dots \times B_{k-1} \times A$ that form tight paths. We use a random greedy procedure, based on a simple version for almost complete balanced bipartite graphs (see \Cref{lem:random matching key lemma abstract,lem:k-partite matching}), to build a perfect matching $M$ in $\Gp$ by starting with vertices in $B_1$, extending them to ordered supported sets in $B_1 \times B_2$, and so on. We then form an auxiliary digraph whose vertices correspond to the vertex sequences of $M$, and for any $P \in M$, we let its in-neighbours (respectively, out-neighbours) be all the $P' \in M$ such that $P' P$ (respectively, $P P'$) forms a tight path. We argue that the randomness preserves the high supported co-degree conditions across the partition $A\sqcup B$ so that this digraph has high minimum semi-degree. This will ensure a directed Hamilton cycle, which corresponds to a tight Hamilton cycle in $G$. This final step of the proof is done in \Cref{sec:Hamilton cycle extremal}. 

The major difference for $\ell$-cycles is that in our balancing step, instead of replacing vertices with tight paths, we use vertex sequences of specified order that ``support $\ell$-paths'' (see \Cref{def:supported-l-paths}) -- intuitively, these consist of $\ell$-paths for ``as long as possible'' and have some small supported sets on either end. The upshot is that concatenating two such sequences forms another sequences that supports an $\ell$-path. We do this to make $|B| =  (t-1)|A|$. In the next step, we define $\Gp$ to be an auxiliary $t$-partite $t$-graph on $B_1 \times \dots \times B_{t-1} \times A$ whose edges are vertex sequences that support $\ell$-paths. The rest of the proof is largely similar.

\paragraph{Strucutre of the next four sections.}
In the next section, we prove fundamental results for our proof strategy, starting with basics about the minimum supported co-degree and some simple graph theoretic statements in \Cref{sec:graph lemmas}. Then \Cref{sec:supported ell paths} describes key properties about vertex sequences that support $\ell$-paths. \Cref{sec:cleaning} provides a general mechanism for finding a nearly spanning subgraph that approximately preserves supported co-degree conditions in a dense host graph. We then use this to analyse our random matching procedure in \Cref{sec:random matching}.

Once we have set up the necessary machinery, as described in our proof sketch, we modify $G$ to ensure desired co-degree properties in \Cref{sec:extremal structure}, balance the part sizes in \Cref{sec:balancing}, and subsequently construct a Hamilton cycle in \Cref{sec:Hamilton cycle extremal}.

\section{Key lemmas}\label{sec:key lemmas extremal}
We start with a quick observation based on the minimum supported co-degree condition.
\begin{obs}
\label{obs:min 1-degree}    
Let $k \ge 3$, $\ell \in [k-1]$, and $t = \floor{\frac{k}{k-\ell}}(k-\ell)$ and let $G$ be a $k$-graph with $\delta^*(G) \ge n - \floor{n/t} - (k-3)$.
    For all $i\in [k-1]$ and for any supported $i$-set $S\in V(G)$, we have 
    \[
		d_G^1(S)\geq n - \floor{\frac{n}{t}} - (i-2) \ge \left(1 - \frac{1}{t}\right)n - k.
    \]
\end{obs}
\begin{proof}
    Let $S=\{v_1,\dots,v_i\}$. Since $S$ is supported, there must exist a set $W=\{w_{i+1},\dots, w_{k}\}$ of $(k-i)$ vertices such that $S\cup W\in E(G)$. In particular, $S'=\{v_1,\dots,v_i,w_{i+1},\dots,w_{k-1}\}$ is a supported $(k-1)$-set. It is clear that $W, N_G^1(S')\subseteq N_G^1(S)$.
    Thus, we have
    \begin{align*}
    d_G^1(S)& \geq d_G^1(S')+|W|-1\\
    & \geq \delta^*(G)+k-i-1\\
	&\geq n - \floor{\frac{n}{t}}-(i-2) \ge \left(1 - \frac{1}{t}\right)n - k, 
    \end{align*}
    where we subtract one in the first inequality to prevent counting $w_{k}\in N_G^1(S')\cap W$ twice.
\end{proof}

\subsection{Graph theoretic lemmas}\label{sec:graph lemmas}
We now prove a simple graph theoretic result about short disjoint paths. A \emph{cherry} in a graph will refer to a path of length two.

\begin{lem}
\label{lem:disjoint edges cherries}
Suppose $H$ is a graph on $n$ vertices with minimum degree $\delta(H)\geq x \ge 1$, and suppose that $A$ is a subset of $V(H)$ satisfying $|A|> n/2 + 9x/2$. Then $H$ contains a linear forest $F$ that consists of $x-1$ cherries with leaves in $A$.
\end{lem}

\begin{proof}
	Let $F'$ be a disjoint union of cherries with leaves in $A$, such that $|F'|$ is maximised. 
	Let $A' := A \setminus V(F')$ and let $E'$ be the set of edges with exactly one end in $A'$. 
	Then, $H[A']$ has maximum degree at most $1$ (otherwise we could add a cherry with vertices in $A'$ to $F'$, contradicting maximality), showing that $|E'| \ge |A'| (x-1)$. 
	We now wish to upper bound $|E'|$, using the following points.
\begin{itemize}
	\item
		For each cherry $uvw$ in $F'$, at most one of $u$, $v$ and $w$ has at least four neighbours in $A'$. Indeed, otherwise we can find two disjoint cherries with centres in $\{u,v,w\}$ and leaves in $A$, and replace $uvw$ by these cherries, contradicting the maximality of $F'$.
		It follows that the number of edges from $A'$ to $\{u,v,w\}$ is at most $|A'| + 6$.
	\item
		There is at most one edge from $A'$ to each vertex in $V(H) \setminus (A \cup V(F'))$, as otherwise we could add an additional cherry to $F$.
\end{itemize}
Altogether, denoting by $y$ the number of cherries in $F'$ (so that $|F'| = 3y$) we find that 
\begin{equation} \label{eqn:cherries}
	|E'| \le y(|A'| + 6) + (n - |A|).
\end{equation}
It remains to show that $y \ge x-1$.
Otherwise, by combining \eqref{eqn:cherries} with $|E'| \ge (x-1)|A'|$, we get $|A'| + |A| \le 6x + n$. This, along with the observation that $|A| - |A'| \le |F'| \le 3x$, implies that $|A| \le n/2 + 9x/2$, a contradiction. 
\end{proof}

Our next preliminary result is a directed version of Dirac's theorem from~\cite{ghouilahouri1960condition}. Given a digraph $D$ with minimum out-degree $\delta^+(D)$ and minimum in-degree $\delta^-(D)$, we define the \emph{minimum semi-degree} of $D$ as $\delta^0(D)=\min\{\delta^+(D),\delta^-(D)\}$.
\begin{lem}[Ghouila-Houri, 1960~\cite{ghouilahouri1960condition}]
\label{lem:directed Dirac}
If $D$ is an $n$-vertex digraph with $\delta^0(D)\ge n/2$, then it has a directed Hamilton cycle.
\end{lem}

\subsection{Supported $\ell$-paths}\label{sec:supported ell paths}
	In this subsection we define the notion of vertex sequences that ``support an $\ell$-path''. Intuitively, we have an $\ell$-path for as long as possible, and require the final few vertices to form a supported set. These are set up so that it is easy to join them up to longer sequences that support an $\ell$-path, and ultimately to a spanning $\ell$-cycle. 
	These will play an important role in our main strategy of splitting our graph into a perfect matching of sequences that support $\ell$-paths that satisfy useful degree conditions, allowing them to combine into a spanning $\ell$-cycle.
	We also define a sequence that supports an extended $\ell$-path, which we will use to construct the required perfect matching of sequences supporting $\ell$-paths. 

	\begin{defn} \label{def:supported-l-paths}
		Let $k \ge 3$, $\ell \in [k-1]$, $t = \floor{\frac{k}{k-\ell}}(k-\ell)$.
		Let $P := w_1 \ldots w_r$ be a sequence of distinct vertices in a $k$-graph $G$. We say that $P$ \emph{supports an $\ell$-path} in $G$ if $r \ge t$, $r$ is divisible by $k-\ell$, $w_{i(k-\ell) + 1} \ldots w_{i(k-\ell)+k}$ is an edge for all integers $i \in [0, \frac{r-t}{k-\ell} - 1]$, and $\{w_{r-t+1}, \ldots, w_r\}$ is a supported set.

		Let $P = w_0 \ldots w_r$ be a sequence of distinct vertices in $G$. We say that $P$ \emph{supports an extended $\ell$-path} if $r \ge t$, $r$ is divisible by $k-\ell$, the sequence $w_1 \ldots w_r$ supports an $\ell$-path, and $w_0 \ldots w_{t-1}$ is a supported set.
	\end{defn}

	We summarise the facts about constructing and connecting sequences supporting $\ell$-paths and extended $\ell$-paths in the following proposition.

	\begin{prop} \label{prop:supported-l-paths}
		Let $G$ be a $k$-graph, where $k \ge 3$, $\ell \in [k-1]$ and $t = \floor{\frac{k}{k-\ell}}(k-\ell)$.
		\begin{enumerate}[label = \rm(\roman*)]
			\item \label{itm:extended-path-1}
				Let $P$ and $Q$ be two sequences of vertices in $G$ of length at least $t$, both of which support an $\ell$-path, and whose vertices are disjoint. If the sequence consisting of the last $t$ vertices in $P$ followed by the first $t-1$ vertices in $Q$ is a tight path, then $PQ$ supports an $\ell$-path. 
			\item \label{itm:extended-path-2}
				Let $P_1, \ldots, P_m$ be sequences of vertices in $G$ of length at least $t$, each of which supports an $\ell$-path, and whose vertices are pairwise disjoint. If $P_iP_{i+1}$ supports an $\ell$-path for every $i \in [m]$ (addition of indices taken modulo $m$) then $P_1 \ldots P_m$ is an $\ell$-cycle.
			\item \label{itm:extended-path-3}
				Let $P = v_0 \ldots v_r$ be a sequence of distinct vertices where $r$ is divisible by $t$. If every subsequence of at most $k$ consecutive vertices that contains at most one vertex $v_i$ with $i$ divisible by $t$ is supported by an edge, then $P$ supports an extended $\ell$-path. 
			\item \label{itm:extended-path-4}
				Let $P$ be a sequence of distinct vertices of length $r$ such that $r = -1 \pmod{k-\ell}$ and $r \ge t-1$, and let $u,v$ be two vertices that do not appear in $P$, such that $uP$ and $Pv$ are tight paths (or supported sets if $r+1 \le k$). Then $uPv$ supports an extended $\ell$-path.
		\end{enumerate}
	\end{prop}
    \begin{proof}
		For \ref{itm:extended-path-1}, say that the length of $P$ is $p$ and that of $Q$ is $q$. Let $PQ = v_1 \dots v_{p+q}$, where $P = v_1 \dots v_p$ and $Q = v_{p+1} \dots v_{p+q}$. Since $P$ and $Q$ both support $\ell$-paths, we know that $p$ and $q$ are divisible by $k-\ell$, implying the same is true of $p+q$. Furthermore, as $q\ge t$, we also know that the last $t$ vertices of $PQ$ correspond to the last $t$ vertices of $Q$, and hence are supported. All that is left to show is that $v_{i(k-\ell) + 1} \ldots v_{i(k-\ell)+k}$ is an edge for all $i\le \frac{p + q - t}{k-\ell} - 1$. Since $P$ and $Q$ support $\ell$-paths (and in particular their length is divisible by $k-\ell$), we only need to restrict our attention to $i \in [\frac{p - t}{k - \ell}, \frac{p}{k - \ell}-1]$. For this, it would suffice to argue that the subsequence $v_{i_1} \dots v_{i_2}$ forms a tight path, where $i_1 = p - t + 1$ and $i_2 = p + \ell$. By assumption, we know that $v_{p - t + 1} \dots v_{p+t-1}$ is a tight path, and we know $t-1\ge \ell$ (see \Cref{obs:useful t info}), which completes the proof.

		Next, for \ref{itm:extended-path-2}, let $v_1 \dots v_p$ denote the sequence $P= P_1 \dots P_m$. Since each $P_i$ supports an $\ell$-path, we know that $p$ is divisible by $k-\ell$. It suffices to show that $e_j := v_{j(k - \ell) + 1} \dots v_{j(k - \ell) + k}$ is an edge for all $j \ge 0$, where we view the indices cyclically modulo $p$. Crucially, since $|P_i|$ is divisible by $k-\ell$ and at least $t$ for every $i$, for each $e_j$, if $e_j$ starts in $P_i$ then the suffix of the sequence $P_iP_{i+1}$ starting at $e_j$ has length at least $k-\ell + t \ge k$, showing that $e_j$ is a subsequence of $P_iP_{i+1}$.
		Furthermore, as the length of $P_1 \dots P_{i-1}$ is divisible by $k - \ell$, and since $P_i P_{i+1}$ supports an $\ell$-path by \cref{itm:extended-path-1}, we conclude that $e_j$ is indeed an edge.

		We now move to the \ref{itm:extended-path-3}. By \cref{prop:special vxs in ell paths}, we know that every subsequence of the form $v_{i(k - \ell) + 1} \dots v_{i(k - \ell) + k}$ contains at most one vertex $v_i$ with $i$ divisible by $t$ and is thus forms an edge, by assumption. The sets $v_0 \dots v_{t-1}$ and $v_{r - t + 1} \dots v_r$ each have length $t$ and thus have only one $v_i$ with $i$ divisible by $t$, showing that they are supported in $G$. Thus, $P$ supports an extended $\ell$-path.

		The final item \ref{itm:extended-path-4} is also straightforward. Since $Pv$ is a tight path of length divisible by $k-\ell$ and at least $t$, it clearly supports an $\ell$-path. Furthermore, as $r \ge t-1$, the tight path $uP$ contains at least $t$ vertices, and the first $t$ of these are supported.
    \end{proof}

\subsection{Cleaning lemmas}\label{sec:cleaning}
    In this subsection, we will prove important results about finding substructures with large supported co-degrees in very dense hypergraphs. Broadly speaking, we will show that we can delete a few edges and find a nearly spanning subgraph that has no isolated vertices, and approximately preserves degree conditions. Similar ideas have appeared other papers regarding exact co-degree bounds for Hamilton cycles, such as~\cite{rodl2011dirac,liu2022hamiltonian}, but these were presented in an ad-hoc manner as intermediate steps in their proofs. We provide more general and formalised results about finding large, dense subgraphs with good degree conditions that we believe could prove to be independently useful in other settings as well. 

	Let us the begin with the first result. Given an almost complete $k$-graph, we will drop to an almost complete subgraph that spans almost all the vertices, and has high supported co-degree conditions. For $k = 2$ this can be done easily be removing all vertices of small degree, and it turns out that a similar but more subtle strategy works for higher uniformities. We point out that Halfpap, Lemons and Palmer~\cite[Lemma 18]{halfpap2025positive} have a similar result with essentially the same proof. However, our other two results in this section, namely \cref{cor:dense subgraphs weak,cor:dense subgraphs strong}, are proved using the first one, and yield a stronger conclusion than that of \cite[Lemma 18]{halfpap2025positive}. We expect these to be applicable in other settings as well.

\begin{lem}
	\label{lem:dense subgraphs}
	Let $1/n \ll \eps\ll \delta \ll 1/k\le 1/2$. Suppose $F$ is an $n$-vertex $k$-graph with $e(F)\ge (1-\eps)\binom{n}{k}$. Then there exists a subgraph $F'\subseteq F$ with no isolated vertices such that $\delta^*(F') \ge (1 - \delta)n$. In particular, $|F'| \ge (1 - \delta)n$ and $e(F') \ge (1 - k\delta)\binom{n}{k}$.
\end{lem}

\begin{proof}
Let $\mu$ satisfy $\eps \ll \mu \ll \delta$. Let us describe how we obtain $F'$ from $F$. We construct a sequence of hypergraphs $F_{k},\dots, F_1$, where each $F_i$ is $i$-uniform. First, we define $F_k=F$. Then, for $i\ge 2$, given $F_{i+1}$, we let $F_i$ be the subgraph of $\partial^{i}{F_{i+1}}$ with $V(F_i)=V(F)$ which contains every $i$-edge $S$ satisfying $d_{F_{i+1}}^1(S)\ge (1-\mu)n$. Finally, in the graph $F_2$, we remove all vertices $v$ such that $d_{F_2}^1(v)\le (1-\mu)n$, and let $F_1$ be the remaining vertex set. Now, we let $V(F')=F_1$, and set $E(F')\subseteq \binom{V(F')}{k}$ to consist of all $k$-sets $S$ which have $\binom{S}{i}\subseteq E(F_i)$ for all $i\in[k]$, that is, every $i$-subset of $S$ is an edge of $F_i$. 

We show that $F'$ satisfies the statement of~\Cref{lem:dense subgraphs} in two parts.
\begin{claim}
\label{claim:dense subgraph vxs}
We have $|F'|\ge (1-2^{k-1}\eps/\mu^{k-1})n \ge (1-\delta)n$.
\end{claim}
\begin{proof}
Our proof proceeds by successively bounding the number of edges in each $F_i$ from below. We know that $e(F_k)=e(F)\ge (1-\eps)\binom{n}{k}$. We define $\mu_i=2^{k-i}\eps/\mu^{k-i}$ for all $i\in[k]$. We claim that, for all such $i$, we have 
\[
e(F_i)\ge (1-\mu_i)\binom{n}{i}.
\]
Indeed, suppose we know that this is true for $F_k,\dots, F_{i+1}$. Write $f := e(F_i)$, so that $f$ is the number of supported $i$-sets $S$ in $F_{i+1}$ that have $d_{F_{i+1}}^1(S)\ge (1-\mu)n$. Consequently, by considering a degree sum over all supported $i$-sets in $F_{i+1}$ and splitting the sum into $i$-sets that are preserved in $F_i$ and those that are not, we see that
\[
\binom{i+1}{i}(1-\mu_{i+1})\binom{n}{i+1}\le fn+\left(\binom{n}{i}-f \right)(1-\mu)n.
\]
Using $\binom{i+1}{i}\binom{n}{i+1}=n\binom{n-1}{i}$ and rearranging, we see that
\begin{align*}
\mu f&\ge \mu\binom{n}{i}-\mu_{i+1}\binom{n-1}{i}-\left(\binom{n}{i}-\binom{n-1}{i} \right)\\
&= \mu\binom{n}{i} - \mu_{i+1}\binom{n-1}{i} - \binom{n-1}{i-1} \\
&\ge \mu\binom{n}{i}-2\mu_{i+1}\binom{n}{i},
\end{align*}
where we use $\binom{n}{i} = \binom{n-1}{i} + \binom{n-1}{i-1}$ and $\binom{n-1}{i-1}\le \mu_{i+1}\binom{n}{i}$ for $n\gg k\ge i$. Dividing throughout by $\mu$ completes the proof of the claim.
In particular, these calculations show that $|F'|=|F_1|\ge (1-\mu_1)n\ge (1-2^{k-1}\eps/\mu^{k-1})n$ as desired.
\end{proof}

\begin{claim}
\label{claim:dense subgraph codegrees}   
All supported sets $S$ in $F'$ of size at most $k-1$ satisfy $d_{F'}^1(S)\ge (1-2^k\mu)n\ge (1-\delta)n$.
\end{claim}
\begin{proof}
Consider a supported set $S$ of size at most $k-1$. 
Notice that a vertex $s \in V(F') \setminus S$ is in $N^1_{F'}(S)$ if and only if $s \in N^1_{F_{|S'|+1}}(S')$ for every non-empty subset $S' \subseteq S$. Since every non-empty subset $S' \subseteq S$ satisfies $S' \in E(F_{|S'|})$, by definition of $F_{|S'|}$ it satisfies $d^1_{F_{|S'|+1}}(S') \ge (1 - \mu)n$.
This means that every subset $S' \subseteq S$ disallows at most $\mu n$ vertices, and moreover we discard at most $(2^{k-1}\eps/\mu^{k-1})n$ vertices to obtain $F_1$. Hence, the number of infeasible vertices $s$ is at most $(2^{k-1}\eps/\mu^{k-1})n+2^{k-1}\mu n\le 2^k\mu n \le \delta n$, as claimed.
\end{proof}
The last claim shows that $\delta^*(F') \ge (1 - \delta)n$, and by choice of $F'$ it has no isolated vertices. This completes the proof of the main assertion of the lemma. For the lower bound on $e(F')$, we use \cref{prop:codeg to deg} and find that every vertex in $F'$ is in at least $(1 - \delta)^{k-1}\frac{n^{k-1}}{(k-1)!} \ge (1 - k\delta) \binom{n}{k-1}$ edges, showing that $e(F') \ge \frac{1}{k}(1-k\delta)n \binom{n}{k-1} \ge (1 - k\delta)\binom{n}{k}$, as claimed.
\end{proof}

Next, we prove an analogue of \Cref{lem:dense subgraphs} where the host hypergraph is any arbitrary $k$-graph.
Specifically, we show that if the host graph $F$ has linear minimum supported co-degree, and $F'$ is formed by removing few edges from $F$, then there is a subgraph $F''$ of $F$, which is obtained by removing few edges, and where for every supported set $S$ in $F''$ of size at most $k-1$, the degree of $S$ in $F''$ is almost as large as it is in $F$.

\begin{cor} \label{cor:dense subgraphs weak}
	Let $1/n \ll \eps \ll \delta \ll 1/k \le 1/2$. 
	Let $F$ be an $n$-vertex $k$-graph with and let $F' \subseteq F$ be a subgraph satisfying $e(F') \ge e(F) - \eps n^k$.
	Then there is a subgraph $F'' \subseteq F'$ such that the following properties hold.
	\begin{itemize}
		\item
			$F''$ has no isolated vertices,
		\item
			$|F''| \ge (1 - \delta)n$,
		\item
			$e(F'') \ge e(F) - \delta n^k$,
		\item
			$d_{F''}(S) \ge d_{F}(S) - \delta n^{k-|S|}$ for every set $S$ of size at most $k-1$ which is supported in $F''$.
	\end{itemize}
\end{cor}

\begin{proof}
	Let $\eps'$ and $\mu$ satisfy $\eps \ll \eps' \ll \mu \ll \delta$. Let $H$ be the complement of $F$, that is, the $k$-graph on $V(F)$ whose edges are the $k$-sets that are non-edges in $F$.
	Write $G := H \cup F$ and $G' := H \cup F'$.
	Apply \Cref{lem:dense subgraphs} to $G'$ with parameters $\eps'$ and $\mu$, noting that $e(G') \ge \binom{n}{k} - (e(F) - e(F')) \ge \binom{n}{k} - \eps n^k \ge (1 - \eps') \binom{n}{k}$. Denote the resulting hypergraph by $G''$.
	Let $V(F'') = V(G'')$ and $E(F'') = E(G'') \setminus E(H)$.
	We claim that $F''$ satisfies the requirements of the corollary.
	Indeed, let $S$ be a supported set in $F''$ of size $s \le k-1$. 
	Then, from \cref{prop:codeg to deg}, we get
	\begin{equation*}
		d_{G''}(S) \ge \frac{((1 - \mu)n)^{k-s}}{(k-s)!}
		\ge \binom{n-s}{k-s} - \delta n^{k-s}
		= d_H(S) + d_F(S) - \delta n^{k-s}.
	\end{equation*}
	Since $d_{F''}(S) = d_{G''}(S) - d_{H}(S)$ we get $d_{F''}(S) \ge d_F(S) - \delta n^{k-s}$, as required for the fourth item. The other items follow directly from the choice of $G''$ and $F''$.
\end{proof}

Finally, we obtain a similar corollary which shows that if the host graph $F$ has linear minimum supported co-degree, and $F'$ is formed by removing few edges from $F$, then there is a subgraph $F''$ of $F$, which is obtained by removing few edges, and where for every supported set $S$ in $F''$ of size at most $k-1$, the vertex co-degree of $S$ in $F''$ is almost as large as it is in $F$.

\begin{cor} \label{cor:dense subgraphs strong}
	Let $1/n \ll \eps \ll \delta \ll \alpha \ll 1/k \le 1/2$. 
	Let $F$ be an $n$-vertex $k$-graph with minimum supported co-degree at least $\alpha n$, and let $F' \subseteq F$ be a subgraph satisfying $e(F') \ge e(F) - \eps n^k$.
	Then there is a subgraph $F'' \subseteq F'$ such that the following properties hold.
	\begin{itemize}
		\item
			$F''$ has no isolated vertices,
		\item
			$|F''| \ge (1 - \delta)n$,
		\item
			$e(F'') \ge e(F) - \delta n^k$,
		\item
			$\ds_{F''}(S) \ge \ds_{F}(S) - \delta n$, for every set $S$ of at most $k-1$ vertices that is supported in $F''$.
	\end{itemize}
\end{cor}

\begin{proof}
	Apply \Cref{cor:dense subgraphs weak} to $F'$ with parameters $\eps$ and $\mu$, where $\eps \ll \mu \ll \delta$, and denote the resulting graph $F''$.
	The graph $F''$ clearly satisfies the first three properties above, so it suffices to prove the fourth item.
	Suppose that for some set of vertices $S$ of size $s \le k-1$, which is supported in $F''$, we have $\ds_{F''}(S) \le \ds_F(S) - \delta n$. This means that there is a set $X$ of size $\delta n$ such that, for every $x \in X$, $S \cup \{x\}$ is supported in $F$ but not in $F''$. 
	Thus, using \cref{prop:codeg to deg} alongside the minimum supported co-degree assumption on $F$, we have
	\begin{align*}
		d_{F}(S) - d_{F''}(S) 
		\ge \big|\{e \in E(F) : S \cup \{x\} \subseteq e \text{ for some $x \in X$}\}\big|
		& \ge \frac{|X| \cdot (\alpha n)^{k-s-1}}{(k-s)!} \\
		& \ge \delta \cdot \alpha^{k-s-1} \binom{n}{k-s}
		> \mu n^{k-s},
	\end{align*}
	a contradiction to the choice of $F''$ according to \Cref{cor:dense subgraphs weak}.
\end{proof}

\subsection{Perfect matchings in almost complete $k$-partite $k$-graphs}\label{sec:random matching}
Our main aim in this subsection is to prove \Cref{cor:k-partite-matching} below, which asserts that, if given a balanced $k$-partite $k$-graph $H$ with large minimum supported co-degree and no isolated vertices, and a not-too-large family $\cF$ of subgraphs of $H$ that are almost complete, there is a perfect matching $M$ such that for every $F \in \cF$, almost every edge of $M$ is in $F$.

We will use the following consequence of Hall's theorem. 
\begin{prop}
\label{prop:bipartite Hall matching}
Let $G=X\sqcup Y$ be a bipartite graph with $|X|=|Y|=n$. If $\delta(G)\geq n/2$, then $G$ has a perfect matching.
\end{prop}
\begin{proof}
Consider any subset $S\subseteq X$. If we can show that $N_G(S)$ has at least $|S|$ vertices, then we are done by Hall's theorem.

Clearly $|N_G(S)|\geq \delta(G)\geq n/2$ for all $S$. So if $|S|\leq n/2$, we are done. Hence, we may suppose $|S|>n/2$. Consider any $y\in Y$. Since $d_G(y)\geq n/2>|X\setminus S|$, the vertex $y$ must have a neighbour in $S$. Since $y$ is simply an arbitrary vertex of $Y$, we conclude that $N_G(S)=Y$, and clearly $|Y|=|X|\geq |S|$.
\end{proof}

The main content of the proof of the aforementioned result about matchings in $k$-partite $k$-graph is the following lemma, which resolves the graph case, i.e.\ when $k = 2$. The matching is found in randomly, thus allows us to deduce that almost all edges of $M$ are in $F$, for a not-too-large collection of subgraphs $\cF$ of $H$ that are almost complete.

\begin{lem}
	\label{lem:random matching key lemma abstract}
	Let $1/n \ll \eps \ll \delta \ll 1$.
	Let $H$ be a balanced bipartite graph on $2n$ vertices with minimum degree at least $(1 - \eps)n$, and let $\cF$ be a collection of subgraphs of $H$, each with at least $(1-\eps) n^2$ edges, such that $|\cF| = e^{o(\sqrt{n})}$. Then, there is a perfect matching $M$ in $H$ such that $|M \cap E(F)| \ge (1-\delta) n$ for every $F \in \cF$. 
\end{lem}
\begin{proof}
	Fix parameters $\alpha$ and $\beta$ such that $\eps \ll \alpha \ll \beta \ll \delta$. 
	Denote the two parts of $H$ by $X$ and $Y$. 
	Let $x_1, \ldots, x_m$ be a sequence of $m := (1 - \beta)n$ distinct vertices in $X$.
	We sequentially pick vertices $y_1,\dots,y_{m}\in Y$ such that each $y_i$ is a uniformly random neighbour of $x_i$ in $Y$, and is distinct from $y_1,\dots, y_{i-1}$. This is always possible since $d(x)\geq (1-\eps) n \ge m$ for all $x\in X$.

	Let $X'$ and $Y'$ be the remaining unmatched vertices of $X$ and $Y$, respectively. By choice, we have $|X'|=|Y'|=\beta n$. 
	Notice that the minimum degree in $H[X', Y']$ is at least $(1 - \eps)n - (1 - \beta)n = (\beta - \eps) n \ge \beta n / 2 = |X'|/2$. Thus, by \Cref{prop:bipartite Hall matching}, there is a perfect matching $M'$ in $H[X', Y']$. Let $M$ be the union of the matching $\{x_1y_1, \dots, x_my_m\}$ with $M'$.
	Then $M$ is a perfect matching in $H$. 

	Fix $F \in \cF$. We will show that $|M \cap E(F)| \ge (1 - \delta) n$, with probability at least $1 - e^{-\Omega(\sqrt{n})}$. A union bound will then conclude the proof.

	Colour the edges of $F$ blue and the remaining edges in $H$ red.
	By assumption, there are at most $\eps n^2$ red edges. 
	We say that a vertex $x \in X$ is \emph{scared} if its red degree into $Y$ is at least $\alpha n$, and we will call it \emph{happy} otherwise. Summing the red degrees of these vertices and using the bound on the number of red edges, we see that the number of scared vertices in $X$ is at most $\eps n/\alpha$.

	For each $j \le m$ such that $x_j$ is happy, let $p_j$ be the probability that $x_j y_j$ is a red edge. Since $x_j$ is happy, we have
	\begin{align*}
		p_j
		\leq \frac{\alpha n}{(1-\eps) n-(j-1)}
		\leq \frac{\alpha n}{(1-\eps) n- m}
		= \frac{\alpha n}{(1-\eps) n- (1 - \beta)n}
		\le \frac{2\alpha}{\beta}.
	\end{align*}
	If we let $R$ denote the random variable corresponding to the number of red edges in $\{x_jy_j : j \in [m], \text{ $x_j$ is happy}\}$, then clearly $R$ is stochastically dominated by a binomial random variable $\hat{R}$ with parameters $(n,2\alpha/\beta)$.
	Hence, using Chernoff bounds (\Cref{lem:Chernoff}), we have
	\begin{align*}
		\bP[R \ge 2\alpha n/\beta+n^{3/4}] 
		& \le \bP[\hat{R} \ge 2\alpha n/\beta+n^{3/4}] \\
		&= \bP[\hat{R} \geq \bE \hat{R}+n^{3/4}]\\
		&\le 2\exp\left(-\frac{n^{3/2}}{6n \alpha/\beta}\right)
		=e^{-\Omega(\sqrt{n})}.
	\end{align*}

	Since the number of red edges in $M$ is at most $R + \beta n + \eps n/\alpha$, we have that, with probability at least $1 - e^{-\Omega(\sqrt{n})}$, the number of red edges in $M$ is at most
	\[
		\frac{2\alpha n}{\beta}+n^{3/4} + \beta n + \frac{\eps n}{\alpha} \le \delta n.
	\]
	Thus, with probability at least $1 - e^{-\Omega(\sqrt{n})}$, $|M \cap E(F)| \ge (1 - \delta)n$.
	By a union bound, with high probability this holds for all $F \in \cF$.
\end{proof}

The next lemma deals with $k$-partite $k$-graphs. In particular, given a collection of subgraphs satisfying some high supported co-degree conditions, we wish to find a matching that shares almost all its edges with each subgraph in the collection. 

\begin{lem}
\label{lem:k-partite matching}
Let $1/n\ll \eps\ll \delta\ll 1/k\le 1/3$. Let $H$ be a $k$-partite $k$-graph with parts $X_1,\dots,X_k$, each having size $|X_i|=n$. Suppose that $H$ has no isolated vertices and, for all $i\in [k]$, every supported set $S$ in $H$ which is disjoint from $X_i$ satisfies $d^1_{X_i}(S)\ge (1-\eps)n$. Let $\cF$ be a collection of subgraphs satisfying the following conditions:
\begin{itemize}
    \item There exists a constant $c=c(k)$ such that $|\cF|\le n^c$,
    \item each $F \in \cF$ has no isolated vertices, and
    \item for any $F\in \cF$, every supported set $S$ of $F$ which is disjoint from $X_i$ satisfies $d^1_{X_i}(S)\ge (1-\eps)n$. 
    \end{itemize}
Then there exists a perfect matching $M$ in $G$ such that $|M\cap E(F)|\ge (1-\delta)n$ for all $F\in \cF$.
\end{lem}

We construct each edge one vertex at a time. In the first step, we form a matching of supported pairs between $X_1$ and $X_2$. We then extend this to a matching of supported triples across $X_1$, $X_2$ and $X_3$, and so on. In each step, we will extend our existing matching using \Cref{lem:random matching key lemma abstract} and maintain certain properties that will allow us to continue the process in the next iteration.

\begin{proof}
	Let $\gamma_1, \ldots, \gamma_k$ be parameters satisfying $\eps = \gamma_1 \ll \ldots \ll \gamma_k = \delta$.
	We will first prove that for every $i \in [k-1]$ there is a perfect matching $M_i$ in $\partial^iH[X_1 \cup \ldots \cup X_i]$ such that for every $F \in \cF \cup \{H\}$ and $v \in V(F) \setminus (X_1 \cup \ldots \cup X_i)$, there are at least $(1 - \gamma_i)n$ edges $S$ in $M_i$ such that $S \cup \{v\}$ is supported in $F$. 

	We prove this by induction on $i$. 
	Notice that we can take $M_1 = X_1$, using that $F$ has no isolated vertices, which implies that any $v \in V(F) \setminus X_1$ is supported and thus $d^1_{X_1}(v) \ge (1 - \eps)n = (1 - \gamma_1)n$.
	For $i \in [2,k-1]$, suppose that $M_{i-1}$ is a perfect matching in $\partial^iH[X_1 \cup \ldots \cup X_{i-1}]$ satisfying the requirements.

	Let $T_i$ be the bipartite graph on $M_{i-1} \sqcup X_i$, where $Sx$, with $S \in M_{i-1}$ and $x \in X_i$, is an edge if $S \cup \{x\}$ is supported in $H$.
	We claim that $\delta(T_i) \ge (1 - \gamma_{i-1})n$.
	Indeed, first notice that, by choice of $M_{i-1}$, every $S \in M_{i-1}$ is supported in $H$ and thus satisfies $d^1_{X_i}(S) \ge (1 - \eps)n$, which exactly means that $d_{T_i}(S) \ge (1 - \eps)n \ge (1 - \gamma_{i-1})n$.
	Now, given $x \in X_i$, recall that $M_{i-1}$ was chosen so that for at least $(1 - \gamma_{i-1})n$ sets $S \in M_{i-1}$ we have that $S \cup \{x\}$ is supported in $H$, which precisely means that $d_{T_i}(x) \ge (1 - \gamma_{i-1})n$.

	For every $F \in \cF \cup \{H\}$ and every $v \in V(F) \setminus (X_1 \cup \ldots \cup X_i)$ let $T_{i,F,v}$ be the subgraph of $T_i$ where $Sx$, with $S \in M_{i-1}$ and $x \in X_i$, is an edge if $S \cup \{x,v\}$ is supported in $F$.
	We claim that $e(T_{i,F,v}) \ge (1 - 2\gamma_{i-1})n^2$.
	Indeed, by choice of $M_{i-1}$, we have that at least $(1-\gamma_{i-1})n$ sets $S$ in $M_{i - 1}$ are such that $S \cup \{v\}$ is supported in $F$. For every such $S$, since $S \cup \{v\}$ avoids $X_i$ and by assumption on $F$, this implies that $d^1_{X_i}(S) \ge (1 - \eps)n$, meaning that $d_{T_{i,F,v}}(S) \ge (1 - \eps)n$.
	Thus $e(T_{i,F,v}) \ge (1 - \gamma_{i-1})n \cdot (1 - \eps)n \ge (1 - 2\gamma_{i-1})n^2$.

	Apply \Cref{lem:random matching key lemma abstract} with the graph $T_i$ and the family $\{T_{i,F,v} : F \in \cF \cup \{H\}, v \in V(F) \setminus (X_1 \cup \ldots \cup X_i)\}$ (which has size at most $(|\cF| + 1) \cdot kn \le n^{c+2}$). This yields a perfect matching $M$ in $T_i$ such that $|M \cap E(T_{i,F,v})| \ge (1 - \gamma_i)n$ for every relevant $F$ and $v$.
	Form $M_i$ by including $S \cup \{x\}$ for every edge $Sx$ in $M$ (with $S \in M_{i-1}$ and $x \in X_i$).
	It is easy to check that $M_i$ is a perfect matching in $\partial^iH[X_1 \cup \ldots \cup X_i]$ that has at least $(1 - \gamma_i)n$ edges $S$ such that $S \cup \{v\}$ is supported in $F$, for every relevant $F$ and $v$, as required.

	Finally, suppose that we found a matching $M_{k-1}$ with the desired properties.
	We proceed similarly to the above to find a matching $M_k$ that satisfies the requirements of the lemma. Define $T_k$ as above, so that $\delta(T_k) \ge (1-\gamma_{k-1})n$.
	For every $F \in \cF$, let $T_{k,F}$ be the subgraph of $T_k$ where $Sx$ (with $S \in M_{k-1}$ and $x \in X_k$) is an edge if $S \cup \{x\}$ is an edge in $F$.
	By choice of $M_{k-1}$, for every $x \in V(F) \cap X_k$, it contains at least $(1 - \gamma_{k-1})n$ edges $S$ such that $S \cup \{x\}$ is supported in $F$, which shows (using that $|V(F) \cap X_k| \ge (1 - \eps)n$ which follows implicitly by the third and fourth items in the statement) that $e(T_{k,F}) \ge (1 - \eps)n \cdot (1-\gamma_{k-1})n$.
	As above, apply \Cref{lem:random matching key lemma abstract} to obtain a perfect matching $M$ in $T_k$ such that $|M \cap T_{k,F}| \ge (1 - \gamma_k)n$ for every $F \in \cF$.
	This corresponds to a perfect matching $M_k$ in $H$ satisfying $|M_k \cap E(F)| \ge (1 - \gamma_k)n = (1 - \delta)n$ for every $F \in \cF$, as required.
\end{proof}

Finally, we deduce the following corollary, where the requirements on each $F \in \cF$ are replaced by $F$ being an almost complete $k$-partite $k$-graph. This is done by applying the `cleaning' lemma \Cref{cor:dense subgraphs strong}. The concrete upper bound given on the size of $\cF$ is based on what we need in a later section.

\begin{cor} \label{cor:k-partite-matching}
\label{cor:final k-partite matching}
Let $1/n\ll \eps\ll \delta\ll 1/k\le 1/3$. Let $H$ be a $k$-partite $k$-graph with parts $X_1,\dots,X_k$, each having size $|X_i|=n$. Suppose that $H$ has no isolated vertices and, for all $i\in [k]$, every supported set $S$ in $H$ which is disjoint from $X_i$ satisfies $d^1_{X_i}(S)\ge (1-\eps)n$. Let $\cF$ be a collection of at most $2n^k$ subgraphs of $H$ satisfying $e(F) \ge (1 - \eps)n^k$.
Then there exists a perfect matching $M$ in $G$ such that $|M\cap E(F)|\ge (1-\delta)n$ for all $F\in \cF$.
\end{cor}

\begin{proof}
	\def \Fc {F_{\mathrm{clean}}}
	Let $\mu$ be a parameter satisfying $\eps \ll \mu \ll \delta$.
	For each $F \in \cF$, apply \Cref{cor:dense subgraphs strong} with the complete $k$-partite $k$-graph on $X_1 \sqcup \ldots \sqcup X_k$ playing the role of $F$, with $F$ playing the role of $F'$, and parameters $\eps$ and $\mu$. Denote the resulting graph $\Fc$.
	Then $\Fc$ has no isolated vertices and $|N^1_{\Fc}(S) \cap X_i| \ge |X_i| - \mu n$ for every supported set $S$ that misses $X_i$. Thus, by \Cref{lem:k-partite matching}, applied with parameter $\mu$ in place of $\eps$ and family $\{\Fc : F \in \cF\}$, a matching with the desired properties exists.
\end{proof}

\section{Structure of extremal hypergraphs}\label{sec:extremal structure} 
Our aim in this section is to modify the set $A$ from the definition of extremal hypergraphs to obtain better control over co-degrees across and within the parts. Our main result in this section is the following.

\begin{prop}
\label{prop:final A' conditions}    
Let $1/n \ll \eps \ll \mu \ll 1/k \le 1/3$, let $\ell \in [k-1]$ and set $t = \floor{\frac{k}{k - \ell}}(k - \ell)$.
Let $G \in \exE_\eps(n,k,\ell)$.
Then there exist subsets $A, A',B' \subseteq V(G)$ and a subgraph $F_B^+ \subseteq \partial^{k-1}[B']$ with vertex set $B'$ that satisfy the following properties.
\begin{enumerate}[label = \rm(\roman*)]
	\item \label{itm:A'-1}
		$V(G)=A' \sqcup B'$ and $A \subseteq A'$,
	\item \label{itm:A'-2}
		$n/t- \mu n \leq |A|,|A'|\leq n/t+\mu n$,
	\item \label{itm:A'-3}
		any supported set $S$ of size at most $k-1$ that intersects $A$ satisfies $d_{B'}^1(S)\geq |B'| - \mu n$, and
	\item \label{itm:A'-4}
		$F_B^+$ has no isolated vertices, $\delta^*(F_B^+) \ge (1 - \mu)|B'|$, and $d_{A}^1(S) \ge |A| - \mu n$ for every $S \in E(F_B^+)$.
\end{enumerate}
\end{prop}

This section will be devoted to the proof of \Cref{prop:final A' conditions}. 
Throughout this section, we fix parameters
\begin{equation*}
	1/n \ll \eps \ll \eps_A \ll \eps_{k-1} \ll \mu \ll 1/k \le 1/3, 
\end{equation*}
and we subsequently fix $\ell$ and $t$ and a $k$-graph $G$ as in the statement of \cref{prop:final A' conditions}.

Recall that the extremality of $G$ implies that there is a set $A$ on $\floor{n/t}$ vertices that has at most $\eps n^2$ supported pairs. In the next claim we modify $A$ so that every vertex of $A$ supports only a few other vertices of $A$.
\begin{claim}
\label{claim:new A conditions}
	There exist subsets $A, B \subseteq V(G)$ such that the following properties hold.
    \begin{enumerate}[label = \rm(\roman*)]
        \item $V(G)=A \sqcup B$,
        \item $n/t-4\eps n/\eps_A \leq |A|\leq n/t$,
        \item $G$ has at most $\eps n^2$ supported pairs with both vertices in $A$, and
        \item every supported set $S$ of size at most $k-1$ that intersects $A$ satisfies $d_A^1(S)\leq \eps_A n$ and $d_B^1(S)\geq |B|-\eps_A n$.
    \end{enumerate}
\end{claim}
\begin{proof}

	Let $A$ be a set of $\floor{n/t}$ vertices with at most $\eps n^2$ supported pairs (which exists since $G \in \exE_\eps (n, k, \ell)$ being extremal) and let $B = V(G) \setminus A$.
	Let $x_A$ be the number of vertices $u\in A$ with $d_A^1(u)\geq \eps_A n/2$.
	Then the number of supported pairs in $A$ is at least $x_A \eps_A n / 4$, but also at most $\eps n^2$, showing that $x_A \le \frac{4\eps n}{\eps_A}$.
	Transfer these $x_A$ vertices from $A$ to $B$. We claim that the modified sets $A$ and $B$ satisfy the requirements of the claim.

	The first three parts follow directly from the choice of $A$. For the final part, let $S$ be a supported set of size at most $k-1$ with $u \in S \cap A$. Clearly $u$ must support every vertex that $S$ supports, that is, we have $N_A^1(S)\subseteq N_A^1(u)$. This implies
    \[
        d_A^1(S)\leq d_A^1(u)\leq \frac{\eps_A n}{2} \le \eps_A n,
    \]
	where the final inequality comes from the modification of $A$. Then, since $d_G^1(S)=d_A^1(S)+d_B^1(S)$ and $d_G^1(S) \ge (1-1/t)n - k$ (by \Cref{obs:min 1-degree}), we have 
	\begin{align*}
		d_B^1(S) & \ge d_G^1(S) - d_A^1(S)  
		\ge (1 - 1/t)n - k - \frac{\eps_A n}{2} 
		\ge |B| - \frac{4\eps n}{\eps_A} - (k+1) - \frac{\eps_A n}{2} \ge |B| - \eps_A n,
	\end{align*}
	so the final item holds too.
\end{proof}

Next, we tackle supported $(k-1)$-sets in $B$ which have high vertex co-degree into $A$.
\begin{defn}
\label{defn:scarce and rich sets}
    We say that a supported $(k-1)$-set $S \subseteq B$ in $B$ is \emph{$A$-scarce} if $d_A^1(S)\leq |A|-\eps_{k-1} n$. Otherwise, we say that it is \emph{$A$-rich}. 
\end{defn}
Let $x_{k-1}$ be the number of $A$-rich $(k-1)$-sets in $B$. In the next claim we show that almost all $(k-1)$-sets in $B$ are $A$-rich.

\begin{claim}
\label{claim:many A rich sets}
    We have $x_{k-1}\ge (1-\eps_{k-1})\binom{|B|}{k-1}$.
\end{claim}
\begin{proof}
    Let $m$ denote the number of edges of $G$ with one vertex in $A$ and $(k-1)$ vertices in $B$. We will double count $m$, obtain a lower and an upper bound, and then compare the two.

    First, let us lower bound $m$. We will construct an edge with $(k-1)$ vertices in $B$ by repeatedly using Claim~\ref{claim:new A conditions} to build larger and larger supported sets until we form an edge. Start with any vertex $u_1\in A$. Suppose that we have constructed a supported $i$-set $\{u_1,\dots, u_i\}$, with $u_2,\dots,u_i\in B$ and $i\le k-1$. By Claim~\ref{claim:new A conditions}, there are at least $|B| - \eps_A n$ choices for a vertex in $B$ that can be used to extend our set to a supported $(i+1)$-set. Thus, we see that
    \begin{align*}
    	m &\geq \frac{|A|}{(k-1)!}\left(|B|-\eps_A n\right)^{k-1}\\
    	&\geq \frac{|A|}{(k-1)!}\left(|B|^{k-1}-(k-1)\eps_A n |B|^{k-2}\right).
    \end{align*}
    Next, we bound $m$ from above. We do this by looking at all possible supported $(k-1)$-sets in $B$, and then counting the maximum possible number of vertices in $A$ that each such $(k-1)$-set could support. Hence, we have
    \begin{align*}
        m&\le x_{k-1}|A|+\left(\binom{|B|}{k-1}-x_{k-1}\right)(|A|-\eps_{k-1}n) \\
        &\le x_{k-1}\eps_{k-1}n+(|A|-\eps_{k-1}n) \cdot \frac{|B|^{k-1}}{(k-1)!}.
    \end{align*} 
    Finally, we compare the two bounds
    \begin{align*}
        x_{k-1}\eps_{k-1}n 
    	&\ge \frac{\eps_{k-1}n |B|^{k-1}- (k-1)\eps_{A}n|A||B|^{k-2}}{(k-1)!} \\
    	&\ge n \cdot \frac{|B|^{k-1}}{(k-1)!} \cdot (\eps_{k-1}- (k-1)\eps_{A}) \\
    	&= n \cdot \binom{|B|}{k-1} \cdot (\eps_{k-1}- (k-1)\eps_{A}) + O(n^{k-1}) \\
    	&\ge n \cdot \binom{|B|}{k-1} \cdot (\eps_{k-1}- k\eps_{A}),
    \end{align*}
    where we use $|A|\le |B|$ (since $|A|\le n/t\le n/2$). Simplifying further yields
    \begin{align*}
    	x_{k-1} \ge \binom{|B|}{k-1}\left(1 - \frac{k\eps_A}{\eps_{k-1}}\right)
    	\ge \binom{|B|}{k-1}(1 - \eps_{k-1}),
    \end{align*}
    proving the claim.
\end{proof}

We are now ready to prove \Cref{prop:final A' conditions}, using the last claim and one of the cleaning lemma from the previous section, and moving a few exceptional vertices from $B$ to $A$.
\begin{proof}[Proof of \Cref{prop:final A' conditions}]

	Apply \Cref{lem:dense subgraphs} to the subgraph $F'$ of $\partial^{k-1}[B]$ consisting of $A$-rich supported $(k-1)$-sets with parameters $\eps_{k-1}$ and $\mu/2$ in place of $\eps$ and $\delta$ respectively, to obtain a subgraph $\Fp_B \subseteq \partial^{k-1}[B]$ that has no isolated vertices and satisfies $\delta^*(\Fp_B) \ge (1 - \mu/2)|B|$; the lemma is applicable due to \Cref{claim:many A rich sets} above. In particular, we have $|\Fp_B| \ge (1 - \mu/2)|B|$.
	Define $B' = V(\Fp_B)$ and $A' := V(G) \setminus B'$. We now prove that Properties \ref{itm:A'-1} to \ref{itm:A'-4} hold. 

	Item \ref{itm:A'-1} is immediate from the choice of $A'$.
	For the lower bound in \ref{itm:A'-2}, note that $|A'| \ge |A| \ge n/t - 4\eps n / \eps_A \ge n/t - \mu n$, using \Cref{claim:new A conditions}.
	For the upper bound we have $|B'| \ge (1 - \mu)|B| \ge |B| - \mu n$, and thus $|A'| = n - |B'| \le n - |B| - \mu n = |A| - \mu n \le n/t - \mu n$, using \Cref{claim:new A conditions} again.
	For \ref{itm:A'-3}, let $S$ be a supported set of size at most $k-1$ that intersects $A$. Then $d^1_A(S) \le \eps_A n$ and $d^1_B(S) \ge |B| - \eps_A n$ by \Cref{claim:new A conditions}. Because $|A' \setminus A| \le \mu n/2$, it follows that $d^1_A(S) \le \eps_A n + \mu n/2 \le \mu n$ and $d^1_{B'}(S) \ge |B'| - \eps_A n - \mu n/2 \ge |B'| - \mu n$.
	Finally, \ref{itm:A'-4} follows from the choice of $\Fp_B$.
\end{proof}

\section{Balancing the sizes of $A'$ and $B'$}\label{sec:balancing}
In the following two sections (especially in the next one) we will use the notions of vertex sequences that support an $\ell$-path or an extended $\ell$-path; see \Cref{def:supported-l-paths}. We will also use the convention that a tight path on $r < k$ vertices is just an ordered supported set of size $r$. 

Our aim in this section is to prove the following proposition that finds a collection $\cP$ of vertex-disjoint sequences that support extended $\ell$-paths, cover all of $A'$, and `balance' the sizes of $A'$ and $B'$, where $A'$ and $B'$ are as given by \Cref{prop:final A' conditions}.

\begin{prop}
\label{prop:size balancing final}
Let $1/n \ll \eps \ll \mu \ll 1/k \le 1/3$, $\ell \in [k-1]$ such that $(k, \ell) \ne (3,1)$, and $t = \floor{\frac{k}{k-\ell}}(k-\ell)$. Suppose that $G \in \exE_\eps(n,k,\ell)$ and let $A,A',B'$ be as given by \Cref{prop:final A' conditions}.
    Then there is a collection $\cP$ of pairwise vertex-disjoint sequences that support extended $\ell$-paths in $G$ such that:
	\begin{enumerate}[label = \rm(\roman*)]
        \item\label{itm:P-1} Every $P\in \cP$ starts and ends in $A$,
		\item\label{itm:P-2} $A' \subseteq V(\cP)$,
		\item\label{itm:P-3} $|\cP| \ge n/t - 4\mu n$,
		\item\label{itm:P-4} $|V(\cP)| \le n/t + 4t\mu n$, 
        \item\label{itm:P-5} $(t-1)|\cP| = |B' \setminus V(\cP)|$,
		\item\label{itm:P-6} every $P \in \cP$ satisfies $|P| = 1$ (so $P$ is a singleton from $A$) or $|P| \ge t+1$.
	\end{enumerate}
\end{prop}

Throughout this section, fix parameters $n, \eps, \mu, k, \ell, t$, a $k$-graph $G$, the sets $A,A',B'$ and the $(k - 1)$-graph $\Fp_B$ as in the statement of \Cref{prop:final A' conditions}, and additionally require $(k, \ell) \ne (3,1)$.
As an intermediate step towards proving \Cref{prop:size balancing final}, we find a collection $\cP'$ with slightly weaker requirements. 

\begin{prop} 
\label{prop:size balancing}
	There is a collection $\cP'$ of pairwise disjoint vertex sequences in $V(G)$ that supported extended $\ell$-paths such that:
	\begin{enumerate}[label = \rm(\roman*)]
		\item\label{itm:P'-1} Every $P \in \cP'$ is either a singleton from $A'$ or satisfies $|P| = t+1$ and starts and ends in $A$, and the number of non-singleton sequences $P$ is at most $\mu n$,
        \item\label{itm:P'-2} $A' \subseteq V(\cP')$,
		\item\label{itm:P'-3} $|\cP'| \ge n/t - \mu n$,
		\item\label{itm:P'-4} $|V(\cP')| \le n/t + t \mu n$,
		\item\label{itm:P'-5} $0 \le |B' \setminus V(\cP')| - (t-1)|\cP'| \le t-1$.
	\end{enumerate}
\end{prop}

In order to prove Proposition~\ref{prop:size balancing}, we first need to establish some preliminary machinery. Let us write $|A'| = \floor{n/t} + x$, which implies $|B'|=n-\floor{n/t}-x$. If $x \le 0$ then we can just let $\cP$ consist of the set of all vertices in $A'$, and this would prove Proposition~\ref{prop:size balancing} (where we use \cref{prop:final A' conditions}~\ref{itm:A'-2} with parameter $\mu/2$ to bound $x$ and conclude~\ref{itm:P'-5}). So we may suppose that $x > 0$. From \cref{prop:final A' conditions}~\ref{itm:A'-2}, we know $x\leq \mu n/2$.
We claim that $\partial^2[A']$ has minimum degree at least $x+1$. Indeed, consider $u \in A'$, and notice that $u$'s degree in $\partial^2[A']$ is $d^1_{A'}(u)$, which satisfies the following, using \Cref{obs:min 1-degree}.
\begin{align*}
	d^1_{A'}(u) = d_G^1(u) - d_{B'}^1(u)
	\ge \left(n - \floor{\frac{n}{t}} + 1\right)- |B'|
	= |A'| - \floor{\frac{n}{t}} + 1
	= x + 1,
\end{align*}
as claimed.
Hence, by \Cref{lem:disjoint edges cherries} and $|A| \ge |A'| - \mu n > |A'|/2 + 5x$, there is a subgraph $M_{A'}$ of $\partial^2[A']$ that consists of $x$ pairwise vertex-disjoint cherries with leaves in $A$.

\begin{rem}\label{rem:exact degree}
    A critical point to note is that this is the only part of the entire proof that utilises the exact bound $\delta^*(G) \ge n - \floor{\frac{n}{t}} - (k - 3)$ via \cref{obs:min 1-degree,lem:disjoint edges cherries}. As mentioned in \Cref{sec:intro}, we can prove \cref{thm:Hamilton cycle extremal} and consequently \cref{thm:main} with the weaker bound of $\delta^*(G) \ge n - \floor{\frac{n}{t}} - (k - 2)$ instead, for most values of $k$ and $\ell$ (namely when $t \ge \ell+2$ or when $\ell = k-1$) which is achieved by modifying \cref{lem:disjoint edges cherries} to potentially allow for an additional cherry or some edges. We discuss how to suitably adapt the rest of the proof in \Cref{sec:improving the bound}.
	%%%% change this part for journal version
\end{rem}

Moving on, we extend the cherries in $M_{A'}$ into sequences of order $t+1$ that support extended $\ell$-paths, as follows.
Let $u_1 u_2 u_3$ denote a cherry of $M_{A'}$ with middle vertex $u_2$. We wish to find distinct vertices $v_1,\dots,v_{t-2}\in B'$ such that $\{u_1, u_2, v_1, \dots, v_{t-2}\}$ and $\{u_2,u_3, v_1, \dots, v_{t-2}\}$ are supported sets in $G$. Given such $v_1, \ldots, v_{t-2}$, we will call the sequence $u_1 u_2 v_1 \dots v_{t-2} u_3$ an \emph{extended cherry} built from the cherry $u_1 u_2 u_3$ of $M_{A'}$ with \emph{extension} $v_1\dots v_{t-2}$. Notice that, by \Cref{prop:supported-l-paths}~\ref{itm:extended-path-4}, extended cherries support extended $\ell$-paths.

\begin{claim}
\label{claim:disjoint extended cherries}    
    For each cherry $u_1 u_2 u_3\in M_{A'}$, there exists an extension $v_1\dots v_{t-2}$ consisting of vertices in $B'$ such that the extensions for distinct cherries are all pairwise vertex-disjoint. 
\end{claim}
\begin{proof}
    For any cherry $u_1 u_2 u_3\in M_{A'}$ with middle vertex $u_2$, we know that $u_1, u_3 \in A$, and therefore $\{u_1,u_2\}$ and $\{u_2,u_3\}$ are supported sets that each contain at least one vertex of $A$. 
	Hence, \Cref{prop:final A' conditions}~\ref{itm:A'-3} implies
    \[
        d_{B'}^1(u_i u_{i+1})\geq |B'| - \mu n, \qquad i\in\{1,2\}.
    \]
    The inequality above tells us that the number of common vertex neighbours of $\{u_1,u_2\}$ and $\{u_2,u_3\}$ in $B'$ satisfies
    \[
        |N_{B'}^1(u_1u_2)\cap N_{B'}^1(u_2 u_3)|\geq |B'|-2\mu n.
    \]
	Pick any common neighbour $v_1\in N_{B'}^1(u_1u_2)\cap N_{B'}^1(u_2 u_3)$. We will build the extension one vertex at a time for $(t - 2)$ steps -- after the $i$th step, we have vertices $v_1,\dots,v_i\in B'$ such that $\{u_1,u_2,v_1,\dots,v_i\}$ and $\{u_2,u_3,v_1,\dots,v_i\}$ are supported $(i+2)$-sets, each containing at least one vertex from $A$. Then, using the same rationale as above, they have at least $|B'|-2\mu n$ common vertex neighbours in $B'$, and we will pick $v_{i+1}$ from this common vertex neighbourhood. By construction, $\{u_1, u_2, v_1, \dots, v_{t-2}\}$ and $\{u_2,u_3, v_1, \dots, v_{t-2}\}$ are supported sets in $G$ as desired.
    
    All that is left is to show that we can make distinct choices each time to ensure that the extensions are disjoint. We know that $M_{A'}$ consists of $x\leq \mu n$ pairwise disjoint cherries. As each extension requires $t-2$ vertices and since we have at least $|B'|-2\mu n \ge t\mu n$ choices for a vertex in each step of extending a cherry, we can find disjoint extensions in $B'$ for each cherry of $M_{A'}$. 
\end{proof}

We are now ready to prove Proposition~\ref{prop:size balancing}.

\begin{proof}[Proof of Proposition~\ref{prop:size balancing}]
	As before, we write $|A'| = \floor{n/t} + x$, where we can assume $0 < x \le \mu n/2$. Let $M_{A'}$ be the subgraph of $\partial^2[A']$ guaranteed by the preceding argument; i.e.\ it consists of $x$ pairwise disjoint cherries. Then, invoking Claim~\ref{claim:disjoint extended cherries}, we extend the cherries of $M_{A'}$ so that the extensions are pairwise vertex-disjoint. 
	Let $\cP'$ be the collection consisting of: a single vertex for each vertex in $A' \setminus V(M_{A'})$, and an extended cherry for each cherry in $M_{A'}$ (so that the extensions are pairwise vertex-disjoint).
	Then item~\ref{itm:P'-1} in statement of \Cref{prop:size balancing} clearly holds, using that singletons and extended cherries support extended $\ell$-paths (see \Cref{prop:supported-l-paths}~\ref{itm:extended-path-4}) and that $M_{A'}$ consists of $x\le \mu n$ cherries. Item~\ref{itm:P'-2} is immediate by construction.
	For~\ref{itm:P'-3} we have
	\begin{equation} \label{eqn:P}
		|\cP'| = |A'| - 2x = \floor{\frac{n}{t}} - x \ge \frac{n}{t} - \mu n.
	\end{equation}
    For~\ref{itm:P'-4}, notice that
	\begin{equation*}
		|V(\cP')| = |A'| + (t-2)x = \floor{\frac{n}{t}} + (t-1)x \le \frac{n}{t} + t \mu n.
	\end{equation*}
	as required.
	Also, using \eqref{eqn:P},
	\begin{equation*}
		|B' \setminus V(\cP')| - (t-1)|\cP'|
		= n - \floor{\frac{n}{t}} - x - (t-2)x - (t-1)\left(\floor{\frac{n}{t}} - x\right)
		= n - t \floor{\frac{n}{t}}.
	\end{equation*}
	Since $0 \le n - t \floor{\frac{n}{t}} < t$, this proves~\ref{itm:P'-5}, the final claim of \Cref{prop:size balancing}.
\end{proof}

To prove \Cref{prop:size balancing final}, we modify the family given by \Cref{prop:size balancing} by first extending single path by an appropriate amount, and then extending all other paths as needed so that they start and end in $A$.
The following claim will allow us to make these extensions.

\begin{claim} \label{claim:path extension}
	Let $a \in A'$, let $W$ be a set of at most $10t\mu n$ forbidden vertices.
	Then there are distinct vertices $v_1, \ldots, v_{t-1}, u_1, \ldots, u_{t-1} \in B' \setminus W$ and $a_1, a_2 \in A \setminus W$ such that the sequence $a_1 v_1 \ldots v_{t-1} a u_1 \ldots u_{t-1} a_2$ supports an extended $\ell$-path. 
\end{claim}

\begin{proof}
    By \cref{prop:supported-l-paths}~\ref{itm:extended-path-3}, it suffices to find distinct $v_1, \ldots, v_{t-1}, u_1, \ldots, u_{t-1} \in B' \setminus W$ and $a_1, a_2 \in A \setminus W$ so that: $v_1 \ldots v_{t-1} a u_1 \ldots u_{t-1}$ is a tight path and the sets $\{a_1, v_1, \ldots, v_{t-1}\}$ and $\{a_2, u_1, \ldots, u_{t-1}\}$ are supported. Indeed, in the sequence $a_1 v_1 \ldots v_{t-1} a u_1 \ldots u_{t-1} a_2$, the only subsequences of consecutive vertices that have at most one index divisible by $t$ (when we number the indices by $0$ to $2t$) are contained in $a_1 v_1 \ldots v_{t-1}$, $v_1 \ldots v_{t-1} a u_1 \ldots u_{t-1}$ or $u_1 \ldots u_{t-1}a_2$, and it is easy to check that for each of them every sequence of at most $k$ consecutive edges forms a supported set.

	First, observe that if $P$ is a tight path of order $p \le 4k$, and $W'$ is a set of size at most $11t\mu n$, then the number of vertices $v\in B' \setminus W'$ such that $Pv$ (similarly $vP$) is a tight path of order $p+1$ is the number of vertex neighbours of the last $\min\{p, k-1\}$ vertices in $P$ (or the first $\min\{p, k-1\}$ vertices in $P$ if we are looking for $vP$) that lie outside $W'$, which, by \cref{obs:min 1-degree,prop:final A' conditions}, is at least
	\begin{equation*}
		n - \frac{n}{t} - k - |A'| - |W'| - p \ge \frac{n}{2t},
	\end{equation*}
	using that $t \ge 3$ (which follows from $(k,\ell) \neq (3,1)$; see \Cref{obs:useful t info}). Staring with $a$ and applying this inequality repeatedly to obtain the vertices $u_1, \ldots, u_{t-1}, v_{t-1}, \ldots, v_1$ in that order, we find that the number of sequences $v_1, \ldots, v_{t-1}, u_1, \ldots, u_{t-1} \in B' \setminus W$ such that $v_1 \ldots v_{t-1} a u_1 \ldots u_{t-1}$ is a tight path is at least $(\frac{n}{2t})^{2(t-1)}$. 

	We argue that there is such a sequence such that $\{v_1, \ldots, v_{t-1}\}$ and $\{u_1, \ldots, u_{t-1}\}$ are supported sets in $\Fp_B$. Indeed, as $V(\Fp_B) = B'$ and $\Fp_B$ has no isolated vertices, by the minimum positive co-degree assumption, the number of sequences $x_1, \ldots, x_{t-1} \in B'$ such that $\{x_1, \ldots, x_{t-1}\}$ is supported in $\Fp_B$ is at least $(1 - \mu)^{t-2}|B'|^{t-1}$. Thus, the number of sequences $x_1, \ldots, x_{t-1} \in B'$ such that $\{x_1, \ldots, x_{t-1}\}$ is not supported in $\Fp_B$ is at most $(1 - (1 - \mu)^{t-2})|B'|^{t-1} \le t \mu n^{t-1}$, implying that the number of sequences $v_1, \ldots, v_{t-1}, u_1, \ldots, u_{t-1} \in B'$ such that one of $\{v_1, \ldots, v_{t-1}\}$ and $\{u_1, \ldots, u_{t-1}\}$ is not supported in $\Fp_B$ is at most $2t\mu n^{2(t-1)} < (\frac{n}{2t})^{2(t-1)}$. Thus, there is a choice of $v_1, \ldots, v_{t-1}, u_1, \ldots, u_{t-1}$ as above such that the $v_i$'s and the $u_i$'s form supported sets in $\Fp_B$. 

	To finish, recall that supported sets of size at most $k-1$ in $\Fp_B$ are $A$-rich. This means that $|N^1(\{v_1, \ldots, v_{t-1}\}) \cap A| \ge |A| - \mu n \ge n/t - 2\mu n$, which implies that there is a vertex $a_1 \in (A \setminus (W \cup \{a\})) \cap N^1(\{v_1, \ldots, v_{t-1}\})$. Similarly, there is a vertex $a_2 \in (A \setminus (W \cup \{a, a_1\}) \cap N^1(\{u_1, \ldots, u_{t-1}\}$.
	The sequence $v_1, \ldots, v_{t-1}, u_1, \ldots, u_{t-1}, a_1, a_2$ satisfies the requirement of the claim.
\end{proof}

\begin{claim} \label{claim:exceptional path extension}
	Let $W$ be a set of at most $10t \mu n$ forbidden vertices and let $r$ satisfy $t+1 \le r \le 2t$ and $r \equiv 1 \pmod{k-\ell}$. Then there is a sequence of distinct $r$ vertices in $V(G) \setminus W$ that starts and ends in $A$, its other vertices are in $B'$, and it supports an extended $\ell$-path.
\end{claim}

\begin{proof}
	We follow a similar strategy to the previous proof. Pick some $a_1 \in A \setminus W$. By the second paragraph in the previous proof, there are at least $(\frac{n}{2t})^{r-2}$ sequences $v_1 \ldots v_{r-1}$ of distinct vertices in $B' \setminus W$ such that $a v_1 \ldots v_{r-1}$ is a tight path in $G$.
	By the third paragraph in the same proof, there is such a sequence where the final $\min\{r-1, k-1\}$ vertices form a supported set in $\Fp_B$.
	This implies that there is $a_2 \in A \setminus (W \cup \{a_1\})$ such that $v_1 \ldots v_{r-1} a_2$ is a tight path in $G$.
	Finally, by \Cref{prop:supported-l-paths}~\ref{itm:extended-path-4}, we conclude that $a_1 v_1 \dots v_{r-2} a_2$ is the desired sequence.
\end{proof}

\begin{proof}[Proof of \Cref{prop:size balancing final}]
Let $\cP'$ be the output of \cref{prop:size balancing}. We modify $\cP'$ in two steps as described above. 
Let $A_0$ be the set of all vertices in $A' \setminus A$ that form singleton sequences in $\cP'$. Then, if we invoke \cref{prop:final A' conditions} with $\mu/2$ in place of $\mu$ (note that this makes all its conclusions strictly stronger, so all preceding results in this section still hold), we have $|A_0| \le |A' \setminus A| \le \mu n$.
By repeatedly applying \Cref{claim:path extension}, we find, for each $a \in A_0$, a sequence $S_a$ consisting of $a$, two vertices from $A$ and $2(t-1)$ vertices from $B'$ (in some order), so that: $S_a$ starts and ends in $A$; it supports an extended $\ell$-path; and the $S_a$'s are pairwise vertex-disjoint.
Define $\cP''$ to the collection of vertex sequences $P$ such that either $P \in \{S_a : a \in A_0\}$ or $P \in \cP'$ is not a singleton vertex of $A$ that appears in some sequence $S_a$ with $a \in A_0$.
Write $r := |B' \setminus V(\cP'')| - (t-1)|\cP''|$.
Then $r = |B' \setminus V(\cP')| - (t-1)|\cP'|$, 
because replacing the three singleton paths from $A'$ that appear in $S_a$ by $S_a$ decreases $|B' \setminus V(\cP')|$ by $2(t-1)$ and $|\cP'|$ by 2, showing that $0 \le r \le t$ via \Cref{prop:size balancing}~\ref{itm:P'-5}. Moreover, by \Cref{prop:size balancing}~\ref{itm:P'-2}, we know  $A' \subseteq V(\cP')$, so we see that $|B' \setminus V(\cP')| + |V(\cP')| = |A'| + |B'| = n$. Hence, we have
\begin{equation*}
	n = r + (t-1)|\cP'| + |V(\cP')|
	\equiv r + (t-1)|\cP'| + |\cP'|
	\equiv r \pmod {k-\ell},
\end{equation*}
using that all sequences in $\cP'$ are pairwise disjoint and have length $1 \pmod{t}$, and that $t$ is divisible by $k - \ell$. Since $n$ is divisible by $k-\ell$, so is $r$. In particular, $r+t+1 \equiv 1 \pmod{k-\ell}$.
By \Cref{claim:exceptional path extension}, we can find a sequence $S$ of $r+t+1$ vertices such that: it starts and ends in $A$; its other vertices are in $B'$; it supports an extended $\ell$-path; and it is disjoint from all non-singleton sequences in $\cP''$, where we use the fact that the sequences $\{S_a: a\in A_0\}$ span at most $2t|A_0| \le 2t \mu n$ vertices and \cref{prop:size balancing}~\ref{itm:P'-1} implies that the other non-singleton sequences in $\cP''$ span at most $(t+1)\mu n$ vertices. 
Let $\cP$ be obtained from $\cP''$ by removing the two singleton sequences in $\cP''$ that appear in $S$ and adding $S$.

It remains to verify that the five desired properties in the statement of \cref{prop:size balancing final} indeed hold. Throughout, we will utilise \cref{prop:size balancing} \crefrange{itm:P'-1}{itm:P'-5}.
The first two are immediate from the construction of $\cP$. 
For the property~\ref{itm:P-3}, we have
\begin{equation*}
	|\cP| = |\cP''| - 1 \ge |\cP'| - 2|A' \setminus A|
	\ge |\cP'| - 2\mu n - 1
	\ge \frac{n}{t} - 4\mu n.
\end{equation*}
For the item~\ref{itm:P-4},
\begin{align*}
	|V(\cP)| \le |V(\cP'')| + (r + t + 1) 
	& \le |V(\cP')| + (2t+1)|A' \setminus A| + (r + t + 1) \\
	& \le \frac{n}{t} + (t + 2t+1) \mu n + 2t
	\le \frac{n}{t} + 4t \mu n.
\end{align*}
Next, by the choice of $\cP$, we have $|\cP| = |\cP''| - 1$ and $|B' \setminus V(\cP)| = |B' \setminus V(\cP'')| - (r + t - 1)$, and therefore
\begin{equation*}
	|B' \setminus V(\cP)| - (t-1)|\cP|
	= |B' \setminus V(\cP'')| - (t-1)|\cP''| - (r + t-1) + t-1 = 0,
\end{equation*}
proving~\ref{itm:P-5}.
Finally, \ref{itm:P-6} is immediate from the construction: all sequences have length $1$ (singletons from $A$), $t+1$ (extended cherries), $2t+1$ (sequences $S_a$ with $a$ being a singleton from $A'$ in $\cP'$), and $r + t+1 \ge t+1$ for the final sequence $S$.
\end{proof}

\section{Building a Hamilton cycle}\label{sec:Hamilton cycle extremal}
Rather than constructing a Hamilton $\ell$-cycle in $G$ directly, we first build a perfect matching in an auxiliary balanced $t$-partite $t$-graph (see $\Gp$, defined below). We will do so in a semi-random manner using~\Cref{lem:k-partite matching}, and argue that the co-degrees in the auxiliary graph are large enough to connect this matching into a Hamilton cycle, and the auxiliary $t$-graph is chosen such that this will correspond to a Hamilton $\ell$-cycle in $G$. 

The $t$-partite $t$-graph will have one part corresponding to the collection $\cP$ found in the previous section, and the other $t-1$ parts will correspond to an arbitrary equipartition of $B'' := B' \setminus V(\cP)$. That these parts have the same size follows from the fifth (and crucial) property~\ref{itm:P-5} of $\cP$, which asserts that $|B''| = (t-1)|\cP|$. Because we intend to use \Cref{lem:k-partite matching}, which applies to $t$-partite $t$-graphs with very high minimum supported co-degree, we need to prove that $\Gp$ indeed has high minimum supported co-degree, and this will follow from the choice of $B'$ as the vertex set of the hypergraph $\Fp_B$ (see \Cref{prop:final A' conditions}), whose edges are $A$-rich sets, and from property~\ref{itm:P-1} of $\cP$, which guarantees that each sequence in $\cP$ starts and ends in $A$. 

Throughout this section, we will frequently recall the notion of $A$-rich sets (recall \Cref{defn:scarce and rich sets}) and, like in the previous section, that of vertex sequences that support (extended) $\ell$-paths (recall \Cref{def:supported-l-paths}). Again, we say that a tight path on $r < k$ vertices is simply an (ordered) supported set of size $r$. As in the previous two sections, we fix
\begin{equation*}
	1/n \ll \eps \ll \mu \ll 1/k \le 1/3, 
\end{equation*}
and we subsequently set $\ell\in [k-1]$ such that $(k, \ell) \ne (3, 1)$, $t = \floor{\frac{k}{k - \ell}}(k - \ell)$, and let $G \in \exE_\eps (n, k, \ell)$ be a $k$-graph. We fix sets  $A, A', B \subseteq V(G)$ as in \cref{prop:size balancing final,prop:final A' conditions} and let $\cP$ be the collection of vertex sequences output by \cref{prop:size balancing final}.

Write $m := |\cP|$, define $B'' := B' \setminus V(\cP)$ (so $|B''| = (t-1)m$ by \cref{prop:size balancing final}~\ref{itm:P-5}) and let $B_1 \sqcup \ldots \sqcup B_{t-1}$ be an arbitrary equipartition of $B''$. We define a $t$-partite $t$-graph $\Gp$ as follows. Let $V(\Gp)= B_1 \sqcup \ldots \sqcup B_{t-1} \sqcup \cP$ and $E(\Gp)$ be the subset of $B_1 \times \ldots \times B_{t-1} \times \cP$ such that any $(v_1, \ldots, v_{t-1}, P) \in B_1 \times \ldots \times B_{t-1} \times \cP$ forms an edge in $\Gp$ if and only if the vertex sequence corresponding to $v_1\dots v_{t-1} P$ (which, we note, has length congruent to $0$ modulo $k-\ell$ by \Cref{prop:size balancing final} and \Cref{def:supported-l-paths}) supports an $\ell$-path in $G$ and $\{v_1,\dots,v_{t-1}\}$ is supported in $\Fp_B$. The partite structure of $\Gp$ lends an inherent ordering to every $t$-edge. Specifically, we will always think of edges of $\Gp$ with vertices as ordered tuples from $B_1 \times \dots \times B_{t - 1} \times \cP$, and will often write them as vertex sequences with this canonical ordering. 

As mentioned above, we would like to apply \Cref{lem:k-partite matching} to $\Gp$. This requires showing that $\Gp$ has high minimum supported co-degree, which is what we show in the next claim.
For notational convenience, we define $B_t := \cP$ and let $\cP_A$ be the set of singleton paths in $\cP$ (which are, by property~\ref{itm:P-1} of $\cP$, vertices in $A$).
Notice that the number of non-singleton paths in $\cP$ is at most $\frac{1}{t+1}(|V(\cP)| - |\cP|) \le 4\mu n$ (using properties~\ref{itm:P-3} and~\ref{itm:P-4} of $\cP$, the fact non-singleton paths have length at least $t + 1$ from~\ref{itm:P-6}, and $t\ge 2$). Hence,
\begin{equation} \label{eqn:PA}
	|\cP_A| \ge |\cP| - 4\mu n.
\end{equation}

\begin{claim}
\label{claim:G+ partite sets sizes and codegrees}
	Let $i \in [t]$ and let $S$ be a supported set in $\Gp$ that avoids $B_i$, then $d^1_{B_i}(S) \ge |B_i| - 5t\mu n$ (where $d^1$ is taken with respect to $\Gp$).
\end{claim}
\begin{proof}
    Since $S$ is supported, it is contained in some edge $e$, which has the form $b_1 \ldots b_t$, with $b_i \in B_i$. Writing $S' = \{b_1, \ldots, b_t\} \setminus \{b_i\}$, it suffices to show $d^1_{B_i}(S') \ge |B_i| - 5t \mu n$ in $\Gp$, as $d^1_{B_i}(S) \ge d^1_{B_i}(S')$.

    We consider two cases: $i = t$ and $i \in [t-1]$.
	Let us begin with the first case. By definition of $\Gp$, we know that $b_1 \ldots b_{t-1}$ is supported by some edge of $\Fp_B$, which means that $b_1 \ldots b_{t-1}$ contained in some $A$-rich set, implying that $b_1 \ldots b_{t-1} a$ is supported in $G$ for all but at most $\mu n$ vertices $a \in A$, and hence for all but at most $\mu n$ vertices $a \in \cP_A$. Since every (ordered) supported $t$-set supports an $\ell$-path, it follows that $b_1\ldots b_{t-1}a \in E(\Gp)$ for all but at most $\mu n$ vertices $a \in \cP_A$. Thus, we find that 
	\begin{equation} 
		d^1_{B_t}(S') \ge d^1_{\cP_A}(b_1 \ldots b_{t-1})
		\ge |\cP_A| - \mu n
        \ge |\cP| - 5 \mu n
		\ge |B_t| - 5t \mu n,
	\end{equation} 
	using \eqref{eqn:PA} and $B_t = \cP$.

	For the second case, recall, from \Cref{prop:size balancing final}, that $b_t=w_1\dots w_r$ is a vertex sequence with $w_1,w_r \in A$. 
	For convenience, define $w_j = w_r$ for $j \ge r+1$.
	For $j \in [i]$, write 
	\begin{equation*}
		e_j=\{b_j, \dots, b_{i-1}, b_{i+1}, \dots, b_{t-1}, w_1, \dots, w_{k-t+j}\}.
	\end{equation*}
	We claim that, for any $b_i \in B_i$, the sequence $b_1 \ldots b_{t}$ is an edge in $\Gp$ if and only if the following conditions hold.
	\begin{itemize}
		\item
			$b_i \in N^1_{\Fp_B}(b_1 \ldots b_{i-1} b_{i+1} \ldots b_{t-1})$,
		\item
			$b_i \in N_G^1(e_j)$ for $j = 1$ and for every other $j \in [i]$ such that $j \equiv 1\pmod{k-\ell}$ and $k - t + j \le r$.
	\end{itemize}
	Indeed, the first condition is equivalent to $b_1 \ldots b_{t-1}$ being supported in $\Fp_B$ (using that $S'$ is supported in $\Gp$, which implies that $b_1 \ldots b_{i-1}b_{i+1} \ldots b_{t-1}$ is supported in $\Fp_B$), so it suffices to check that the second condition is equivalent to $b_1 \ldots b_t w_1 \ldots w_r$ supporting an $\ell$-path.
	If $r = 1$ one just needs to check when $b_1 \ldots b_{t-1}w_1$ is supported in $G$, which is exactly when $b_i \in N_G^1(e_1)$ (again using that $S'$ is supported in $\Gp$, showing that $e_1$ is supported in $G$), so we may assume that $r \ge 2$, which implies $r \ge t+1$ by \ref{itm:P-6} in \Cref{prop:size balancing final}.
	Recall that for $b_1 \ldots b_{t-1} w_1 \ldots w_r$ to support an $\ell$-path we need that every $k$ consecutive vertices in this sequence, that start at an index which is $1 \pmod{k-\ell}$, form an edge, and that the last $t$ vertices in the sequence form a supported set in $G$. The latter holds due to $w_1 \ldots w_r$ supporting an extended $\ell$-path and $r \ge t+1$. The former is equivalent to the second item above (noting that the sequence has length at least $2t \ge k$).

	In particular, since $S'$ is supported in $\Gp$ we know that there is some $b_i \in B_i$ that satisfies the above requirements, showing that $b_1 \ldots b_{i-1} b_{i+1} \ldots b_{t-1}$ is supported in $\Fp_B$ and that $e_j$ is supported in $G$ for $j = 1$ and every other $j \le \min\{i, r-k+t\}$ with $j \equiv 1 \pmod{k-\ell}$; denote the set of such $j$'s by $J$.

	Since, for every $j \in J$, we have that $e_j$ is a supported set in $G$ that contains $w_1 \in A$, it satisfies $d^1_{B'}(e_j) \ge |B'| - \mu n$ by \Cref{prop:final A' conditions}~\ref{itm:A'-3}, and thus $d^1_{B_i}(e_j) \ge |B_i| - \mu n$. Moreover, by \Cref{prop:final A' conditions}~\ref{itm:A'-4}, we have $d^1_{\Fp_B}(b_1 \ldots b_{i-1}b_{i+1}\ldots b_{t-1}) \ge |B'| - \mu n$ in $B'$. 
    Altogether, we have $d^1_{B_i}(S') \ge |B_i| - t \mu n \ge |B_i| - 5t\mu n$ in $\Gp$.
\end{proof}

Before proceeding further, it would be useful to describe an outline of our strategy for the remaining stages of the proof. Our next step will be to find a perfect matching $M$ in $\Gp$, which corresponds to a collection of pairwise disjoint vertex sequences in $G$ that support $\ell$-paths, and also span $G$. We then form an auxiliary digraph $\Dp$ whose vertices are the edges of $M$, and we direct an edge from $e$ to $f$ whenever the vertex sequence $ef$ supports an $\ell$-path. In particular, a directed Hamilton cycle in $\Dp$ will correspond to a spanning $\ell$-cycle in $G$ (using \Cref{prop:supported-l-paths}~\ref{itm:extended-path-2}). Thus, we need to show that $D$ has large minimum semi-degree, and then~\Cref{lem:directed Dirac} will provide the desired cycle. We will show that each edge of $\Gp$ has many potential in-neighbour and out-neighbour edges in $\Gp$. We then find the matching $M$ in $\Gp$ using~\Cref{lem:k-partite matching}, and argue that the digraph $\Dp$ formed as a result of this matching maintains high semi-degree conditions.

To formalise this, for an edge $e = b_1\ldots b_{t-1} P \in E(\Gp)$ (where $b_i \in B_i$ and $P \in \cP$), let $\Fp_e$ be the subgraph of $\Gp$ consisting of edges $f = b_1' \ldots b_{t-1}' P'$ (where $b_i' \in B_i$ and $P' \in \cP$) such that the vertex sequence $ef = b_1 \ldots b_{t-1} P b_1' \ldots b_{t-1}' P'$ supports an $\ell$-path in $G$, and similarly define $\Fm_e$ as the subgraph of $\Gp$ consisting of edges $f$ such that $fe$ supports an $\ell$-path in $G$.

\begin{claim} \label{claim:Fp}
		$e(\Fp_e) \ge |B_1|^t - 5t\mu n^t$ for every $e \in E(\Gp)$.
\end{claim}

\begin{proof}
	Write $e=b_1\dots b_{t-1} P$. We wish to find edges $f=b_1' \dots b_{t-1}' P'\in \Gp$ such that $ef$ supports an $\ell$-path in $G$. By assumption, $e$ supports an $\ell$-path, which implies that the last $t$ vertices of $e$, say $w_1 \dots w_{t}$, constitute a supported set in $G$. Since we know that $f$ supports an $\ell$-path in $G$ as well, if we can show that $w _1 \dots w_{t} b_1' \dots b_{t-1}'$ is a tight path, then by \Cref{prop:supported-l-paths}~\ref{itm:extended-path-1} we have $f \in E(\Fp_e)$.

	From \Cref{prop:size balancing final}~\ref{itm:P-1}, we know that $w_t\in A$, and then \Cref{prop:final A' conditions}~\ref{itm:A'-3} implies that $d_{B'}^1(S)\ge |B'|-\mu n$ and consequently $d_{B_i}^1(S)\ge |B_i|-\mu n$ for any $G$-supported set $S$ containing $w_{t}$ and for any $i\in [t-1]$. We show by induction that, for $i \in [t-1]$, the number of vertex sequences $b_1' \dots b_i'\ \in B_1 \times \ldots \times B_i$ such that $w_1 \dots w_{t} b_1' \ldots b_i'$ is a tight path in $G$ and $\{b_1', \dots, b_i'\}$ is supported in $\Fp_B$ is at least $(|B_1| - 2\mu n)^i$. This holds for $i = 1$ by the arguments above. Let $i\in [2,t-1]$ and suppose the assertion holds up to $i-1$. Let $b_1' \dots b_{i-1}'\in B_1 \times \dots \times B_{i-1}$ be such that $w_1 \dots w_{t} b_1' \dots b_{i-1}'$ is a tight path and $\{b_1', \ldots, b_{i-1}'\}$ is supported in $\Fp_B$. Then, for $b_i' \in B_i$ to be such that $w_1 \dots w_{t} b_1' \dots b_i'$ is a tight path and $\{b_1', \ldots, b_i'\}$ is supported in $\Fp_B$, we exactly need 
	\begin{equation*}
		b_i' \in N^1_G(w_{s_i} \ldots w_{t} b_1' \ldots b_{i-1}') \cap N^1_{\Fp_B}(b_1' \ldots b_{i-1}') \cap B_i, \quad s_i := \max\{t-(k-i)+1, 1\}. 
	\end{equation*}
	By the previous remark and the structure of $\Fp_B$ given by \cref{prop:final A' conditions}~\ref{itm:A'-4}, there are at least $|B_i| - 2\mu n = |B_1| - 2\mu n$ options for $b_i'$, showing that the number of sequences $b_1',\ldots, b_i'$ as required is at least $(|B_1| - 2\mu n)^i$, as claimed.
	To finish off, we recall that $b_1' \ldots b_{t-1}' a$ is an edge in $\Gp$ for every $a \in \cP_A \cap N_G^1(b_1' \ldots b_{t-1}')$. Since $\{b_1, \ldots, b_{t-1}'\}$ is supported in $\Fp_B$, it is $A$-rich and thus the number of such $a$ is at least $|\cP_A| - \mu n \ge |\cP| - 5 \mu n = |B_1| - 5\mu n$, by \eqref{eqn:PA}.
	In total we conclude that the number of edges $f$ in $\Gp$ that are in $\Fp_e$ is at least $(|B_1| - 2\mu n)^{t-1} (|B_1| - 5\mu n) \ge |B_1|^t - (2(t-1)\mu n + 5\mu n)|B_1|^{t-1} \ge |B_1|^t - 5t\mu n^t$.
\end{proof}

\begin{claim} \label{claim:Fm}
		$e(\Fm_e) \ge |B_1|^t - 5t\mu n^t$ for every $e \in E(\Gp)$.
\end{claim}

\begin{proof}
    Again, set $e = b_1 \dots b_t P$. Recall that we want edges $f=b_1' \dots b_{t-1}' P' \in \Gp$ such that $fe$ supports an $\ell$-path in $G$. From \eqref{eqn:PA}, we know that most elements of $B_t=\cP$ consist of singleton vertices from $A$. For simplicity, we will only try to find edges in $\Fm_e$ that are of the form $f=b_1' \dots b_{t-1}'a$, where $a\in \cP_A$. Analogous to the situation above, since $e$ and $f$ support $\ell$-paths in $G$, in order for $fe$ to do so as well, it will suffice to argue that $b_1' \ldots b_{t-1'} a b_1 \ldots b_{t-1}$ is a tight path in $G$.

	We inductively argue that, for each $i\in [0,t-1]$, there are at least $(|B_1| - 2\mu n)^{i} (|B_1| - 5\mu n)$ sequences $b_{t-i}' \dots b_{t-1}' a \in B_{t-i} \times \dots \times B_{t-1} \times \cP_A$ such that $b_{t-i}' \dots b_{t-1}' a b_1 \dots b_{t-1}$ is a tight path in $G$ and $\{b_{t-i}', \dots, b_{t-1}'\}$ is supported in $\Fp_B$. Start with $i=0$. By definition of $\Gp$, we know that $\{b_1, \dots, b_{t-1}\}$ is $A$-rich, meaning that $a b_1 \dots b_{t-1}$ is supported in $G$ for all but at most $\mu n$ vertices $a \in \cP_A$ and hence, by \eqref{eqn:PA}, for at least $|\cP_A|-\mu n \ge |B_1| - 5\mu n$ vertices $a \in \cP_A$. Let $i\in [t-1]$ and suppose the induction hypothesis holds for $i-1$. Let $b_{t-i+1}' \dots b_{t-1}' a \in B_{t-i+1} \times \dots \times B_{t-1} \times \cP_A$ be such that $b_{t-i+1}' \dots b_{t-1}' a b_1 \dots b_{t-1}$ is a tight path in $G$ and $\{b_{t-i+1}', \ldots, b_{t-1}'\}$ is supported in $\Fp_B$. Then, for $b_{t-i}'$ to be a valid choice, we require exactly that
	\begin{equation*}
		b_{t-i}' \in N_G^1(b_{t-i+1}' \dots b_{t-1}' a b_1 \dots b_{s_i}) \cap N_{\Fp_B}^1(b_{t-i+1}' \ldots b_{t-1}') \cap B_i, \quad s_i := \min\{k-i-1, t-1\}.
	\end{equation*}
	Crucially, as $a\in \cP_A \subseteq A$, we can use the same rationale as in the above proof of \cref{claim:Fp} to conclude that there are at least $|B_{t-i}|-2\mu n=|B_1|-2\mu n$ choices for $b_{t-i}'$. This proves the inductive statement, and the same calculations as in the proof of \Cref{claim:Fp} complete the proof of this claim.
\end{proof}

We are now finally ready to prove~\Cref{thm:Hamilton cycle extremal}.

\begin{proof}[Proof of~\Cref{thm:Hamilton cycle extremal}]
	Fix a constant $\mu'$ such that $\mu\ll \mu'\ll 1/k\le 1/3$, let $\Gp$ be the auxiliary $t$-partite $t$-graph defined above, and let $\cF := \{\Fp_e, \Fm_e : e \in E(\Gp)\}$ be as defined above.
	Notice that $\Gp$ and $\cF$ satisfy the requirements of \Cref{cor:k-partite-matching} (with parameters $|B_1| = m$, $\mu'$ and $t$ in place of $n$, $\eps$ and $k$ respectively), using \Cref{claim:G+ partite sets sizes and codegrees,claim:Fp,claim:Fm}. Thus, by \Cref{cor:k-partite-matching} there is a perfect matching $M$ in $\Gp$ such that $|M \cap E(\Fp_e)|, |M \cap E(\Fm_e)| \ge m/2$ for every $e \in E(\Gp)$.

	As discussed previously, we now define a corresponding auxiliary digraph $\Dp$. Set $V(\Dp)=E(M)$, so the vertices of $\Dp$ correspond to edges of $\Gp$. For any $e\in V(\Dp)$, define its in and out neighbourhoods to be
	\[
		N^+(e)=\{f\in V(\Dp): f\in \Fp_e\} \text{ and } N^-(e)=\{f\in V(\Dp): f\in \Fm_e\}.
	\]
    From the choice of $M$ and $\Dp$, we have $\delta^0(\Dp) \ge m/2 = |\Dp|/2$. In particular, if we let $M = \{P_1, \ldots, P_m\}$ where we think of each $P_i = (b_1, \dots, b_{t - 1}, P) \in B_1 \times \dots \times B_{t-1} \times \cP$ as the vertex sequence $b_1 \dots b_{t-1} P$, then \Cref{lem:directed Dirac} can be applied to find a directed Hamilton cycle $P_1 \ldots P_m$ in $\Dp$.

	The crucial observation we make here is that $M$ consists of pairwise disjoint vertex sequences in $\Gp$ that span $G$, each of which supports an $\ell$-path in $G$. By \Cref{prop:supported-l-paths}~\ref{itm:extended-path-2}, we conclude that the vertex sequence $P_1 \ldots P_m$ is a Hamilton tight cycle in $G$.
\end{proof}

\section{Open problems}\label{sec:conclusion}
    Our main result determines (up to an additive error of $1$) the optimal minimum supported co-degree condition that guarantees a spanning $\ell$-cycle in an $n$-vertex $k$-graph (when $k-\ell$ divides $n$, a necessary divisibility condition) for all $k \ge 3$ and $\ell \in [k-1]$ except when $(k,\ell) = (3,1)$, and as mentioned in \Cref{sec:intro}, an upcoming paper of Mycroft and Z{\'a}rate-Guer{\'e}n resolves this remaining open case.

    These results contribute to the broader research theme of finding the best possible minimum supported co-degree to ensure a specified spanning substructure, an area which poses some tantalising questions. Our original motivation for working on this problem arose from Illingworth, Lang, M\"uyesser, Parczyk and Sgueglia~\cite[Conjecture $1.4$]{illingworth2025spanning}, who conjectured that $\delta^*(G) \ge (1-1/k)n$ guarantees a tight Hamilton cycle. The main focus of their paper is the minimum supported co-degree required to ensure a spanning $(k-1)$-sphere -- a subset of edges in a $k$-graph that induces a homogeneous simplicial complex homeomorphic to the $(k-1)$-sphere $\mathbb{S}^{k-1}$, and whose $0$-skeleton spans the entire vertex set. They tackle the following conjecture of Georgakopoulos, Haslegrave, Montgomery and Narayanan~\cite{georgakopoulos2022spanning} (paraphrased).
	\begin{conj}\label{conj:spheres}
		Every $k$-graph $G$ of order $n>k \ge 2$ with $\delta^*(G) \ge n/2$ and with no isolated vertices contains a spanning copy of $\mathbb{S}^{k-1}$.\footnote{In the original phrasing of the conjecture $G$ is required to be tightly connected -- meaning that any two edges can be joined by a tight walk -- but this follows from the other assumptions.}
	\end{conj}
    The chief result in~\cite{georgakopoulos2022spanning} provides an asymptotically optimal minimum co-degree condition for a $3$-graph to contain a vertex-spanning copy of any \emph{surface}, meaning an arbitrary connected, closed $2$-manifold. In particular, this addresses the minimum co-degree necessary to guarantee a spanning $2$-sphere in a $3$-graph.
    \begin{thm}[{\cite[Theorem 1.4]{georgakopoulos2022spanning}}]\label{thm:spheres-codegree}
        Let $\cS$ be an arbitrary surface and $\eps>0$. Then any sufficiently large $n$-vertex $3$-graph with $\delta(G) \ge n/3 + \eps n$ contains a spanning copy of $\cS$. Moreover, for any $n \in \mathbb{N}$, there exists an $n$-vertex $3$-graph $H$ with $\delta(H) = \floor{\frac{n}{3}} - 1$ such that there are at most $2\ceil{\frac{n}{3}}$ vertices in the $0$-skeleton of a copy of any surface in $H$.
    \end{thm}
    
    The authors of~\cite{illingworth2025spanning} utilise the same blow-up framework we use to prove \cref{thm:Hamilton cycle non-extremal} (as mentioned, \cref{lem:blow-up chain} is based on Lemma 2.1 of their paper) to prove an asymptotic weakening of this conjecture, that is, that $\delta^*(G) \ge n/2 + o(n)$ suffices. We believe that the methods used here can potentially be used to improve this to a stability result for \cref{conj:spheres}, and possibly achieve the same improvement for \cref{thm:spheres-codegree} as well. Perhaps one could build on such results to prove an exact bounds, thereby fully settling these conjectures, though this is likely to require some careful extremal analysis.

\providecommand{\bysame}{\leavevmode\hbox to3em{\hrulefill}\thinspace}
\providecommand{\MR}{\relax\ifhmode\unskip\space\fi MR }
% \MRhref is called by the amsart/book/proc definition of \MR.
\providecommand{\MRhref}[2]{%
  \href{http://www.ams.org/mathscinet-getitem?mr=#1}{#2}
}
\providecommand{\href}[2]{#2}

\appendix

\section{Proof of \Cref{lem:blow-up chain}} \label{appendix:proof-blow-up-tiling}

	In this section we sketch the proof of \Cref{lem:blow-up chain}, restated here.

	\lemBlowUpChain*

	As mentioned earlier, this lemma can be proved by following the arguments of the proof of \cite[Lemma 2.1]{illingworth2025spanning} and making a few small modifications. We note that Lang and Sanhueza-Matamala \cite[Lemma 6.3]{lang2024hypergraph} provide a somewhat more general lemma that could be used for out purpose instead, but both assumption and conclusion are slightly different to the ones in \Cref{lem:blow-up chain}.
	We now state \cite[Lemma 2.1]{illingworth2025spanning}.

	\begin{lem}[Lemma 2.1 in \cite{illingworth2025spanning}]
		\label{lem:blow-up chain-original}
		Let $1/n\ll 1/m_2\ll 1/m_1\ll 1/s,\gamma\ll \eps,1/k\le 1/2$. 
		Let $G$ be an $n$-vertex $k$-graph with no isolated vertices and with $\delta^*(G) \ge (1/2+\eps)n$, and let $\cS$ be the family of $s$-vertex $k$-graphs with no isolated vertices and with $\delta^*(S) \ge (1/2 + \eps/2)s$. 
		Then there exists a sequence of $k$-graphs $F_1,\dots, F_r$ and a sequence of subgraphs $F_1^*,\dots, F_r^*\subseteq G$ such that the following hold for all $i,j\in [r]$:

		\begin{itemize}
			\item $F_i$ is a copy of a $k$-graph in $\cS$,
			\item $F_i^*$ is a $(\gamma,m_i^*)$-nearly-regular blow-up of $F_i$ for some $m_i^*\in[m_1,m_2]$,
			\item $V(F_1^*)\cup\dots\cup V(F_r^*)=V(G)$,
			\item $V(F_i^*)\cap V(F_j^*)=\emptyset$ unless $j \in \{i-1,i,i+1\}$,
			\item $V(F_i^*)\cap V(F_{i+1}^*)$, with $i \in [r-1]$, consists of exactly $k$ vertices that induce an edge in $F_i^*$ and $F_{i+1}^*$ that is disjoint from the singleton parts of both blow-ups (if they exist).
		\end{itemize}
	\end{lem}

	For convenience, we make the following definition (inspired by terminology from \cite{lang2024hypergraph}).

	\begin{defn}
		Let $G$ be an $n$-vertex $k$-graph and let $\cS$ be a family of $s$-vertex $k$-graphs.
		We say that $G$ \emph{satisfies $\cS$ $r$-robustly} if for every set $W$ of $r$ vertices in $G$, there are at least $(1 - 1/s^2)\binom{n-r}{s-r}$ sets $U$ of $s$ vertices in $G$ such that $W \subseteq U$ and $G[U] \in \cS$.
	\end{defn}

	Notice that \Cref{lem:blow-up chain} and \Cref{lem:blow-up chain-original} differ on two points. First, instead of considering general families $\cG$ and $\cS$ of $n$-vertex and $s$-vertex $k$-graphs (where, in the terminology we just defined, every $G \in \cG$ satisfies $\cS$ $r$-robustly for every $r \le 2k$), the latter lemma is stated for specific choice $\cG_0$ and $\cS_0$, namely where $\cG_0$ is the family of $n$-vertex $k$-graphs with no isolated vertices and with minimum supported codegree at least $(1/2+\eps)n$, and $\cS_0$ is the family of $s$-vertex $k$-graphs with no isolated vertices and with minimum supported codegree at least $(1/2 + \eps/2)s$. Notice that, by \cite[Lemma 3.8]{illingworth2025spanning}, it is indeed the case that every $G \in \cG_0$ satisfies $\cS$ $r$-robustly for every $r \le 2k$.\footnote{In fact, they prove a stronger bound, namely that for every $G \in \cG_0$ and set $W$ of $r$ vertices in $G$, where $r \le 2k$, there are at least $(1 - e^{-\sqrt{s}})\binom{n-r}{s-r}$ sets $U$ of $s$ vertices such that $W \subseteq U$ and $G[U] \in \cS_0$. The weaker bound of $(1 - 1/s^2)\binom{n-r}{s-r}$ suffices for the purpose of the proof of \Cref{lem:blow-up chain-original}.}
	The second difference is more minor: while in \Cref{lem:blow-up chain} we require the sequence $F_1^*, \ldots, F_r^*$ to form a cyclic chain structure, namely that every two consecutive graphs intersect in one edge, etc., in \Cref{lem:blow-up chain-original} they are only required to form a chain, meaning that the `linking edge' between $F_r^*$ and $F_1^*$ is missing.

	In short, in order to modify the proof of \Cref{lem:blow-up chain-original} from \cite{illingworth2025spanning} to prove \Cref{lem:blow-up chain}, we can simply replace each mention of the specific properties $\cG_0$ and $\cS_0$ by general properties $\cG$ and $\cS$ such that every $G \in \cG$ satisfies $\cS$ $r$-robustly for every $r \le 2k$. Additionally, we simply add one more `link' to the chain $F_1^*, \ldots, F_r^*$ in order to make it cyclic.

	In the rest of the section, we give a more detailed proof sketch of \Cref{lem:blow-up chain}.
	Let $\eta$ and $m_{1.5}$ be constants such that $1/n \ll \eta \ll 1/m_2 \ll 1/m_{1.5} \ll m_1$. Fix families $\cG$ and $\cS$ as above and let $G \in \cG$. We proceed in three steps. Here a \emph{packing} in a graph $G$ is a collection of pairwise vertex-disjoint subgraphs of $G$.
	
\paragraph{Finding an almost perfect packing of $G$ by regular blow-ups of graphs in $\cS$.}

	In this step we find a collection of vertex-disjoint $(0,m_2)$-regular blow-ups of graphs in $\cS$ within $G$ that cover all but at most $\eta n$ vertices in $G$. That this is possible follows immediately by Lemma 4.4 from the breakthrough paper of Lang \cite{lang2023tiling}, providing very general conditions for the existence of a perfect packing of a graph $G$ by copies of a hypergraph $F$. The lemma is applicable for a graph $G$ and a family $\cS$ exactly when for every vertex $w \in V(G)$ there are at least $(1 - 1/s + \mu) \binom{n-1}{s-1}$ many $s$-sets of vertices in $G$ that contain $w$ and induce a graph in $\cS$, for some constant $\mu$, and this is implied by $G$ satisfying $\cS$ $1$-robustly.

\paragraph{Finding a perfect packing of $G$ by nearly-regular blow-ups of graphs in $\cS$.}
	Here the idea is to start with an almost perfect packing as in the previous step, and, while there is an uncovered vertex, greedily finding a $(0, m_{1.5})$-nearly-regular blow-up of a graph in $\cS$, in such a way that these new blow-ups are pairwise vertex-disjoint and together they cover only a small proportion of each blow-up from the previous step. The desired perfect packing is obtained by taking the new nearly-regular blow-ups together with the old regular blow-ups but with the vertices from the new blow-ups removed.
	To accomplish this, we use that the Tur\'an number of any $k$-partite $k$-graph is zero, implying that any $n$-vertex $k$-graph with $\Omega(n^k)$ edges contains a $(0,m_3)$-regular blow-up of an edge when $m_3 \ll n$, and a pigeonhole argument, to conclude that every vertex $v$ is contained in many $(0,m_{1.5})$-nearly-regular blow-ups of a graph in $\cS$. To find the desired perfect packing, we use that $G$ contains $\cS$ $1$-robustly, and this is the only property of $\cG_0$ and $\cS_0$ that is used in \cite{illingworth2025spanning} for this step.

\paragraph{Finding a cyclic chain.}
	\def \br {b_{\mathrm{right}}}
	\def \bl {b_{\mathrm{left}}}
	\def \er {e_{\mathrm{right}}}
	\def \el {e_{\mathrm{left}}}
	Denote by $B_1^*, \ldots, B_r^*$ the perfect packing of $G$ by nearly regular blow-ups of graphs in $\cS$, found in the previous step.
	For each $i \in [r]$ we let $\el^i$ and $\er^i$ be two disjoint edges in $B_i$ that avoid the vertex in $B_i$ corresponding to the singleton part (if it exists), and let $\bl^i$ and $\br^i$ be the blow-ups of $\el^i$ and $\er^i$ in $B_i^*$.
	Now, using \cite[Lemma 4.3]{illingworth2025spanning} and that $G$ satisfies $\cS$ $2k$-robustly (corresponding to the collection of copies of graphs in $\cS$ containing one copy of $\er^i$ and one copy of $\el^{i+1}$), we find that there are many $(0,m_1)$-regular blow-ups of graphs in $\cS$ that contain a copy of $\er^i$ and $\el^{i+1}$.
	One can then conclude that we can pick such blow-ups $C_i^*$ so that they are pairwise vertex-disjoint, and together cover only a small proportion of each $B_i^*$. Now, for each $i \in [r]$, remove from $B_i^*$ all vertices participating in some $C_j^*$ except for the vertices of one copy of $\er^i$ contained in $C_i^*$ and one copy of $\el^i$ contained in $C_{i-1}^*$. We thus obtain a cyclic chain $B_1^*, C_1^*, \ldots, B_r^*, C_r^*$.
	Again, all we need to know about $\cG$ and $\cS$ here is that every $G \in \cG$ satisfies $\cS$ $2k$-robustly.
	Also, note that the only difference between getting a chain and getting a cyclic chain is that for the latter we also need to define $C_r^*$, and for the former it suffices to find $C_1^*, \ldots, C_{r-1}^*$.

\section{Improving \cref{thm:main}}\label{sec:improving the bound}
	Recall that our main theorem asserts that every $n$-vertex $k$-graph with minimum supported codegree at least $(1-1/t)n - (k-3)$ has a spanning $\ell$-cycle. This condition is tight when $\ell = \frac{k-1}{2}$ (and $t$ divides $n$ and $n/t+1$ is even) and is off by at most $1$ in general, as can be seen in \Cref{ex:lower bd strong} and \Cref{ex:lower bd weak}, respectively. 
	We now briefly describe how to improve the bound in \cref{thm:Hamilton cycle extremal} and consequently in \cref{thm:main} to $\delta^*(G) \ge (1 - 1/t)n - (k - 2)$ when $t \ge \ell+2$ or when $\ell = k-1$ (i.e.\ we are seeking a Hamilton tight cycle).
	Notice that we always have $t \ge \ell+1$ (see \Cref{obs:useful t info}). This leaves open the question of whether the bound $(1-1/t)n - (k-2)$ is tight or off by $1$ when $t = \ell+1$ and $\ell \neq \frac{k-1}{2}$, or when $\ell = \frac{k-1}{2}$ and $n$ does not satisfy certain divisibility conditions mentioned above.

\paragraph{The case $t \ge \ell+2$.}

	The crucial difference here, as alluded to in \Cref{rem:exact degree}, is that \cref{obs:min 1-degree} now instead becomes $d^1_G(S) \ge n - \floor{n/t} - (i-1)$. 
	This implies that $\partial^2[A']$ has minimum degree at least $x$, not $x+1$ as before, and we can no longer guarantee the existence of $x$ cherries with leaves in $A$. Instead, we then modify \cref{lem:disjoint edges cherries} to instead obtain a linear forest $F$ consisting of edges and cherries, both with leaves in $A$, such that $|F| \in \{3x, 3x+1\}$. 
	We extend a cherry $u_1 u_2 u_3$ of $F$ to $u_1 u_2 v_1 \dots v_{t-2} u_3$ as before. In a similar manner, we also extend any edge $u_1 u_2 \in E(F)$ to a sequence $u v_1 \dots v_{t-2} u_1 u_2$ for $v_i \in B'$ and $u\in A$ (using a strategy similar to \cref{claim:path extension,claim:exceptional path extension}) such that $u v_1 \dots v_{t-2}$ and $v_1 \dots v_{t-2} u_1 u_2$ are both supported. Observe that this sequence does not support an extended $\ell$-path according to \Cref{def:supported-l-paths}, but if we change the definition so as to additionally allow a sequence $v_0 \ldots v_t$ to be considered to support an extended $\ell$-path if $v_1 \ldots v_t$ and $v_0 \ldots v_{t-2}$ are supported, then the rest of the analysis carries through. Indeed, the crucial point is that if $U = u_1 \ldots u_{2t}$ is a sequence of distinct vertices such that $u_1 \ldots u_{2t-2}$ is a tight path and $u_{t+1} \ldots u_{2t}$ is supported, then $U$ supports an $\ell$-path. Indeed, clearly the last $t$ vertices in $U$ form a supported set, so it remains to check that for every $i \in [0, \frac{t}{k-\ell} - 1]$, the subsequence $u_{i(k-\ell)+1} \ldots u_{i(k-\ell)+k}$ is an edge. Notice that if $i \le \frac{t}{k-\ell}-1$ then $i(k-\ell) + k \le t - (k-\ell) + k = t + \ell \le 2t-2$, using $\ell \le t-2$, so $u_{i(k-\ell)+1} \ldots u_{i(k-\ell)+k}$ is contained in $u_1 \ldots u_{2t-2}$, which is a tight path, and thus it is an edge. This allows us to form edges $b_1 \dots b_{t - 1} P \in E(\Gp)$ with $b_i \in B_i$ and $P \in \cP$ corresponding to an ``extended edge of $F$'' of the form $u v_1 \dots v_{t-2} u_1 u_2$, as described above.

\paragraph{The case $\ell = k-1$ (tight cycles).}
	In this case the strategy we have outlined above will work by extending only the cherries in $F$ but not the edges of $F$ (and simply treating these as supported sets of size two in $\cP$). The reason this modification works for tight cycles but fails in general is that it is important that the sequences in the family $\cP$ in \Cref{prop:size balancing final} have length $1 \pmod{k-\ell}$ and indeed this holds for sequences of length $2$ when $\ell = k-1$, but fails otherwise.

\end{document}